\documentclass[11pt]{article}

\usepackage{amsmath,amsfonts,amsthm,amscd,amssymb,graphicx}

\usepackage{subfigure}
\usepackage{xcolor}

\numberwithin{equation}{section}
\usepackage{hyperref}

\newtheorem{theorem}{Theorem}[section]

\newtheorem{lemma}[theorem]{Lemma}

\newtheorem{proposition}[theorem]{Proposition}

\newtheorem{remark}[theorem]{Remark}

\newcommand{\TT}{\mathbb{T}}
\newcommand{\RR}{\mathbb{R}}
\newcommand{\cO}{\mathcal{O}}

\newcommand{\CC}{\mathbb{C}}
\newcommand{\cC}{\mathcal{C}}

\newcommand{\cH}{\mathcal{H}}
\newcommand{\cR}{\mathcal{R}}
\newcommand{\ZZ}{{\mathbb Z}}

\def\cP{\mathcal{P}}

\def\Ff{\widehat{f}}

\def\Frho{\widehat{\rho}}
\def\Fphi{\widehat{\phi}}
\def\FE{\widehat{E}}
\def\Fh{\widehat{h}}

\def\FS{\widehat{S}}

\def\FG{\widehat{G}}

\def\Tf{\widetilde{f}}

\def\Trho{\widetilde{\rho}}
\def\TE{\widetilde{E}}

\def\TG{\widetilde{G}}

\def\Tphi{\widetilde{\phi}}

\def\Ff{\widehat{f}}

\def\Fphi{\widehat{\phi}}
\def\FE{\widehat{E}}
\def\Frho{\widehat{\rho}}

\def\Tf{\widetilde{f}}
\def\TE{\widetilde{E}}
\def\Trho{\widetilde{\rho}}

\def\bj{{\bf j}}

\def\Tphi{\widetilde{\phi}}

\begin{document}

\title{Landau damping and survival threshold} 

\author{Toan T. Nguyen\footnotemark[1]
}

\maketitle

\footnotetext[1]{Penn State University, Department of Mathematics, State College, PA 16802. Email: nguyen@math.psu.edu. }

\maketitle

\begin{abstract}

In this paper, we establish the large time asymptotic behavior of solutions to the linearized Vlasov-Poisson system near general spatially homogenous equilibria $\mu(\frac12|v|^2)$ with connected support on the torus $\mathbb{T}^3_x \times \RR^3_v$ or on the whole space $\RR^3_x \times \RR^3_v$,
including those that are non-monotone. 
The problem can be solved completely mode by mode for each spatial wave number, and their longtime dynamics is intimately tied to the ``survival threshold'' of wave numbers computed by 
$$\kappa_0^2 = 4\pi \int_0^\Upsilon \frac{u^2\mu(\frac12 u^2)}{\Upsilon^2-u^2} \;du$$  
where $\Upsilon$ is the maximal speed of particle velocities. 
It is shown that purely oscillatory electric fields exist and obey a Klein-Gordon's type dispersion relation for wave numbers below { and up to} the threshold, thus rigorously confirming the existence of Langmuir's oscillatory waves { for a non-trivial range of spatial frequencies in this linearized setting}. At the threshold, the phase velocity of these oscillatory waves enters the range of admissible particle velocities, namely there are particles that move at the same propagation speed of the waves. It is this exact resonant interaction between particles and the oscillatory fields that causes the waves to be damped, classically known as Landau damping. Landau's law of decay is explicitly computed and is sensitive to the decaying rate of the background equilibria. The faster it decays at the maximal velocity, the weaker Landau damping is.  Beyond the threshold, the electric fields are a perturbation of those generated by the free transport dynamics and thus decay rapidly fast due to the phase mixing mechanism. 

\end{abstract}

\tableofcontents

%%%%%%%%%%%%%%%%%%%%%%

%%%%%%%%%%%%%%%%%%%%%

\section{Introduction}

%%%%%%%%%%%%%%%%%%%%%

Of great interest in plasma physics is to establish the large-time behavior of charged particles and their possible final states in a non-equilibrium state. For hot plasmas, collisions may be neglected, and meanfield models such as Vlasov kinetic equations to include self-consistent electromagnetic forces are widely used in the literature. Despite being one of the simplest models used in plasma physics, the Vlasov-Poisson system exhibits extremely rich physics, including phase mixing, Landau damping, plasma oscillations, and coherent structures.  

In this paper, we identify precisely the relaxation mechanism and the large time behavior of solutions to the linearized Vlasov-Poisson system near spatially homogeneous steady states of the form $\mu(e)$ with $e = \frac12|v|^2$ {being} the particle energy. The linearized Vlasov-Poisson system then reads
\begin{equation}
\label{linVP}
\left\{ \begin{aligned}
\partial_t f + v\cdot \nabla_x f + v\cdot E\mu'(e) &= 0 
\\
E = -\nabla_x \phi, \qquad - \Delta_x \phi = \rho[f]&= \int_{\RR^3} f(t,x,v)\; dv,
\end{aligned}
\right.
\end{equation}
posed {on the torus $\TT^3_x \times \RR^3_v$} or on the whole space $\RR_x^3 \times \RR_v^3$, with initial data $f(0,x,v) = f_0(x,v)$. {In \eqref{linVP}, we stress that $E(t,x)$ is self-consistently defined through $\rho[f]$.} The linearized problem can be solved completely mode by mode for each spatial frequency, see Section \ref{sec-FT}. Indeed, letting $\lambda \in \CC$ and $k\in \RR^3$ be the temporal and spatial frequencies or the dual Laplace-Fourier variables of $(t,x)$, the Laplace-Fourier transform of the electric field $E(t,x)$, {denoted by $\TE_k(\lambda)$,} can be computed by 
\begin{equation}\label{linElambda} 
\TE_k(\lambda) = \frac1{D(\lambda,k)} \TE_k^{\mathrm{free}}(\lambda)
\end{equation}
for each wave number ${k\in \RR^3\setminus\{0\}}$, where $\TE_k^{\mathrm{free}}(\lambda)$ denotes the Laplace-Fourier transform of the {\em free} electric field $E^{\mathrm{free}}(t,x)$ generated solely by the free transport dynamics\footnote{{ Explicitly, $E^{\mathrm{free}} = -\nabla_x (-\Delta_x)^{-1}\rho^{\mathrm{free}}$, with the {\em free} density $\rho^{\mathrm{free}} = \int f_0(x-vt, v)\; dv$, namely the density generated by the free transport dynamics $\partial_t f^{\mathrm{free}} + v\cdot \nabla_x f^{\mathrm{free}} =0$ with the same initial data $f_0(x,v)$.}}, while $D(\lambda,k)$ is the symbol for the linearized problem \eqref{linVP} in the spacetime frequency space, see Section \ref{sec-FT}. This symbol $D(\lambda,k)$, also known as the dielectric function, is a central function in plasma physics \cite{Trivelpiece}, defined by 
\begin{equation}\label{def-Dintro}
D(\lambda,k) = 1 - \frac{1}{|k|^2} \int_{\RR^3} \frac{ik \cdot v}{ \lambda + ik \cdot v} \mu'(e)\; dv,
\end{equation}
for each wave number ${k\in \RR^3\setminus\{0\}}$ and temporal frequency $\lambda \in \CC$. { Note that the total mass $\int \rho(t,x) \; dx = \iint f(t)\; dxdv$ is conserved in time, and so density $\Frho_k(t)$ is constant at zero frequency $k=0$ (hence, $\FE_k(t) =0$ at $k=0$). For this reason, we focus on the dispersion relation at $k\not =0$. However, throughout the paper, the analysis and the behavior of $D(\lambda,k)$ near $k=0$ play a crucial role.}

In addition, observe that for each $k\not =0$, the zeros $\lambda(k)$ of the dielectric function $D(\lambda,k)=0$ are the ``eigenvalues'' of the linearized Vlasov-Poisson system \eqref{linVP}. {Namely, for each fixed $k\not =0$, if $D(\lambda(k),k) =0$, it follows that the function 
\begin{equation}\label{def-mode}
f(t,x,v) =e^{\lambda(k)t + ik\cdot x} \frac{  ik \cdot v}{\lambda(k) +  ik \cdot  v} \mu'(e),
\end{equation}
together with the electric field $E(t,x) =-ik e^{\lambda(k)t + ik\cdot x}$, solves to the linearized Vlasov-Poisson system \eqref{linVP} (i.e. a mode solution of \eqref{linVP}) for $t\ge 0$. For this reason, $\lambda(k)$ is often referred to as dispersion relation of the linearized electric field, and the behavior of $e^{\lambda(k)t}$ gives the leading dynamics of the linearized electric field. Naturally, in the whole space, a physically relevant solution to the linearized problem consists of a superposition (or in form of a wave packet) of the above mode solutions. We also note that in the case of the torus $\TT^3_x \times \RR^3_v$, the spatial frequencies $k$ are discrete, namely $k \in \ZZ^3$. Throughout the paper, most of the analysis focus on the whole space case with $k\in \RR^3$, leaving the torus case to be treated in a few remarks, see Section \ref{sec-torus}.}

\subsection{Penrose's stable regime}\label{sec-disp}

In view of \eqref{def-Dintro} and the representation \eqref{linElambda}, three regimes follow: 

\begin{itemize}

\item $| k | \gg 1$: {\em free transport regime. }
In this case the electric field is {formally} negligible with respect to the transport part, since $D(\lambda,k) \to 1$, {as $k\to \infty$}. As a consequence, the linearized electric field is a perturbation of that generated by the free transport dynamics, which decays rapidly fast to $0$, with a speed 
proportional to $k$, exponentially if data are analytic and polynomially if data are Sobolev. This exponential damping is at the heart of Mouhot-Villani's celebrated proof of the asymptotic behavior of solutions to the nonlinear Vlasov-Poisson system in the periodic case, see \cite{MV, BMM, GNR1}. 

\item $| k | \sim 1$: {\em Penrose's stable regime.} In this regime, the electric field and the free transport are of the same
magnitude, and the plasma may or may not be stable, depending on the background profile $\mu(\cdot)$. 
It is spectrally stable if and only if $D(\lambda,k)$ never vanishes on $\Re \lambda >0$, which holds for a large class of positive radial equilibria \cite{Penrose, MV}, see also Section \ref{sec-spectral}. Under a stronger, quantitative Penrose stability condition: 
namely, { there is a positive constant $\theta_0$ so that }
\begin{equation}\label{sPenrose}
\inf_{k} \inf_{ \Re \lambda \ge 0} |D(\lambda,k)|  \ge \theta_0 >0,
\end{equation}
{ in which the infimum is taken over $\RR^3\setminus\{0\}$ in the whole space case and over $\ZZ^3\setminus\{0\}$ for the torus case,} the dynamics can again be approximated by that of the free transport and therefore the main damping mechanism in this regime is again {\em phase mixing}. Indeed, this was justified also for the nonlinear problem with analytic or Gevrey data on the torus, see \cite{MV, BMM, GNR1}. See also \cite{BMM1,HKNF2} for the screened Vlasov-Poisson system on the whole space, for which \eqref{sPenrose} holds for $k\in \RR^3$, and the free transport dynamics remains dominant. Specifically, we establish in \cite{GNR1, HKNF2} that the linearized electric field $\FE_k(t)$ in this Penrose's stable regime can be written as 
\begin{equation}\label{linEkt} 
\FE_k(t) = \FE_k^{\mathrm{free}}(t) +\FG_k\star_t  \FE_k^{\mathrm{free}}(t)
\end{equation}
for each wave number $k$, where $\star_t$ denotes the convolution in time, $\FE_k^{\mathrm{free}}(t)$ is again the free transport electric field and $\FG_k(t)$ is exponentially localized $|\FG_k(t)|\lesssim e^{-\langle kt\rangle}$, leading to a much simplified proof of the nonlinear Landau damping  \cite{GNR1} and a construction of echoes solutions for a large class of Sobolev data \cite{GNR2}. We mention that such a representation of the electric field was also established for the weakly collisional regime \cite{CLN}.

{\item $| k | \ll 1$: {\em Landau damping's regime.} It remains to comment on the low frequency regime, which we shall discuss in the next section. }

\end{itemize}

\subsection{Landau damping}

We next focus on the regime where $|k|\ll1$. It turns out that in this regime, the strong Penrose stability condition \eqref{sPenrose} never holds for {\em any non-trivial equilibria!} {Namely, there are dispersion relations $\lambda(k)$ so that $D(\lambda(k),k) =0$, but $\Re \lambda(k)\to 0$ in the limit of $k\to 0$, yielding the failure of the uniform lower bound in \eqref{sPenrose}. Describing the behavior of $\lambda(k)$ in the limit of $k\to 0$ is an important subject in plasma physics, which we shall now discuss.} Indeed, it is classical in the physical literature that at the very low frequency, plasmas oscillate and disperse with a Schr\"odinger type dispersion relation 
\begin{equation}\label{Langmuir0}\Im \lambda_\pm(k) = \pm \Big( \tau_0+ \frac{\tau_1^2}{2\tau^3_0}|k|^2 + \cO(|k|^4)\Big) \end{equation}
for $|k|\ll1$, where $\tau_j^2 =2\pi \int_{\RR} u^{2j+2} \mu(\frac{1}{2}u^2) \;du$, $j=0,1$. These oscillations are classically known as Langmuir's waves in plasma physics \cite{Trivelpiece}. Naturally, the central question is that whether such oscillations are damped. Landau in his 1946 seminal paper \cite{Landau} addressed this very issue, and managed 
to compute the dispersion relation $\lambda = \lambda_\pm(k)$ (i.e. solutions of $D(\lambda_\pm(k),k) = 0$) for Gaussians $\mu = e^{-\frac12 |v|^2}$ and later extended for any positive radial equilibria \cite{Trivelpiece}, yielding\footnote{{Here, we state the results for radial equilibria $\mu(\frac12 |v|^2)$ in three dimension. A similar formula can be derived in other dimensions, see Remark \ref{rem-monotone}.}} 
\begin{equation}
\label{Landau-law}
\Re \lambda_\pm(k) = -  \frac{\pi ^2}{\tau_0}  [u^3\mu(\frac12 u^2)]_{u = \nu_*(k)} 
 (1 + \cO(|k|)),
\end{equation}
for $|k|\ll1$, where $\nu_*(k) = \frac{\tau_0}{|k|}$. Note in particular that \eqref{sPenrose} fails as $|k|\to 0$, since $\Re \lambda_\pm(k) \to 0$ rapidly fast. {Precisely, it follows from \eqref{Landau-law} that the damping rate is polynomially small {of order $\Re\lambda_\pm(k) \sim -c_0|k|^{2N_0-3}$ in the limit of $k\to 0$} for power-law equilibria $\mu(e) \sim \langle e\rangle^{-N_0}$, and super exponentially small of order $\Re\lambda_\pm(k) \sim -c_0|k|^{-3}e^{-c_1/|k|^2}$ in the limit of $k\to 0$ for Gaussian equilibria $\mu = e^{-|v|^2}$ for some positive constants $c_0,c_1$. Note that the faster $\mu(v)$ decays, the weaker this damping rate is.} 

As a consequence of dispersion relations \eqref{Langmuir0}-\eqref{Landau-law}, the electric field is {\it not} exponentially 
decreasing (i.e. $\Re \lambda(k)\to 0$) for each mode at the very low frequency regime $|k|\ll1$, but {\it oscillatory} like a Schr\"odinger type equation.

\begin{figure}[t]
\centering
\includegraphics[scale=.6]{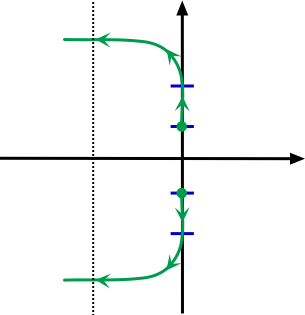}
\put(-82,84){$0$}
\put(-187,80){$\{\Re\lambda \lesssim -|k|\}$}
\put(-20,180){$\mathbb{C}$}
\put(-60,110){$i\tau_0$}
\put(-60,135){$i\sqrt{\tau_0^2 + \kappa_1^2}$}
\put(-120,175){$\lambda_+(k)$}
\put(-120,5){$\lambda_-(k)$}
\put(-65,70){$-i\tau_0$}
\put(-65,46){$-i\sqrt{\tau_0^2 + \kappa_1^2}$}
\caption{\em Depicted are the solutions $\lambda_\pm(k)$ to the dispersion relation that start from $\lambda_\pm(0)=\pm i \tau_0$ {at $k=0$}, remain on the imaginary axis and obey a Klein-Gordon's dispersion relation $\tau_*(|k|)\sim \sqrt{\tau_0^2 + |k|^2}$ for all $0\le |k|\le \kappa_0$ {up until the survival threshold $|k|=\kappa_0$ at which $\lambda_\pm(\kappa_0) = \pm i\sqrt{\tau_0^2 + \kappa_1^2}$}, and then depart from the imaginary axis {as soon as $|k|>\kappa_0$} due to Landau damping towards the phase mixing regime $\{ \Re \lambda \lesssim -|k|\}$. The group velocity $\tau_*'(k)$ is strictly increasing, while the phase velocity $\nu_*(k) = \tau_*(|k|)/|k|$ is strictly decreasing in $|k|$, with $\nu_*(0) = \infty$ and $\nu_*(\kappa_0) = \Upsilon$. 
}
\label{fig-LDthreshold}
\end{figure}

\subsection{Survival threshold}\label{sec-threshold}

As a matter of facts, the three regimes described in the previous sections apply precisely to the case when equilibria are {\em positive} for all $v\in \RR^3$. For compactly supported equilibria, we shall show that there is a finite critical wave number $\kappa_0$, {which is strictly positive and may not be small}, below which the Penrose stability condition \eqref{sPenrose} fails. See Figure \ref{fig-LDthreshold} for an illustration of the threshold. To precise this threshold, we set 
\begin{equation}\label{def-Upsilon} \Upsilon: = \sup \Big\{ |v|, \qquad \mu(\frac12|v|^2) \not =0\Big\}
\end{equation}
to be the maximal speed of particle velocities, which can be finite or infinite (e.g., compactly supported equilibria or Gaussian equilibria). We then introduce the survival wave number threshold $\kappa_0$ defined by 
 \begin{equation}\label{def-introkappa0}
\kappa_0^2 = 4\pi \int_0^\Upsilon \frac{u^2\mu(\frac12 u^2)}{\Upsilon^2-u^2} \;du.
\end{equation}
As equilibria $\mu(\cdot)$ are non negative and decay sufficiently fast as $u\to \Upsilon$, $\kappa_0^2$ is well-defined and finite. Note that $\kappa_0=0$ if $\Upsilon=\infty$ (e.g., when $\mu(\cdot)$ is a Gaussian or real analytic), while $\kappa_0$ is strictly positive and may be large if $\Upsilon<\infty$ (precisely, when $\mu(\cdot)$ is compactly supported). Throughout the paper, we consider equilibria with connected support, namely $\mu(\frac12|v|^2)>0$, whenever $|v|<\Upsilon$. 

Our main results, Theorem \ref{theo-mainLandau} below, assert that   

\begin{itemize}

\item {\em Plasma oscillations:} for $0\le |k|\le \kappa_0$, there are exactly two pure imaginary solutions $\lambda_\pm(k) = \pm i \tau_*(|k|)$ of the dispersion relation $D(\lambda,k)=0$, which obey a Klein-Gordon type dispersion relation 
$$\tau_*(|k|) \approx \sqrt{1 + |k|^2}$$ 
(and in particular, $\tau_*(|k|)$ is a strict convex function in $|k|$). These oscillatory modes experience no Landau damping $\Re \lambda_\pm(k)=0$, but disperse in space, since the group velocity $\tau_*'(k)$ is strictly increasing in $|k|$. This dispersion leads to a $t^{-3/2}$ decay of the electric field in the physical space. These oscillations are known as Langmuir's waves in plasma physics \cite{Trivelpiece}. 

In addition, the phase velocity of these oscillatory waves $\nu_*(k) = \tau_*(|k|)/|k|$ is a decreasing function in $|k|$ with $\nu_*(0) = \infty$ and $\nu_*(\kappa_0) = \Upsilon$ (the maximal speed of particle velocities). 

\item {\em Landau damping:} as $|k|$ increases past the critical wave number $\kappa_0$, the phase velocity of Langmuir's oscillatory waves enters the range of admissible particle velocities, namely $\nu_*(k) <\Upsilon$. That is, {recalling $\Upsilon$ denotes the maximal particle speed, see \eqref{def-Upsilon},
there are particles with velocity $v$ that resonate with the waves, that is $|v| = \nu_*(k)$.} This resonant interaction causes the dispersion functions $\lambda_\pm(k)$ to leave the imaginary axis, and thus the purely oscillatory modes get damped. Landau  \cite{Landau} computed this law of damping for Gaussians (and hence, $\kappa_0=0$) as reported in \eqref{Landau-law}. For the case of compactly supported equilibria $\Upsilon <\infty$, we have $\kappa_0>0$, and the Landau's law of decay can be explicitly computed, giving  
\begin{equation}\label{Landau-law2}
\Re \lambda_\pm(k) =-\frac{2\pi^2}{\kappa_1^2} [u\mu(\frac{1}{2}u^2)]_{u = \nu_*(k)} (1 + \cO(|k| - \kappa_0)),
\end{equation}
as $|k|\to \kappa_0^+$, where $\nu_*(k) = \Upsilon - \frac{2\kappa_0}{\kappa_1^2} (|k| - \kappa_0) $ for some positive constant $\kappa_1$. See Theorem \ref{theo-mainLandau} for the details. That is, the vanishing rate of equilibria at the maximal velocity dictates the Landau damping rate of the oscillations at the critical wave number. This leads to a transfer of energy from the potential energy (i.e. the $L^2$ norm of $E$) to the
kinetic energy of the system, {recalling the total energy of the linearized Vlasov-Poisson system \eqref{linVP}
\begin{equation}\label{energyH}
\mathcal{E}[f] = \frac12\iint_{\RR^3\times \RR^3} \frac{|f|^2}{|\mu'(e)|}\; dxdv + \frac12 \int_{\RR^3} |E|^2\; dx\end{equation}
is conserved in time. Observe that in view of mode solutions of the form $E=-ik e^{\lambda_\pm(k)t + ik\cdot x}$, only the real part of dispersion relation $\lambda_\pm(k)$ contributes to this exchange of energy.} This transfer of energy at the resonant velocity defines the classical notion of Landau damping. In the other words, Landau damping occurs due to the resonant interaction between particles and the oscillatory waves. 

\item {\em Penrose's stable regime:} for $|k| > \kappa_0$, the strong Penrose stability condition \eqref{sPenrose} holds, and therefore the behavior of the electric field is governed by the free transport dynamics as discussed in the previous section.

\end{itemize}

%\subsection{Landau damping}

Note that the Landau damping mechanism is much weaker than the dispersion of oscillations at the critical wave number $|k|\sim \kappa_0$. 
The faster the profile $\mu(e)$ vanishes at the maximal velocity $\Upsilon$, the weaker Landau damping is. In particular, it is super exponentially small for Gaussians.
The main mechanism is therefore the dispersion of the electric field,  which is seen on the {\it imaginary} part
of Landau's dispersion relation \eqref{Langmuir0}, 
whereas the Landau damping rate is seen on its {\it real} part of  \eqref{Landau-law}. {The main results of this work are to capture the survival threshold that characterizes the dynamics of the linearized Vlasov-Poisson system \eqref{linVP} into three regimes as described above, and furthermore, to provide quantitative decay estimates on the Green function and solution operators. We shall review related works on Landau damping in Section \ref{sec-related}. }

\subsection{Equilibria}\label{sec-mu} 

To state the main results of this paper, let us precise the assumptions on equilibria $\mu(\frac12|v|^2)$ that we consider. Set $\Upsilon>0$ to be the maximal particle velocity as in \eqref{def-Upsilon}, and let $K_0,N_0$ be fixed constants with $K_0, N_0\ge 4$. Throughout the paper, we assume that

\begin{itemize}

\item $\mu(\frac12|v|^2) >0$ for all $|v|< \Upsilon$. 

\item $\mu(\frac12|v|^2)$ are in $C^{K_0}$. 

\item In the case when $\Upsilon = \infty$, 
\begin{equation}\label{decay-mu}
|\partial_e^{\alpha}\mu(e)| 
\lesssim \langle e\rangle^{-N_0}, \qquad \forall e\ge 0,\quad \forall |\alpha|\le K_0. 
 \end{equation}

\item In the case when $\Upsilon < \infty$, the limit 
\begin{equation}\label{vanish-mu}
\lim_{|v|\to \Upsilon} \frac{\partial_e^\alpha\mu(\frac12|v|^2)}{(\Upsilon-|v|)^{N_0}}  
 \end{equation}
exists and is positive for $|\alpha|\le K_0$. 
\end{itemize}

{We note in particular that the background density $\int \mu(\frac12 |v|^2)\; dv$ is finite and is assumed to be equal to that of background ions in a nonlinear setting, leading to the linearized problem as stated in \eqref{linVP}.}
Note also that both the regularity and decaying rate assumptions on $\mu(e)$ are not optimal, and we made no attempts in optimizing them in this paper, despite the fact that they play a crucial role in deriving the decay estimates and in calculating the Landau damping rate. Apparently, there are many equilibria that satisfy the above assumptions, including Gaussians, any real analytic equilibria, or any radial functions in $v$ that are positive in the interior of its support $\{|v|<\Upsilon\}$. The positivity is used to exclude possible embedded eigenvalues in the essential spectrum of the free transport operator subject to the support of $\mu(v)$, see Section \ref{sec-emlambda}. {We are however not aware of any examples where embedded eigenvalues exist, when support of $\mu(v)$ is not connected.} Furthermore, we stress that no monotonicity was made on the equilibria (e.g., see Remark \ref{rem-monotone}). The vanishing rate of $\mu(\frac12|v|^2)$ as $|v|\to \Upsilon$ dictates the Landau damping rate at the critical wave number as discussed above, see also Theorem \ref{theo-mainLandau}.

\subsection{Main results}

We are now ready to state our first main result of this paper, {focusing on the whole space case.} 

\begin{theorem}\label{theo-mainLandau} Fix an $N_0\ge 4$, and let $\mu(\frac12|v|^2)$ be a non-negative equilibrium as described in Section \ref{sec-mu}, $\Upsilon$ be the maximal speed of particle velocities defined as in \eqref{def-Upsilon}, and set 
 \begin{equation}\label{def-taukappa0}
\begin{aligned}\tau_j^2 &= 2\pi \int_{\{|u|<\Upsilon\}} u^{2j+2}\mu(\frac12 u^2)\;du, 
\\ \kappa_j^2 &=2\pi \int_{\{|u|<\Upsilon\}}  \frac{u\mu(\frac{1}{2}u^2)}{(\Upsilon-u)^{j+1}} \;du ,
\end{aligned}\end{equation}
for $j\ge 0$. Then, the spacetime symbol $D(\lambda,k)$ defined as in \eqref{def-Dintro} of the linearized Vlasov-Poisson system \eqref{linVP} {in the whole space $\RR_x^3\times \RR_v^3$} satisfies the following:

\begin{itemize} 

\item For each $k\in \RR^3$, $D(\lambda,k)$ is analytic and nonzero in $\Re \lambda>0$.

\item For $0\le |k|\le \kappa_0$, $D(\lambda,k)$ has exactly two pure imaginary solutions $\lambda_\pm(k) = \pm i \tau_*(k)$, where $\tau_*(k)$ is {$C^{N_0-2}$ regular}, strictly increasing in $|k|$, with $\tau_*(0) = \tau_0$ and $\tau_*(\kappa_0) = \sqrt{\tau_0^2 + \kappa_1^2}$. In particular, 
$$ \tau_0\le \tau_*(k) \le \sqrt{\tau_0^2 + \kappa_1^2},$$ 
and there are positive constants $c_0, C_0$ so that 
\begin{equation}\label{intr0-KGtaustar} 
 c_0 |k|\le \tau'_*(k) \le C_0|k|, \qquad  c_0 \le \tau''_*(k) \le C_0,\qquad \forall~ 0\le |k|\le \kappa_0.
\end{equation}
In addition, the phase velocity $\nu_*(k) = \tau_*(k)/|k|$ is strictly decreasing in $|k|$, with $\nu_*(0) = \infty$ and $\nu_*(\kappa_0) = \Upsilon$.

\item There is a $\delta_0>0$ so that the two {unique} solutions $\lambda_\pm(k)$ of $ D(\lambda,k) = 0$ can be extended $C^{N_0-2}$ smoothly for $\kappa_0\le |k|\le  \kappa_0+\delta_0$, with $\lambda_\pm(\kappa_0) = \pm i \tau_*(\kappa_0)$, and satisfy the following Landau's law of damping:

\begin{itemize}

\item[(i)] If $\Upsilon = \infty$, then for $|k|\le \delta_0$, we have 
% \begin{equation}\label{Landau-intro}
%\begin{aligned}
%\Re \lambda_\pm(k) &= -  \frac{\pi ^2\tau_0^2}{|k|^3} \mu\Big(\frac12\frac{\tau_0^2}{|k|^2}\Big) (1 + \cO(|k|)).
%%\\
%%\Im \lambda_\pm(k) &=  \pm\Big( \tau_0 + \frac{\tau_1^2}{2\tau_0^3}|k|^2 + \cO(|k|^4) \Big).
%\end{aligned}\end{equation}
 \begin{equation}\label{Landau-intro}
\Re \lambda_\pm(k) = -  \frac{\pi ^2}{\tau_0}  [u^3\mu(\frac12 u^2)]_{u = \nu_*(k)} 
 (1 + \cO(|k|)),
\end{equation}
where the phase velocity $\nu_*(k) = \frac{\tau_0}{|k|}$.

\item[(ii)] If $\Upsilon <\infty$, then for $\kappa_0\le  |k| \le \kappa_0+\delta_0$, we have  

\begin{equation}\label{Landau-cpt-intro}
\Re \lambda_\pm(k) =-\frac{2\pi^2}{\kappa_1^2} [u\mu(\frac{1}{2}u^2)]_{u = \nu_*(k)} (1 + \cO(|k| - \kappa_0)),
\end{equation}
where  the phase velocity $\nu_*(k) = \Upsilon - \frac{2\kappa_0}{\kappa_1^2} (|k| - \kappa_0) .$ %+ \cO((|k| - \kappa_0)^2)$.

%\begin{equation}\label{Landau-cpt-intro}
%\Re \lambda_\pm(k) = -  \frac{2\pi ^2\kappa_0 \Upsilon}{\kappa_1^2} \mu(e_\Upsilon) (1 + \cO(|k| - \kappa_0)),
%\end{equation}
%where $e_\Upsilon= \frac12 [\Upsilon - \frac{2\kappa_0}{\kappa_1^2} (|k| - \kappa_0)]^2$ and $\kappa_1^2 = 4\pi \Upsilon\int_{\{|u|<\Upsilon\}} \frac{u^2\mu(\frac12 u^2)}{(\Upsilon^2-u^2)^2} \;du.$

\end{itemize}

\item For any $\delta>0$, there is a $c_\delta>0$ so that the strong Penrose stability condition holds
\begin{equation}\label{sPenrose-kappa0}
\inf_{|k|\ge \kappa_0+\delta} ~\inf_{ \Re \lambda \ge 0} |D(\lambda,k)|  \ge c_\delta >0.
\end{equation}

\end{itemize}

\end{theorem}

% \begin{equation}\label{Landau}
%\Re \lambda_\pm(k) = -  \frac{\pi ^2\tau_0^2}{|k|^3} \mu\Big(\frac12\frac{\tau_0^2}{|k|^2}\Big) (1 + \cO(|k|)),
%\end{equation}
%

Theorem \ref{theo-mainLandau} confirms the physical discussions given in the previous sections, see Section \ref{sec-threshold}. In particular, Langmuir's plasma oscillations survive Landau damping for all the wave numbers less than the survival threshold $\kappa_0$, while Landau's law of damping is present and explicitly computed at $\kappa_0$, see \eqref{Landau-intro}-\eqref{Landau-cpt-intro}. Beyond $\kappa_0$, the strong Penrose stability condition is ensured, and the free transport dynamics is a good approximation for the large time behavior of solutions to the linearized Vlasov-Poisson problem.  In particular, we note that oscillations obey a Klein-Gordon's dispersion relation: namely $\tau_*(k) \sim \sqrt{1+|k|^2}$, see \eqref{intr0-KGtaustar}. In particular, oscillations follow the dispersion of a Schr\"odinger's type at the very low frequency  as stated in \eqref{Langmuir0}. Finally, we stress that the smooth extension of $D(\lambda,k)$ at the survival threshold may not be analytic in $\lambda$, since $\mu(v)$ may not be analytic in $v$. 

Our next main result provides quantitative decay estimates on the electric field of the linearized Vlasov-Poisson system \eqref{linVP}. Precisely, we prove the following. 

\begin{theorem}\label{theo-main} Let $\mu(\frac12|v|^2)$ be a non-negative equilibrium as described in Section \ref{sec-mu}, and let $E$ be the electric field of the linearized Vlasov-Poisson system \eqref{linVP} { in the whole space $\RR^3_x\times \RR^3_v$}. Suppose that the initial data $f_0$ has zero average $\iint f_0(x,v)\; dxdv =0$, and satisfies 
\begin{equation}\label{data-assumptions}
\sup_{x,v} \langle x \rangle^5 \langle v\rangle^5 |\partial_{x}^\alpha \partial_v^\beta f_0(x,v)| \lesssim 1
\end{equation}
for $|\alpha|+|\beta|\le 2$. Then, for all $t\ge 0$, we can write 
\begin{equation}\label{rep-Eosc}
\begin{aligned}
E &= \sum_\pm E^{osc}_\pm(t,x) + E^r(t,x),
\end{aligned}
\end{equation}
where 
\begin{equation}\label{osc-decay}
\begin{aligned}
\| E^{osc}_\pm(t)\|_{L^p_x} & \lesssim \langle t\rangle^{-3(1/2-1/p)} , \qquad p\in [2,\infty),
\\
\| E^r(t)\|_{L^p_x} &\lesssim  \langle t\rangle^{-3+3/p} , \qquad p\in [1,\infty]. 
\end{aligned}\end{equation}
In addition, $E^r(t)$ behaves like the electric field generated by the free transport in the sense that $t\partial_x$-derivatives of $E^r(t)$ satisfy the same estimates as those for $E^r(t)$. 
\end{theorem}

Theorem \ref{theo-main} asserts that the leading dynamics of the electric field is indeed oscillatory and dispersive, which behaves like a Klein-Gordon wave of the form $e^{\pm i t\sqrt{1 - \Delta_x} }$, while the remainder decays faster, whose derivatives gain extra decay, extending the previous works \cite{BMM-lin,HKNF3} for analytic equilibria. {We stress that derivatives of the Klein-Gordon components $E^{osc}_\pm(t,x)$ do not gain any extra decay, which is one of the main obstruction for the nonlinear problem.} In fact, in view of Theorem \ref{theo-mainLandau}, we are able to describe precisely the oscillatory electric field $E^{osc}_\pm(t,x)$, namely 
\begin{equation}\label{true-Eosc}
\begin{aligned}
E^{osc}_\pm(t,x) &= G_\pm^{osc}(t)  \star_{t,x} S(t,x) 
\end{aligned}
\end{equation}
where $G_\pm^{osc}(t,x)$ is the oscillatory Green function whose Fourier transform is equal to $e^{\lambda_\pm(k)t} a_\pm(k)$, for some compactly supported and smooth Fourier symbols $a_\pm(k)$ and for dispersion functions $\lambda_\pm(k)$ constructed as in Theorem \ref{theo-mainLandau}. The source $S(t,x)$ is the density generated by the free transport dynamics, namely $S(t,x) =\int f_0(x - v t, v) \;dv.$ See Proposition \ref{prop-decompE} for the precise description of the oscillatory component. The stated dispersive estimates follow from those for the Klein-Gordon's type dispersion $e^{\lambda_\pm(k)t} a_\pm(k)$. We stress that unlike \cite{BMM-lin,HKNF3}, the equilibria $\mu(v)$ may not be analytic and therefore, the resolvent kernel of the linearized Vlasov-Poisson system does not have an analytic extension to the stable half plane in $\Re \lambda <0$. As a consequence, isolating the poles to compute the residue of the resolvent kernels (i.e. the oscillatory component) and deriving decay estimates for the remainder turn out to be rather delicate, see Section \ref{sec-Green} for the details. 

%The additional assumption on initial data $\Ff_{0,k}(v)$ that vanishes in a neighborhood of $|k| = \kappa_0$ is purely technical, avoiding the fact that the spacetime symbol $D(\lambda,k)$, see Theorem \ref{theo-mainLandau}, is in general not analytic past the imaginary axis (in particular, near the dispersion relation $\lambda_\pm(\kappa_0) = \pm i \tau_*(\kappa_0)$). We stress that the assumption is not needed in the case when $D(\lambda,k)$ is meromorphically extended past the imaginary axis for $k$ near $\kappa_0$ (e.g., when $\mu(v)$ is real-analytic). 

The fact that the electric field is oscillatory at temporal frequencies $\pm \tau_0$ can also be seen through the following ``plasma oscillation'' equation  
\begin{equation} \label{oscillations1}
\partial_t^2 E + \tau^2_0 E = \nabla_x \cdot \int  v\otimes vf(t,x,v)  dv
\end{equation}
which can be easily derived from the conservation of mass and momentum of the system \eqref{linVP}, {noting with $\tau_0^2 = \int_{\RR^3}\mu(\frac12 |v|^2)\; dv$, see Remark \ref{rem-mass}.}
Namely, up to these oscillations, the electric field is approximately equal to $\nabla_x \cdot \int f v\otimes v dv$, that is a {\it derivative} of the kinetic energy. Surprisingly
this term is {\it local} in space, in strong contrast with $\nabla_x \Delta_x^{-1}$ which is completely {\it global} in space.
It should therefore decays as fast as (if not faster than) the kinetic energy $\int |v|^2 fdv$ or the density $\int f dv$. This reflects in the stated bounds \eqref{osc-decay} on $E^r$: namely, it is of order $t^{-3}$ (in fact, one may obtain a decay of order $t^{-4}$ upon a further integration by parts in time, see Section \ref{sec-bdE}), compared with the decay of order $t^{-2}$ for the electric field near vacuum case \cite{Bardos}. 

\subsection{Main results on the torus case}\label{sec-torus}

{

In this section, we provide a few remarks on the torus case: namely, the linearized Vlasov-Poisson system posed on the torus $\TT^3_L \times \RR^3_v$, where $\TT^3_L = [0,2\pi L]^3$ (namely functions with period $L$ for $L>0$). Then, we obtain the following theorem which is a discrete version of Theorem \ref{theo-mainLandau}. 

\begin{theorem}\label{theo-mainLandauT} Fix a torus $\TT^3_L = [0,2\pi L]^3$ with $L>0$. Let $\mu(\frac12|v|^2)$ be a non-negative equilibrium as described in Section \ref{sec-mu}, $\Upsilon$ be the maximal speed of particle velocities defined as in \eqref{def-Upsilon}, and introduce $\kappa_0$ as in \eqref{def-introkappa0}. Then, the spacetime symbol $D(\lambda,k)$ defined as in \eqref{def-Dintro} of the linearized Vlasov-Poisson system \eqref{linVP} {on the torus $\TT_L^3\times \RR_v^3$} satisfies the following:

\begin{itemize} 

\item For each $k\in L^{-3}\ZZ^3$, $D(\lambda,k)$ is analytic and nonzero in $\Re \lambda>0$.

\item For each $k\in L^{-3}\ZZ^3 \setminus\{0\}$ with $|k|\le \kappa_0$, $D(\lambda,k)$ has exactly two pure imaginary solutions $\lambda_\pm(k) = \pm i \tau_*(k)$, for some non-vanishing $\tau_*(k)$.

\item For any $\delta>0$, there is a $c_\delta>0$ so that the strong Penrose stability condition holds
\begin{equation}\label{sPenrose-kappaT0}
\inf_{|k|\ge \kappa_0+\delta} ~\inf_{ \Re \lambda \ge 0} |D(\lambda,k)|  \ge c_\delta >0.
\end{equation}

\end{itemize}

\end{theorem}

Theorem \ref{theo-mainLandauT} is simply a discrete version of Theorem \ref{theo-mainLandau}, namely $k\in L^{-3}\ZZ^3$. Interestingly, there are time-periodic solutions $e^{\pm it\tau_*(k) + ik\cdot x}$ to the linearized Vlasov-Poisson system \eqref{linVP} near compactly supported equilibria $\mu(\frac12|v|^2)$ on the torus $\TT^3_L \times \RR^3_v$, provided that $L$ is large enough so that $ L^{-3}\ZZ^3 \cap \{|k|\le \kappa_0\} \not = \emptyset$, noting $\kappa_0>0$ since $\mu$ is compactly supported. As a consequence, the linearized electric field $E(t)$ does not decay, but time periodic, for such a mode. Precisely, we obtain the following theorem which is a discrete version of Theorem \ref{theo-main}.

\begin{theorem}\label{theo-mainT} Fix a torus $\TT^3_L = [0,2\pi L]^3$ with $L>0$. Let $\mu(\frac12|v|^2)$ be a non-negative equilibrium as described in Section \ref{sec-mu}, and let $E$ be the electric field of the linearized Vlasov-Poisson system \eqref{linVP} { on the torus $\TT^3_L\times \RR^3_v$}. Suppose that the initial data $f_0$ has zero average $\iint f_0(x,v)\; dxdv =0$, and satisfies 
\begin{equation}\label{data-assumptionsT}
\sup_{x,v} \langle v\rangle^5 |\partial_{x}^\alpha \partial_v^\beta f_0(x,v)| \lesssim 1
\end{equation}
for $|\alpha|+|\beta|\le 2$. Then, for all $t\ge 0$, we can write 
\begin{equation}\label{rep-EoscTorus}
\begin{aligned}
E &= \sum_{k\in \ZZ^3/L^3 , \;|k|\le \kappa_0} e^{\pm it\tau_*(k) + ik\cdot x} ik+ E^r(t,x),
\end{aligned}
\end{equation}
where the dispersion relation $\tau_*(k)$ is constructed as in Theorem \ref{theo-mainLandauT}, while the remainder $E^r(t,x)$ has its Fourier transform $\FE_k^r(t)$ satisfies the following phase mixing estimates
\begin{equation}\label{EPMix}
|\FE_k^r(t)| \le C \langle kt\rangle^{-N} 
\end{equation} 
uniformly in $k\in L^{-3}\ZZ^3 $, for some constant $N$ depending only on the regularity of $\mu(v)$. 
\end{theorem}

In absence of time-periodic modes in \eqref{rep-EoscTorus} (for instance, when $\mu(v)$ is positive and therefore $\kappa_0 =0$), Theorem \ref{theo-mainT} is by now classical and often referred to as linear Landau damping on the torus, see, e.g., \cite{MV, GNR1}. The new contribution of this work is to establish the appearance of time-periodic modes $e^{\pm it\tau_*(k) + ik\cdot x}$ that occur {\em below the survival threshold $\kappa_0$}, where $\tau_*(k)$ solve $D(\pm i\tau_*(k),k) =0$, namely 
$$\frac{1}{|k|^2} \int_{\RR^3} \frac{k \cdot v}{ \tau_*(k) + k \cdot v } \mu'(e)\; dv = 1$$
 for $0<|k|\le \kappa_0$ (and $k$ is discrete in $L^{-3}\ZZ^3$). As will be seen in the proof, $\tau_*(k) / |k| \ge \Upsilon$, where $\Upsilon$ is the maximal particle speed defined as in \eqref{def-Upsilon}, and therefore the above integration in $v$ is well-defined. Physically speaking, time-periodic solutions exist, since there are no particles whose velocity resonate to that of the oscillatory waves whose phase velocity $\omega(k) = {\tau_*(k)}/{|k|}$. It would be interesting to establish time-periodic or quasi-periodic solutions to the corresponding nonlinear problem.   

{For the rest of the paper, we shall focus on the whole space case $\RR^3_x\times \RR^3_v$. As the linearized problem is solved mode by mode with continuous spatial frequencies $k\in \RR^3$, the torus case is treated as a special case where spatial frequencies $k$ are discrete in $\ZZ^3/L^3$. 
}}

\subsection{Related works}\label{sec-related}

{ 

Landau \cite{Landau}, see also \cite{Maslov, Trivelpiece}, established damping or decay of the electric field via mode analysis of the linearized Vlasov-Poisson problem \eqref{linVP}, but do not provide quantitative decay rates of its solutions. As discussed in Section \ref{sec-threshold}, this classical damping mechanism is realized upon computing the {\em real part} of the dispersion relation $\lambda(k)$, see \eqref{Landau-law} and \eqref{Landau-law2}, for each spatial frequency $k\in \RR^3$ (or discrete on $\ZZ^3/L^3$ for the torus case). This Landau damping rate is sensitive to the decay of $\mu(v)$ at the maximal speed: reading off from \eqref{Landau-law} and \eqref{Landau-law2}, the faster $\mu(v)$ decays, the weaker this Landau damping is. 

The first mathematical work that captures this sensitivity of decay of $\mu(v)$ is due to Glassey and Schaeffer \cite{Glassey, Glassey1}, where the authors proved that for the linearized Vlasov-Poisson system near a Maxwellian on the whole line, the electric field cannot in general decay faster than $1/ (\log t)^{13/2}$ in $L^2$ norm, while near polynomially decaying equilibria at rate $\langle v\rangle^{-\alpha}$, $\alpha>1$, it cannot decay faster than $t^{-\frac{1}{2(\alpha-1)}}$. In addition, there is no Landau damping (i.e. no decay of the electric field in $L^2$ norm) near compactly supported equilibria. Theorem \ref{theo-mainLandau} in this work establishes rigorously the damping rate and the sensitivity of decaying equilibria $\mu(v)$, see \eqref{Landau-intro}-\eqref{Landau-cpt-intro}. As a matter of facts, though the formula \eqref{Landau-intro} for analytic equilibria is well-documented in the physical literature \cite{Landau, Trivelpiece}, the rigorous mathematical proof appears missing until now\footnote{after this work was done and released on the arXiv, we learned that another proof of the linear Landau damping for analytic equilibria is also provided independently in \cite{Ionescu}.}. On the other hand, the formula \eqref{Landau-cpt-intro} appears new in the literature, especially the fact that there is a survival threshold of spatial frequencies below which oscillatory modes exist and do not damp. The results are obtained for both the torus and whole space cases.  

Concerning quantitative decay of the linearized electric field, the linear Landau damping and exponential decay on the torus are established and well-understood, see \cite{Degond, MV, GNR1}, for analytic equilibria (and hence, the survival threshold $\kappa_0 =0$). Theorems \ref{theo-mainLandauT} and \ref{theo-mainT} obtained in the previous section generalize the previous works for compactly supported equilibria (hence, the survival threshold $\kappa_0$ is always positive), including in particular oscillatory modes that are not damped. We also mention a related study \cite{Mahir,Mahir1} for the gravitational case. 

Turning to the whole space case, as discussed at length in the previous sections, the presence of small spatial frequencies complicates the spectrum of the linearized problem, including the failure of the Penrose condition \eqref{sPenrose}, the lack of exponential decay (due to no spectral gap), and the existence of purely oscillatory modes, see Figure \ref{fig-LDthreshold}. The linear Landau damping or decay of the electric field now has an additional damping mechanism: namely, the dispersion that comes from the {\em imaginary part} of the dispersion relation $\lambda(k)$, which we shall now review. 

For the screened Vlasov-Poisson system, the Penrose condition \eqref{sPenrose} remains valid, and therefore the {\em real part} of the dispersion relation is bounded away from zero and the exponential decay for each Fourier mode can be obtained as established in \cite{BMM1,HKNF2}. 

In the unscreened case \eqref{linVP}, the Penrose condition \eqref{sPenrose} always fails. The dispersion of the electric field comes from the {\em imaginary part} of the dispersion relation, see \eqref{Langmuir0}, which is of Klein-Gordon's type dispersion and provides decay of order $t^{-3/2}$ in $L^\infty$ norms.  This linear dispersive decay was proven independently in \cite{BMM-lin} for Gaussian equilibria and in 
\cite{HKNF3} for general analytic equilibria, see also \cite{HKNF4}. The main results of the present work, Theorems \ref{theo-mainLandau} and \ref{theo-main} thus extend the previous results for compactly supported equilibria, leading to the existence of Klein-Gordon's oscillatory waves, also known as Langmuir's waves \cite{Trivelpiece}, for a non-trivial range of spatial frequencies. The fact that dispersion remains a Klein-Gordon's type dispersion is highly non-trivial, since the asymptotic expansion \eqref{Langmuir0} for small $k$ is no longer valid for intermediate frequencies $k$ (precisely, for all $k$ up to the survival threshold $\kappa_0>0$). In addition, the Landau damping rate is also provided for both non-compact and compactly supported equilibria. Finally, we stress that the quantitative decay estimates are not only established for solutions to the linearized problem, but also for the corresponding solution operators (i.e. pointwise Green functions and semi-group estimates) which are useful for the nonlinear problem. 
    
}

\subsection{Notation}

{
We use the notation $\widehat{~\cdot~}$ and $\widetilde{~\cdot~}$ to denote the Fourier transform in $\RR^3_x$ and the Laplace-Fourier transform in $\RR_+ \times \RR^3_x$, namely   
\[
\widehat{f}_k = \int_{\mathbb{R}^{3}} e^{-ix \cdot k} f(x) dx, 
\qquad	\widetilde{f}_k(\lambda) = \int_{0}^{\infty} e^{-\lambda t} \widehat{f}_k(t) dt ,
\]
for $\Re \lambda \ge 0$ and $k \in \RR^3$. { Throughout the paper, for sake of convenience, we shall use the spacetime convolution notation
\[
	G \star_{t} f(t) = \int_0^t G(t-s)f(s)\; ds, \qquad 	G \star_{t,x} f(t,x) = \int_0^t\int_{\RR^3} G(t-s,x-y)f(s,y)\; dyds
\]
with time integration taken over $[0,t]$, as we shall only deal with functions that vanish for $t<0$. }
}

\subsection{Acknowledgements} 
This work is part of a collaborative effort to understand the large-time behavior of charged particles and establish their possible final states in a non-equilibrium state, and the author would like to thank his collaborators C. Bardos, D. Han-Kwan, E. Grenier, I. Rodnianski, and F. Rousset for their many invaluable insights and countless discussions on the subject. The research is supported in part by the NSF under grant DMS-2054726.

\section{{Laplace-Fourier} approach}\label{sec-FT}

In this section, we introduce the {Laplace-Fourier} approach to study the linearized problem \eqref{linVP}. The spectral analysis is classical in physics, and has been pioneered by Landau \cite{Landau}. Throughout the section, we denote by $\Ff_k(t,v), \Fphi_k(t)$ the Fourier transform in $x$ of $f(t,x,v), \phi(t,x)$, respectively, while ${\Tf_k(\lambda,v)}, \Tphi_k(\lambda)$ denote their {Laplace-Fourier} transform in $t,x$.   

\subsection{Resolvent equation}

We first derive the resolvent equation for \eqref{linVP}. Precisely, we obtain the following. 

\begin{lemma}\label{lem-resolvent} Let $\phi(t,x)$ be the electric potential of the linearized Vlasov-Poisson system \eqref{linVP}, and $\Tphi_k(\lambda)$ be the {Laplace-Fourier} transform of $\phi(t,x)$. Then, for each $k\in \RR^3\setminus\{0\}$ and $\lambda\in \CC$, there hold
\begin{equation}\label{resolvent}
\begin{aligned}
\Tphi_k(\lambda) &= \frac{1}{D(\lambda,k)} \Tphi^{\mathrm{free}}_k(\lambda),
\end{aligned}
\end{equation}
where 
\begin{equation}\label{def-Dlambda}
\begin{aligned}
D(\lambda,k) &:=  1 - \frac{1}{|k|^2} \int \frac{ ik\cdot v}{\lambda +  ik \cdot  v} \mu'(e)\;dv, \qquad \Tphi^{\mathrm{free}}_k(\lambda):= \frac{1}{|k|^2}\int \frac{\Ff_{0,k}(v)}{\lambda +  ik \cdot  v} dv,
\end{aligned}
\end{equation}
with $e = \frac12|v|^2$. 

\end{lemma}

\begin{remark}
Observe that the function $\Tphi^{\mathrm{free}}_k(\lambda)$ on the right-hand side of \eqref{resolvent} is the {Laplace-Fourier} transform of the electric potential generated by the free transport dynamics $\partial_t f^{\mathrm{free}} + v\cdot \nabla_x f^{\mathrm{free}} =0$. Hence, the equation \eqref{resolvent} asserts that the electric potential for the linearized Vlasov-Poisson system can be solved completely in terms of the potential generated by the free transport dynamics through the resolvent kernel $ \frac{1}{D(\lambda,k)} $. 
\end{remark}

\begin{proof}[Proof of Lemma \ref{lem-resolvent}]
Taking Laplace-Fourier transform of \eqref{linVP} with respect to variables $(t,x)$, respectively, we obtain 
\begin{equation}\label{VP-lambda1}
\begin{aligned}
(\lambda +  ik \cdot  v ) \Tf_k &= \Ff_{0,k} +ik \cdot v \mu'(e)\Tphi_k  
\end{aligned}\end{equation}
which gives 
$$  \Tf_k  = \frac{  ik \cdot v}{\lambda +  ik \cdot  v} \mu'(e)\Tphi_k+ \frac{\Ff_{0,k}}{\lambda +  ik \cdot  v} .$$
Integrating in $v$ { and recalling $\Trho_k (\lambda)= \int \Tf_k(\lambda,v)\; dv$}, we get $$
\begin{aligned} 
\Trho_k (\lambda) &= \Big( \int \frac{ ik\cdot v}{\lambda +  ik \cdot  v} \mu'(e) \;dv\Big)  \Tphi_k  + \int \frac{\Ff_{0,k}}{\lambda +  ik \cdot  v} dv ,
\end{aligned}$$
which gives \eqref{resolvent}, upon recalling the Poisson equation ${\Tphi_k(\lambda) = \frac{1}{|k|^2}\Trho_k(\lambda)}$. 
 \end{proof}

\begin{lemma}\label{lem-Dlambda} Let $\mu(e)$ be the equilibria as described in Section \ref{sec-mu}, and $D(\lambda,k)$ be defined as in \eqref{def-Dlambda}. Then, for each $k\not =0$, we can write 
\begin{equation}\label{def-Dlambda1}
\begin{aligned}
D(\lambda,k) 
= 1 - \frac{1}{|k|^2}\cH(i\lambda/|k|), \qquad \cH(z) 
= 2\pi \int_{\RR} \frac{u \mu(\frac{1}{2}u^2)}{z-u} \;du,
\end{aligned}
\end{equation}
{where $\cH(z)$ is well-defined on $\Im z\ge 0$ and analytic in the upper half plane $\{ \Im z >0\}$.}
In particular, $D(\lambda,k)$ is analytic in $\Re\lambda>0$, and 
\begin{equation}\label{unibd-cH} |\partial_z^n\cH(i\lambda/|k|)| \lesssim 1,
\end{equation}
uniformly for $k\in \RR^3$, $\Re \lambda \ge 0$, and $0\le n<N_0$.  
\end{lemma}

%\begin{equation}
%\label{Landau}
%\tau_*(k) = \tau_0 + \frac{\tau_1^2}{2\tau_0^3}|k|^2 + \mathcal{O}(|k|^4) 
%\end{equation}
%for some positive constant $\tau_1^2$ and for . 

\begin{proof} By definition, we may write 
$$
\begin{aligned}
D(\lambda,k) 
& = 1 + \frac{1}{|k|^2} \int_{\RR^3} \frac{k \cdot v}{ i\lambda  - k \cdot  v} \mu'(e)\; dv .
\end{aligned}
$$
For $k\not =0$, we introduce the change of variables 
\begin{equation}\label{change-vu}
u = \frac{k\cdot v}{|k|} , \qquad w = v -   \frac{(k\cdot v)k}{|k|^2},
\end{equation}
with the Jacobian determinant equal to one. Note that $|v|^2 = u^2 + |w|^2$. This yields 
$$
\begin{aligned}
D(\lambda,k) 
& = 1 + \frac{1}{|k|^2} \int_{\RR} \frac{u}{  i\lambda/|k| - u} \Big(\int_{w\in k^\perp} \mu'(\frac12u^2 + \frac12 |w|^2)\; dw\Big)\; du
\\
& = 1 + \frac{1}{|k|^2} \int_{\RR} \frac{1}{  i\lambda/|k| - u } \frac{d}{du}\Big(\int_{w\in k^\perp} \mu(\frac12u^2 + \frac12 |w|^2)\; dw\Big)\; du.
\end{aligned}
$$
For each $u \in \RR$, set \begin{equation}\label{def-kappau}
\kappa(u) = \int_{w\in k^\perp } \mu(\frac{1}{2}(u^2 + |w|^2))\; dw.
\end{equation}
Now parametrizing the hyperplane $k^\perp$ via the polar coordinates with radius $r = |w|$, and then setting $s = \frac{1}{2}(u^2 + r^2)$, we have 
\begin{equation}\label{3Duse}
\kappa(u) = 2\pi \int_0^\infty \mu(\frac{1}{2}(u^2 + r^2))\; r dr = 2\pi \int_{\frac12 u^2}^\infty \mu(s) \; ds.
\end{equation}
In particular, this gives $\kappa'(u) = -2\pi u \mu(\frac{1}{2}u^2)$, and therefore, 
$$
\begin{aligned}
D(\lambda,k) 
& = 1 - \frac{2\pi}{|k|^2} \int_{\RR} \frac{u \mu(\frac12 u^2)}{ i\lambda/|k|  -u }\; du,
\end{aligned}
$$
which gives \eqref{def-Dlambda1}, upon setting $\cH(z)$ to be the integral term with $z = i\lambda/|k|$. Since $\mu(e)$ decays rapidly in the particle energy, {$\cH(z)$ and its derivatives are well-defined, and therefore analytic in $z$} on the upper half plane $\{ \Im z >0\}$. In addition, we may write 
\begin{equation}\label{redef-cH}
\cH(z) 
= 2\pi \int_{\RR} \frac{u \mu(\frac{1}{2}u^2)}{z-u} \;du
= 2\pi \int_0^\infty e^{iz t} \int_{\RR} e^{-iu t} u \mu(\frac{1}{2}u^2)\;dudt.
\end{equation}
Set
\begin{equation}\label{def-Nt}
N(t) = 2\pi \int_{\RR} e^{-iu t} u \mu(\frac{1}{2}u^2)\;du.
\end{equation} 
Taking integration by parts in $u$ and using the regularity of $\mu(e)$, we obtain 
\begin{equation}\label{bounds-Nt}
|\partial_t^n N(t)|\le C_n\langle t \rangle^{-K_0} ,
\end{equation}
for any $0\le n<N_0$, where $K_0, N_0$ are regularity and decay indexes as described in Section \eqref{sec-mu}. 
This gives 
\begin{equation}\label{bd-cH1}
|\cH(z)| 
\le  C_0 \int_0^\infty e^{- \Im z t}  \langle t \rangle^{-K_0} dt \lesssim 1,
\end{equation}
for any $\Im z\ge 0$. Similar bounds hold for $z$-derivatives, giving \eqref{unibd-cH} as claimed. 
\end{proof}

%\begin{remark}\label{rem-mode}
%For each $k\in \RR^3\setminus\{0\}$, a solution $\lambda = \lambda(k)$ of the dispersion relation $D(\lambda,k)=0$ 
%yields a mode solution to the linearized Vlasov-Poisson problem. Precisely, for each $k\not =0$, if $D(\lambda(k),k) =0$, then the function 
%\begin{equation}\label{def-mode}f(t,x,v) =e^{\lambda(k)t + ik\cdot x} \frac{  ik \cdot v}{\lambda(k) +  ik \cdot  v} \mu'(e)
%\end{equation}
%solves the linearized Vlasov-Poisson system \eqref{linVP} with electric potential $\phi(t,x) =e^{\lambda(k)t + ik\cdot x}$. {A physically relevant solution to the linearized problem consists of a superposition (or in form of a wave packet) of the above mode solutions.}
%\end{remark}

\begin{remark}\label{rem-monotone} Note that the proof of Lemma \ref{lem-Dlambda} makes a crucial use of the three-dimensional space for velocities $v$, precisely the computation of \eqref{3Duse}. A similar calculation works in higher dimensions, but not in dimension one or two. In particular, this yields the monotonicity of marginals $\kappa(u)$, namely $\kappa'(u) = -2\pi u \mu(\frac{1}{2}u^2)$, without any monotonicity assumption on $\mu(\cdot)$. 
\end{remark}

{
\begin{remark}\label{rem-mass}
Note that $\int_{\RR^3} \mu(\frac12 |v|^2) \; dv= \tau_0^2$ with $\tau_0$ defined as in \eqref{def-taukappa0}. Indeed, we compute   
$$
\begin{aligned} \int_{\RR^3} \mu(\frac12 |v|^2) \; dv &= - \frac13 \int_{\RR^3} |v|^2 \mu'(\frac12 |v|^2) \; dv = - \int_{\RR^3} |v_1|^2 \mu'(\frac12 |v|^2) \; dv
\\&= - \int_{\RR} v_1 \frac{d}{dv_1} \Big(\int_{\RR^2}\mu(\frac12(|v_1|^2 + |w|^2)) \; dw\Big)dv_1 = - \int_{\RR} v_1 \kappa'(v_1) \; dv_1
\end{aligned}
$$
for $\kappa(u)$ defined as in \eqref{def-kappau} (noting $k^\perp = \RR^2$ in this case). Using $\kappa'(u) = -2\pi u \mu(\frac{1}{2}u^2)$ as computed above, see \eqref{3Duse}, we obtain $\int_{\RR^3} \mu(\frac12 |v|^2) \; dv= \tau_0^2$ as claimed.   
\end{remark}
}

\subsection{Spectral stability}\label{sec-spectral}

In view of the resolvent equation \eqref{resolvent} and its mode solution \eqref{def-mode}, the zeros $\lambda = \lambda(k)$ to the dispersion relation $D(\lambda,k)=0$ plays a crucial role in studying the large time dynamics of the linearized problem \eqref{linVP}. In this section, we shall prove that there is no solution in the right half plane $\Re \lambda>0$. Namely, we obtain the following. 

\begin{proposition}\label{prop-nogrowth} Let $\mu(\frac12 |v|^2)$ be any non-negative radial equilibria in $\RR^3$. Then, the linearized system \eqref{linVP} has no nontrivial mode solution of the form $e^{\lambda t + ik\cdot x} \Ff_k(v)$ with $\Re \lambda >0$ for any nonzero function $\Ff_k(v)$. 
\end{proposition}

\begin{proof} In view of a mode solution \eqref{def-mode}, it suffices to prove that for any $k\in \RR^3\setminus\{0\}$, $D(\lambda,k) \not =0$ for $\Re \lambda >0$, or equivalently $\cH(z)\not =|k|^2$ for $\Im z >0$. Indeed, in view of \eqref{def-Dlambda1}, we first write
$$
\begin{aligned}
\cH(z) &= 2\pi \int_{\RR} \frac{u \mu(\frac{1}{2}u^2)}{z-u} \;du 
\\&= 2\pi \int_{\RR} \frac{(\Re z-u)u \mu(\frac{1}{2}u^2)}{|z-u|^2} \;du - 2i\pi \Im z \int_{\RR} \frac{u \mu(\frac{1}{2}u^2)}{|z-u|^2} \;du .
\end{aligned}$$
Hence, suppose that $\cH(z)=|k|^2$ for some $z$ with $\Im z >0$. This implies that $\Im \cH(z) =0$, which gives $$2\pi\int_{\RR} \frac{u \mu(\frac{1}{2}u^2)}{|z-u|^2} \;du =0,$$ 
since $\Im z>0$. On the other hand, {using the above identity,} $\cH(z) = |k|^2$ is now reduced to 
$$|k|^2 + 2\pi \int_{\RR} \frac{u^2 \mu(\frac{1}{2}u^2)}{|z-u|^2} \;du =0,$$
which is a contradiction, since $\mu \ge 0$. 
The proposition follows. 
\end{proof}

\subsection{No embedded eigenvalues}\label{sec-emlambda}

In this section, we study the dispersion relation $D(\lambda,k)=0$ on the imaginary axis $\lambda = i\tau$ so that $|\tau| < |k|\Upsilon$, where $\Upsilon$ as in \eqref{def-Upsilon}. Note that this is the region where $\lambda$ is in the interior of the essential spectrum of the free transport dynamics $\partial_t + v\cdot \nabla_x$, for particle velocity $v$ in the support of the equilibria, {namely $\lambda \in \mathrm{Range}(ik\cdot v)$ for $|v|\le \Upsilon$.} {Precisely, we obtain the following.

\begin{proposition}\label{prop-noembedded} Let $\mu(\frac12 |v|^2)$ be any non-negative radial equilibria with connect support in $\RR^3$, and let $\Upsilon$ be the maximal particle speed as in \eqref{def-Upsilon}. Then, the linearized system \eqref{linVP} has no 
nontrivial mode solution of the form $e^{\lambda t + ik\cdot x} \Ff_k(v)$, when $\lambda = i\tilde\tau|k|$ with $|\tilde\tau| < \Upsilon$, for any nonzero function $\Ff_k(v)$. In addition, for any compact subset $U$ in $\{| \tilde\tau|< \Upsilon\}$, there is a positive constant $c_U$ so that 
 \begin{equation}\label{lowbd-Ditau2}
|D(i\tilde\tau|k|,k)| \ge c_U \Big(1 + \frac{1}{|k|^2}\Big), \qquad \forall ~ \tilde\tau\in U,
\end{equation}
uniformly for any $k\not=0$. 

\end{proposition}
}
\begin{proof}
We first recall from \eqref{def-Dlambda1} that for $z = i\tilde\gamma - \tilde\tau$, 
$$\cH(i\tilde\gamma - \tilde\tau) 
= 2\pi \int_{\{|u| < \Upsilon\}} \frac{u \mu(\frac{1}{2}u^2)}{i\tilde\gamma - \tilde\tau-u} \;du.
$$
Therefore, for $|\tilde\tau|<\Upsilon$, using the Plemelj's formula, we obtain 
\begin{equation}\label{Plemelj}
 \lim_{\tilde\gamma \to 0^+}\cH(i\tilde\gamma - \tilde\tau) 
=  -2\pi P.V. \int_{\{|u|<\Upsilon\}} \frac{u \mu(\frac{1}{2}u^2)}{u + \tilde\tau} du  {~+~} 2 i \pi^2 \tilde\tau \mu(\frac{1}{2}\tilde\tau^2) ,
\end{equation}
where $P.V.$ denotes the Cauchy principal value associated to the singularity at $u =- \tilde\tau$. Recalling \eqref{def-Dlambda1}, {for $\lambda = (\tilde\gamma + i\tilde \tau)|k|$ with $\tilde \gamma\to 0$}, we thus obtain 
\begin{equation}\label{form-Ditau2}
\begin{aligned}
D(i\tilde\tau|k|,k) &{
= 1 - \frac{1}{|k|^2}\lim_{\tilde \gamma \to 0^+}\cH(i\tilde \gamma - \tilde \tau)}
\\&= 1 + \frac{2\pi }{|k|^2} P.V. \int_{\{|u|<\Upsilon\}} \frac{u \mu(\frac{1}{2}u^2)}{u+\tilde\tau} du {~-~} \frac{2i \pi^2 }{|k|^2} \tilde\tau \mu(\frac{1}{2}\tilde\tau^2)
.
\end{aligned}
\end{equation}
In particular, this implies that 
$$|D(i\tilde\tau|k|,k)|\ge |\Im D(i\tilde\tau|k|,k)|\gtrsim  \frac{1}{|k|^2} |\tilde\tau| \mu(\frac{1}{2}\tilde\tau^2)$$
for any $|\tilde\tau|<\Upsilon$, recalling that $\mu(\cdot)$ is positive on its support. On the other hand, at $\tilde\tau=0$, we have 
$$
D(0,k) = 1 + \frac{2\pi }{|k|^2} \int_{\{|u|<\Upsilon\}}  \mu(\frac{1}{2}u^2) du, 
$$
which again gives $|D(0,k)| \gtrsim 1 + |k|^{-2}$, since $\mu\ge 0$. This proves that there are no embedded eigenvalues that lie on the imaginary axis so that $|\lambda| < |k|\Upsilon$. In fact, for any compact subset $U$ in $\{| \tilde\tau|< \Upsilon\}$, since $\mu(\cdot)$ is strictly positive on $U$, there is a positive constant $c_U$ so that 
$$
|D(i\tilde\tau|k|,k)| \ge c_U \Big(1 + \frac{1}{|k|^2}\Big), \qquad \forall ~ \tilde\tau\in U,
$$
proving \eqref{lowbd-Ditau2}. 
\end{proof}

\begin{remark} It follows from Proposition \ref{prop-noembedded} that if $\Upsilon = \infty$, there are no pure oscillatory modes on the imaginary axis. When $\Upsilon <\infty$, we can also compute from \eqref{form-Ditau2} that 
$$
\begin{aligned}
D(\pm i|k|\Upsilon ,k) &= 1 + \frac{2\pi }{|k|^2} \int_{\{|u|<\Upsilon\}} \frac{u \mu(\frac{1}{2}u^2)}{u \pm \Upsilon} du 
\\&= 1 -\frac{2\pi }{|k|^2} \int_{\{|u|<\Upsilon\}} \frac{u^2 \mu(\frac{1}{2}u^2)}{\Upsilon^2 - u^2} du,
\end{aligned}$$
which yields 
\begin{equation}\label{lowbd-Ditau3}
D(\pm i|k|\Upsilon ,k) = 1 -\frac{\kappa_0^2}{|k|^2}.
\end{equation}
In particular, if $\Upsilon <\infty$, the lower bound \eqref{lowbd-Ditau2} holds for $U = \{| \tilde\tau|\le \Upsilon\}$ for any $|k|\le \frac12 \kappa_0^2$ (recalling $\kappa_0>0$, when $\Upsilon<\infty$).
\end{remark}

\subsection{Langmuir's oscillatory waves}

In this section, we prove the existence of pure imaginary solutions to the dispersion relation $D(\lambda,k)=0$ for $\lambda = i\tau$, necessarily for $|\tau|\ge |k|\Upsilon$, when $\Upsilon< \infty$, as no such solutions exist for $|\tau|< |k|\Upsilon$, see Section \ref{sec-emlambda}. This in particular confirms the existence of pure oscillatory modes or Langmuir's waves known in the physical literature. Precisely, we obtain the following. 

\begin{theorem}\label{theo-LangmuirE} Fix an $N_0\ge 4$, and let $\mu(\frac12|v|^2)$ be a non-negative equilibrium as described in Section \ref{sec-mu}, $\Upsilon$ be as in \eqref{def-Upsilon}, and {$\tau_j^2, \kappa_j^2$ as in \eqref{def-taukappa0}}. Then, for any $0\le |k| \le \kappa_0$, there are exactly two zeros $\lambda_\pm= \pm i \tau_*(k)$ of the electric dispersion relation $ D(\lambda,k) = 0$
that lie on the imaginary axis $\{\Re \lambda =0\}$, where $\tau_*(k)$ is {$C^{N_0-2}$ regular}, strictly increasing in $|k|$, with $\tau_*(0) = \tau_0$ and $\tau_*(\kappa_0) = \sqrt{\tau_0^2 + \kappa_1^2}$. In particular, 
\begin{equation}\label{range-taus}\tau_0\le \tau_*(k) \le  \sqrt{\tau_0^2 + \kappa_1^2},\end{equation}
and there are positive constants $c_0, C_0$ so that 
\begin{equation}\label{lowerbound-taustarDE} 
 c_0 |k|\le \tau'_*(k) \le C_0|k|, \qquad  c_0 \le \tau''_*(k) \le C_0,
\end{equation}
for $0\le |k|\le \kappa_0$. In addition, the phase velocity $\nu_*(k) = \tau_*(k)/|k|$ is strictly decreasing in $|k|$ with $\nu_*(0) = \infty$ and $\nu_*(\kappa_0) = \Upsilon$. 
\end{theorem}

\begin{remark} 
In the case when $\Upsilon =\infty$, Theorem \ref{theo-LangmuirE} in fact applies, with $\kappa_0 =0$ and $\kappa_1=0$, giving only trivial oscillatory modes $\lambda_\pm(0) = \pm i \tau_0$. 

\end{remark}

\begin{proof} In view of Section \ref{sec-emlambda}, it suffices to consider the case when $\Upsilon<\infty$ and $\lambda = i\tau$, with $|\tau|\ge |k| \Upsilon$. Recall from \eqref{def-Dlambda1} that the dispersion relation $D(i\tau,k)=0$ is equivalent to solving the equation 
\begin{equation}\label{eqs-cHk} 
 |k|^2 = \cH(z) = 2\pi \int_{\RR} \frac{u \mu(\frac{1}{2}u^2)}{z-u} \;du = 2\pi \int_{\{|u|<\Upsilon\}} \frac{u^2\mu(\frac12 u^2)}{z^2- u^2} \;du
\end{equation}
for real value $z = i\lambda/|k| = -\tau/|k|$. Note that since $|\tau|\ge |k|\Upsilon$, we have $|z|\ge \Upsilon$ and so the above integral makes sense, recalling that $\mu(\frac12 u^2)$ decays rapidly to zero as $u \to \pm \Upsilon$. 
Therefore, $\cH(z)$ is a well-defined, radial, and real-valued function on $[\Upsilon, \infty]$, with $\cH(\infty)=0$ and $\cH(\Upsilon) = \kappa_0^2$, 
where $\kappa_0$ is defined as in \eqref{def-taukappa0}. 
In addition, $\cH'(z) <0$ for $z> \Upsilon$, and so $\cH(z)$ is a bijection from $[\Upsilon, \infty]$ to $[0,\kappa_0^2]$. As a result, 
the inverse map $\cH^{-1}(\cdot)$ is well-defined from $[0,\kappa_0^2]$ to $[\Upsilon, \infty]$. {It follows from a direct calculation of derivatives of $\frac{1}{z^2-u^2}$ that the regularity of $\cH(\cdot)$ in $z$ depends on decay properties of $\mu(\frac12 u^2)$ as $|u|\to \Upsilon$ (recalling $|z|\ge \Upsilon >|u|$). Using the decay assumptions on $\mu(\cdot)$ in Section \ref{sec-mu}, it follows that $\cH \in C^{N_0-2}$, and so is $\cH^{-1}(\cdot)$.  } 

The existence of solutions $z_*(k)$ to \eqref{eqs-cHk} now follows straightforwardly. Indeed, for $|k|> \kappa_0$, there are no zeros of \eqref{eqs-cHk}, since $|k|^2$ is not in the range of $\cH(\cdot)$. On the other hand, for $0<|k| \le \kappa_0$, there is a radial function $z_*(k)\in [\Upsilon, \infty)$ in $|k|$ so that $z_*(k) = \cH^{-1}(|k|^2)$. Equivalently, there are zeros $\tau = \pm \tau_*(k)$ of the dispersion relation $D(i\tau,k)=0$, which is of the form 
\begin{equation}\label{form-taustar} 
\tau_*(k) =|k| \cH^{-1}(|k|^2),
\end{equation}
for any $k\not =0$ so that $|k| \le \kappa_0$. Since $\cH^{-1}(\cdot)$ is {$C^{N_0-2}$ smooth}, so is $\tau_*(k)$ in $|k|$. 
In addition, it follows from $|k|^2 = \cH(z_*(k))$ that $z_*'(k) = 2|k| / \cH'(z_*(k))$. Since $\cH'(z) <0$ for $z> \Upsilon$, $ z_*(k)$ is strictly decreasing, with $z_*(0) = \infty$ and $z_*(\kappa_0) = \Upsilon$. This proves that the phase velocity $\nu_*(k) = \tau_*(k)/|k| =z_*(k)$ is strictly decreasing in $|k|$.

It remains to study the dispersive property of $\tau_*(k)$. It turns out convenient to write 
\begin{equation}\label{def-omega}
\cH(z) = \frac{1}{z^2}\omega(\frac{1}{z^2}), \qquad \omega(y) := 2\pi\int_{\{|u|<\Upsilon\}} \frac{u^2\mu(\frac12 u^2)}{1- y u^2} \;du
\end{equation}
for $y \in [0,\Upsilon^{-2}]$. {As $\mu(\frac12u^2)$ decays rapidly in $u$, it is clear that $\omega(y)$ is a $C^{N_0-2}$ regular function, for $N_0$ being the decaying rate of $\mu(\cdot)$ as introduced in Section \ref{sec-mu}.} It then follows from the monotonicity of $\omega(y)$ that $\tau_0^2\le \omega(y)\le \kappa_0^2\Upsilon^2$, since by definition, $\omega(0) = \tau_0^2$ and $\omega(\Upsilon^{-2}) = \kappa_0^2\Upsilon^2$, where $\tau_0^2$ and $\kappa_0^2$ are defined as in \eqref{def-taukappa0}. 
Next, set $x_*(k) = \tau_*(k)^2$. By construction, $x_*(k)$ satisfies 
\begin{equation}\label{eqs-xstarE}
x_* = \omega(|k|^2/x_*) \end{equation}
which gives
\begin{equation}\label{x-bounds} \tau_0^2 \le x_*(k) \le \kappa_0^2\Upsilon^2.\end{equation}
Note that by definition, $ \kappa_0^2\Upsilon^2 = \tau_0^2 + \kappa_1^2$, giving \eqref{range-taus}. {In addition, it follows from the regularity of $\omega(y)$ and the identity \eqref{eqs-xstarE} that $x_*(k)$ is also a $C^{N_0-2}$ function.} Now, taking the derivative of the equation $x_* = \omega(|k|^2/x_*)$ and denoting $' = \frac{d}{d(|k|)}$, we get 
\begin{equation}\label{cal-D2x}
\begin{aligned}
\Big[1+  \frac{|k|^2}{x_*^2}  \omega'(\frac{|k|^2}{x_*})\Big] x_*'(|k|) &=  \frac{ 2|k|}{x_*} \omega'(\frac{|k|^2}{x_*}) ,
\\
\Big[1+  \frac{|k|^2}{x_*^2}  \omega'(\frac{|k|^2}{x_*}) \Big] x_*''(|k|) 
&=  \frac{2(x_* - |k| x'_*)^2}{x_*^3} \omega'(\frac{|k|^2}{x_*})  
+ \frac{|k|^2(2x_*- |k|x'_*)^2}{x_*^4} \omega''(\frac{|k|^2}{x_*})  ,
\end{aligned}
\end{equation}
noting that each term on the right-hand side is nonnegative, since $\omega'(y) \ge0$ and $\omega''(y)\ge 0$ (in fact, by a direct calculation, all derivatives of $\omega(y)$ are nonnegative {for $y \in [0,\Upsilon^{-2}]$}). In addition, since $y \omega'(y) \le \Upsilon^{-2}\omega'(\Upsilon^{-2})$, which is finite, we have 
\begin{equation}\label{bd-coff1}
C_0^{-1} \le 1 + \frac{|k|^2}{x_*^2}  \omega'(\frac{|k|^2}{x_*}) \le C_0,
\end{equation}
upon using \eqref{x-bounds}. This gives
\begin{equation}\label{Dx-bounds}
C_0^{-1} |k| \le x_*'(k) \le C_0 |k| , \qquad C_0^{-1}\le x_*''(k)  \le C_0 , \qquad \forall ~0\le |k|\le \kappa_0,\end{equation}
for some positive constant $C_0$. Next, we prove the convexity of $\tau_*(k)$. By definition, $\tau_*(k) = \sqrt{x_*(k)}$, we compute 
$$ \tau_*'(k) = \frac{x_*'(k)}{2\sqrt{x_*(k)}} , \qquad \tau_*''(k) = \frac{2 x''_*(k) x_*(k) - x_*'(k)^2}{4 x_*(k)^{3/2}}.$$ 
The estimates on $\tau_*'(k)$ {and the upper bound on $\tau''_*(k)$} follow at once from those on $x_*(k), x_*'(k)$, {and $x_*''(k)$}, see \eqref{x-bounds} and \eqref{Dx-bounds}. We shall  now prove that 
\begin{equation}\label{lowbd-tauk}
\tau_*''(k) \ge c_0, \qquad ~\forall~0\le |k|\le \kappa_0,
\end{equation}
for some positive constant $c_0$. In view of \eqref{x-bounds}, it suffices to obtain a lower bound for $2x_*''(k) x_*(k) - x_*'(k)^2 $. Using \eqref{cal-D2x}, we compute 
$$
\begin{aligned} 
\Big[ 1 + \frac{|k|^2}{x_*^2}  \omega'(|k|^2/x_*) \Big]x_*'(k)^2 &= \frac{ 2|k| x'_*}{x_*} \omega'(|k|^2/x_*) 
\\
2\Big[ 1 + \frac{|k|^2}{x_*^2}  \omega'(\frac{|k|^2}{x_*}) \Big]x_*''(k) x_*(k) 
&=  \frac{4(x_* - |k| x'_*)^2}{x_*^2} \omega'(\frac{|k|^2}{x_*})  
\\&\quad + \frac{2|k|^2(2x_*- |k|x'_*)^2}{x_*^3} \omega''(\frac{|k|^2}{x_*}) 
 \end{aligned}$$
and so 
$$
\begin{aligned}
\Big[ &1 + \frac{|k|^2}{x_*^2}  \omega'(\frac{|k|^2}{x_*}) \Big] \Big (2x_*''(k) x_*(k) - x_*'(k)^2 \Big) 
\\& = \frac{2(2x_* - |k| x'_*)}{x_*^2} 
 \Big[ (x_* - 2|k| x'_*) \omega'(\frac{|k|^2}{x_*})  
+ (2x_*- |k|x'_*)\frac{|k|^2}{x_*} \omega''(\frac{|k|^2}{x_*}) \Big] 
\end{aligned}$$
in which recalling \eqref{bd-coff1}, the factor $1 + \frac{|k|^2}{x_*^2}  \omega'(\frac{|k|^2}{x_*}) $ is harmless.  
Using again \eqref{cal-D2x}, we note that 
$$
\begin{aligned}
2 x_* - |k|x'_* = 2 x_* - 2|k|^2 \frac{ \frac{1}{ x_*} \omega'(|k|^2/x_*) }{1 + \frac{|k|^2}{x_*^2}  \omega'(|k|^2/x_*) } =  \frac{ 2 x_* }{1 + \frac{|k|^2}{x_*^2}  \omega'(|k|^2/x_*) } 
\\
x_* - 2 |k|x'_* = x_* - 4|k|^2 \frac{ \frac{1}{ x_*} \omega'(|k|^2/x_*) }{1 + \frac{|k|^2}{x_*^2}  \omega'(|k|^2/x_*) } =  \frac{  x_* - \frac{3|k|^2}{ x_*} \omega'(|k|^2/x_*) }{1 + \frac{|k|^2}{x_*^2}  \omega'(|k|^2/x_*) }
\end{aligned}
$$
which in particular yields that $2 x_* - |k|x'_* \gtrsim 1$ on $[0,\kappa_0]$, recalling  \eqref{x-bounds}. 
Therefore, 
$$
\begin{aligned}
\Big[ &1 + \frac{|k|^2}{x_*^2}  \omega'(\frac{|k|^2}{x_*}) \Big] \Big[ (x_* - 2|k| x'_*) \omega'(\frac{|k|^2}{x_*})  
+ (2x_*- |k|x'_*)\frac{|k|^2}{x_*} \omega''(\frac{|k|^2}{x_*}) \Big] 
\\
&= x_* \omega'(\frac{|k|^2}{x_*}) - \frac{3|k|^2}{x_*}[\omega'(\frac{|k|^2}{x_*})]^2 + 2|k|^2 \omega''(\frac{|k|^2}{x_*}) ,
\end{aligned}$$
and so $$
\begin{aligned}
\Big( &1 + \frac{|k|^2}{x_*^2}  \omega'(\frac{|k|^2}{x_*}) \Big)^2 \Big (2x_*''(k) x_*(k) - x_*'(k)^2 \Big) 
\\& = \frac{2(2x_* - |k| x'_*)}{x_*^2} 
 \Big[  x_* \omega'(\frac{|k|^2}{x_*}) - \frac{3|k|^2}{x_*}[\omega'(\frac{|k|^2}{x_*})]^2 + 2|k|^2 \omega''(\frac{|k|^2}{x_*}) \Big] .
\end{aligned}$$
Recalling \eqref{bd-coff1} and the fact that $2 x_* - |k|x'_* \gtrsim 1$, it suffices to study the terms in the bracket. Let $y_* = |k|^2/x_*$.  Recalling that $x_* = \omega(y_*)$, we consider 
$$A_* = \omega(y_*) \omega'(y_*) - 3y_* \omega'(y_*)^2 + 2  y_* \omega(y_*)\omega''(y_*) .$$
Recalling \eqref{def-omega}, we compute 
$$
\omega'(y) = 2\pi \int_{\{|u|<\Upsilon\}} \frac{u^4\mu(\frac12 u^2)}{(1- y u^2)^2}\;du 
, \qquad \omega''(y) = 4\pi \int_{\{|u|<\Upsilon\}} \frac{u^6\mu(\frac12 u^2)}{(1- y u^2)^3}\;du .
$$
By the H\"older's inequality, we have 
$$
\int_{\{|u|<\Upsilon\}} \frac{u^4\mu(\frac12 u^2)}{(1- y u^2)^2}\;du  \le \Big(  \int_{\{|u|<\Upsilon\}} \frac{u^2\mu(\frac12 u^2)}{1- y u^2}\;du \Big)^{1/2} \Big( \int_{\{|u|<\Upsilon\}} \frac{u^6\mu(\frac12 u^2)}{(1- y u^2)^3}\;du \Big)^{1/2}
$$
That is, $\omega'(y)^2 \le \frac12 \omega(y) \omega''(y)$. This yields 
$$
\begin{aligned}
A_* &= \omega(y_*) \omega'(y_*) - 3y_* \omega'(y_*)^2 + 2  y_* \omega(y_*)\omega''(y_*)
\\
&\ge  \omega(y_*) \omega'(y_*) +\frac12  y_* \omega(y_*)\omega''(y_*) 
\\
&\ge  \omega(0) \omega'(0) 
\end{aligned}$$
in which the last inequality was due to the monotonicity of $\omega(y)$ and $\omega'(y)$. Since $\omega(0)$ and $\omega'(0)$ are strictly positive, we have $A_* \gtrsim 1$, and hence $2x_*''(k) x_*(k) - x_*'(k)^2  \gtrsim 1$. The lower bound \eqref{lowbd-tauk} follows. 
This completes the proof of Theorem \ref{theo-LangmuirE}. 
\end{proof}

\subsection{Landau damping}

Theorem \ref{theo-LangmuirE} yields the existence of the solutions $\lambda_\pm(k) = \pm i \tau_*(k)$ to the dispersion relation  $D(\lambda,k) =0$ for $|k|\le \kappa_0$, while no pure imaginary solutions exist for $|k|>\kappa_0$. Note that by definition, $\kappa_0 =0$ when $\Upsilon =\infty$. In this section, we study how these curves $\lambda_\pm(k) $ leave the imaginary axis as $|k|>\kappa_0$, and in particular compute the Landau damping rate: necessarily $\Re \lambda_\pm(k) <0$ for $|k|>\kappa_0$ in view of Proposition \ref{prop-nogrowth}. 

Precisely, we obtain the following.

\begin{theorem}\label{theo-Landau}  Fix an $N_0\ge 4$, and let $\mu(\frac12|v|^2)$ be a non-negative equilibrium as described in Section \ref{sec-mu}, $\Upsilon$ be as in \eqref{def-Upsilon}, and {$\tau_j^2, \kappa_j^2$ as in \eqref{def-taukappa0}}. Then, for any $0<|k| - \kappa_0\ll1$, there are exactly two zeros $\lambda_\pm(k)$ of the electric dispersion relation $ D(\lambda,k) = 0$ that are {$C^{N_0-2}$ regular} in $|k|$, with $\lambda_\pm(\kappa_0) = \pm i \tau_*(\kappa_0)$, and satisfy the following Landau damping rates: 

\begin{itemize}

\item If $\Upsilon = \infty$, then for $|k|\ll1$, we have 
% \begin{equation}\label{Landau}
%\Re \lambda_\pm(k) = -  \frac{\pi ^2|k|}{\tau^2_0}  [u^4\mu(\frac12 u^2)]_{u = \nu_*(k)} 
% (1 + \cO(|k|)),
%\end{equation}
 \begin{equation}\label{Landau}
\Re \lambda_\pm(k) = -  \frac{\pi ^2}{\tau_0}  [u^3\mu(\frac12 u^2)]_{u = \nu_*(k)} 
 (1 + \cO(|k|)),
\end{equation}
where the phase velocity $\nu_*(k) = \frac{\tau_0 + \cO(|k|^2)}{|k|}$.  

% \begin{equation}\label{Landau}
%\Re \lambda_\pm(k) = -  \frac{\pi ^2\tau_0^2}{|k|^3} \mu\Big(\frac12\frac{\tau_0^2}{|k|^2}\Big) (1 + \cO(|k|)),
%\end{equation}
%

\item If $\Upsilon <\infty$, then for $0< |k| - \kappa_0\ll1$, we have  
\begin{equation}\label{Landau-cpt}
\Re \lambda_\pm(k) =-\frac{2\pi^2}{\kappa_1^2} [u\mu(\frac{1}{2}u^2)]_{u = \nu_*(k)} (1 + \cO(|k| - \kappa_0)),
\end{equation}
where  the phase velocity $\nu_*(k) = \Upsilon - \frac{2\kappa_0}{\kappa_1^2} (|k| - \kappa_0) + \cO((|k| - \kappa_0)^2)$.

% and $\kappa_1^2 = 4\pi \Upsilon\int_{\{|u|<\Upsilon\}} \frac{u^2\mu(\frac12 u^2)}{(\Upsilon^2-u^2)^2} \;du.$

%
%\begin{equation}\label{Landau-cpt}
%\Re \lambda_\pm(k) = -  \frac{2\pi ^2\kappa_0 \Upsilon}{\kappa_1^2} \mu(e_\Upsilon) (1 + \cO(|k| - \kappa_0)),
%\end{equation}
%
%where $e_\Upsilon= \frac12 [\Upsilon - \frac{2\kappa_0}{\kappa_1^2} (|k| - \kappa_0)]^2$ and $\kappa_1^2 = 4\pi \Upsilon\int_{\{|u|<\Upsilon\}} \frac{u^2\mu(\frac12 u^2)}{(\Upsilon^2-u^2)^2} \;du.$

\end{itemize}

\end{theorem}

\begin{proof}
To proceed, recalling from \eqref{def-Dlambda1}, we study the dispersion relation $|k|^2 = \cH(i\lambda/|k|)$, which we recall 
\begin{equation}\label{reexp-cH}
\begin{aligned}
\cH(z) 
&= 2\pi \int_{\RR} \frac{u \mu(\frac{1}{2}u^2)}{z-u} \;du 
%= \frac{2\pi}{z} \int_{\RR} \frac{u \mu(\frac{1}{2}u^2)}{1-u/z} \;du,
\end{aligned}\end{equation}
which is analytic in $\Im z> 0$ and sufficiently regular up to the real axis $\Im z =0$, see \eqref{Plemelj}. Observe that the function $\cH(z)$ may not have an analytic continuation past the real axis, since $\mu(\cdot)$ may not be analytic. {However, in the case when $\Upsilon<\infty$, it is classical \cite{Whi} that $\cH(z)$ can be extended holomorphically to $\CC\setminus [-\Upsilon, \Upsilon]$ as an analytic extension from $\cH(z)$ on the upper half plane $\{\Im z>0\}$, since the function $\mu(\frac12 u^2)$ has support contained in $\{|u|\le \Upsilon\}$ and vanishes rapidly at the maximal speed $|u|=\Upsilon$.} On the other hand, when viewing as a function in $\RR^2$ and using the uniform bounds \eqref{unibd-cH} in $C^{N_0-2}$, we may also apply the classical Whitney's extension theorem \cite{Whi} to extend $\cH(z)$ to small neighborhoods of $[-\Upsilon, \Upsilon]$, as a {$C^{N_0-2}$} smooth function in $z$. 
{Note however that the two extensions need not to be identical. Indeed, in view of the Plemelj's formula, see \eqref{Plemelj}, we compute the limiting value of $\cH(z)$ from the lower half plane $\{\Im z<0\}$, namely   
\begin{equation}\label{Plemelj1}
 \lim_{\tilde\gamma \to 0^-}\cH(i\tilde\gamma - \tilde\tau) 
=  -2\pi P.V. \int_{\{|u|<\Upsilon\}} \frac{u \mu(\frac{1}{2}u^2)}{u + \tilde\tau} du  {~-~} 2 i \pi^2 \tilde\tau \mu(\frac{1}{2}\tilde\tau^2) ,
\end{equation}
which is different from those coming from the upper half plane $\{\Im z>0\}$ due to the last term (noting the sign change), and therefore the analytic extension of $\cH(z)$ is not identical to the Whitney's $C^{N_0-2}$ extension, since $4 i \pi^2 \tilde\tau \mu(\frac{1}{2}\tilde\tau^2) \not =0$ for $|\tilde \tau| <\Upsilon$. In what follows, we focus the extension of $\cH(i\lambda/|k|)$ in the neighborhood of $\lambda_\pm(\kappa_0) = \pm i \tau_*(\kappa_0)$, with $\tau_*(\kappa_0) = \kappa_0\Upsilon$, or precisely the extension of $\cH(z)$ in the neighborhood of $z = \pm \Upsilon$, at which the last term in \eqref{Plemelj1} vanishes. As a result, the difference of the two extensions is negligible, since $\mu(\frac12 u^2)$ vanishes rapidly as $|u|\to \Upsilon$.} We shall now establish the existence and behavior of the zeros $\lambda_\pm(k)$ of the dispersion relation $|k|^2 = \cH(i\lambda/|k|)$ for $|k|$ sufficiently close to $\kappa_0$.

\subsubsection*{Case 1: $\Upsilon=\infty$}

Let us start with the case when $\Upsilon = \infty$. Note that in the case, we have $\kappa_0 =0$, $\tau_*(0) = \tau_0$, and so we study the dispersion relation for $|k|\ll1$. Using the geometric series of $\frac{1}{1-y}$, we write  
$$ \frac{1}{z-u} = \frac{1}{z}\frac{1}{1-u/z} = \frac{1}{z}\sum_{j=0}^{2m} \frac{u^j}{z^j} + \frac{u^{2m+1}}{z^{2m+1} (z-u)}$$
for any $m\ge 0$. Putting this into \eqref{reexp-cH} and using the fact that $\mu(\frac12 u^2)$ is even in $u$, we get 
\begin{equation}\label{exp-cH}
\begin{aligned}
\cH(z) 
&= 
2\pi\sum_{j=0}^{m-1}\frac{1}{z^{2j+2}}\int_{\RR} u^{2j+2} \mu(\frac{1}{2}u^2) \;du + \frac{2\pi}{z^{2m+1}}\int_{\RR} \frac{u^{2m+2} \mu(\frac{1}{2}u^2)}{z-u} \;du,
\end{aligned}\end{equation}
for $0\le m \le N_0-2$, where $N_0$ is defined as in \eqref{decay-mu}. 
Set
\begin{equation}\label{def-tauj}\tau_j^2 =2\pi \int_{\RR} u^{2j+2} \mu(\frac{1}{2}u^2) \;du, \qquad \cH^R(z) = - \frac{2\pi}{z^{2m-1}}\int_{\RR} \frac{u^{2m+2} \mu(\frac{1}{2}u^2)}{z-u} \;du.
\end{equation}
The dispersion relation $|k|^2 = \cH(i\lambda/|k|)$ then becomes
\begin{equation}\label{disp1}
\lambda^2 = -  \sum_{j=0}^{m-1} (-1)^j\frac{|k|^{2j}}{\lambda^{2j}} \tau_j^2+ \cH^R(i\lambda/|k|),
 \end{equation}
 where 
$$\cH^R(i\lambda/|k|) =- \frac{2i\pi  (-1)^{m}|k|^{2m-1}}{\lambda^{2m-1}}\int_{\RR} \frac{u^{2m+2} \mu(\frac{1}{2}u^2)}{i\lambda/|k| -u} \;du.$$
Following \eqref{bd-cH1}, we have $|\cH^R(i\lambda/|k|)| \lesssim |k|^{2m-1} |\lambda|^{-2m+1}$. The existence of $\lambda_\pm(k)$ that satisfies the dispersion relation \eqref{disp1} follows from the implicit function theorem for $|k|\ll1$. { Indeed, at $\lambda = \pm i\tau_0$ and $k=0$, the $\lambda$-derivative of the left hand side of \eqref{disp1} is equal to $2\lambda = \pm 2i \tau_0 \not =0$, while the $\lambda$-derivative of the right hand side is clearly bounded by $C_0|k|^2$, since $\lambda = \pm i\tau_0 \not =0$. The implicit function theorem applies to the equation \eqref{disp1} yields the existence of $\lambda_\pm(k)$ for $|k|\ll1$.} Note that for $\lambda = i\tau$, the summation 
$$\sum_{j=0}^{m-1} (-1)^j\frac{|k|^{2j}}{\lambda^{2j}} \tau_j^2 = \sum_{j=0}^{m-1} \frac{|k|^{2j}}{\tau^{2j}} \tau_j^2$$
gives a real and strictly positive value of order $\tau_0^2 + \cO(|k|^2)$ for every real value $\tau$. {Recalling $\lambda = i\tau$,} this yields ${\lambda_\pm(k) = \pm i (\tau_0 + \cO(|k|^2))}$ for $|k|\ll1$. By induction, 
we in fact obtain 
\begin{equation}\label{exp-lambdapm} 
\lambda_\pm(k) = \pm i \sum_{j=0}^{m-1} a_j |k|^2 + \cO(|k|^{2m-1})
\end{equation}
for $|k|\ll1$, where $a_j$ are some nonnegative coefficients and can be computed in terms of $\tau_j$. For instance, $a_0 = \tau_0$ and $a_1 = \frac{\tau_1^2}{2\tau_0^3}$. In particular, taking $m = N_0-2$, we have 
\begin{equation}\label{upper-Relambda}
|\Re \lambda_\pm(k)| \lesssim |k|^{2N_0-5}
\end{equation} 
for $|k|\ll1$. Let us study further the real part of the dispersion relation $\lambda_\pm(k)$. From \eqref{disp1}, with $m=1$, the curves $\lambda_\pm(k)$ solve  
$$\lambda^2 = - \tau_0^2 - \frac{2i\pi|k|}{\lambda} \int_{\RR} \frac{u^{4} \mu(\frac{1}{2}u^2)}{i\lambda/|k| -u} \;du$$
where we have dropped $\pm$ for sake of presentation. Write $\lambda = \gamma + i\tau$. Note that $\gamma < 0$, since no solution exists for $\gamma\ge 0$ as shown in Proposition \ref{prop-nogrowth} and Theorem \ref{theo-LangmuirE}. Taking the imaginary part of the above, we get  
$$ 2 \gamma \tau =  - \frac{2\pi|k| \gamma}{|\lambda|^2} \Re \int_{\RR} \frac{u^{4} \mu(\frac{1}{2}u^2)}{i\lambda/|k| -u} \;du - \frac{2\pi|k| \tau}{|\lambda|^2} \Im \int_{\RR} \frac{u^{4} \mu(\frac{1}{2}u^2)}{i\lambda/|k| -u} \;du$$
which yields 
$$ \Big[ 2 \tau + {\frac{2\pi|k|}{|\lambda|^2}} \Re \int_{\RR} \frac{u^{4} \mu(\frac{1}{2}u^2)}{i\lambda/|k| -u} \;du\Big] \gamma = - \frac{2\pi|k| \tau}{|\lambda|^2} \Im \int_{\RR} \frac{u^{4} \mu(\frac{1}{2}u^2)}{i\lambda/|k|-u} \;du.$$
Following \eqref{bd-cH1}, the integral term on the left hand side is bounded, while $\tau = \pm (\tau_0 + \cO(|k|^2))$ and $\gamma = \cO(|k|^{2N_0-5})$. This yields 
\begin{equation}\label{est-gamma0}
 \Big[  \tau_0 + \cO(|k|)\Big] \gamma = \frac{\pi|k|}{\tau_0} \Im \int_{\RR} \frac{u^{4} \mu(\frac{1}{2}u^2)}{u+ \tau/|k| - i\gamma/|k|} \;du .
\end{equation}
Since $\gamma = \cO(|k|^{2N_0-5})$ for some large $N_0$ as in \eqref{decay-mu}, $\gamma/|k|$ tends to $0$ as $|k|\to 0$. We thus write 
$$
\begin{aligned}
\Im \int_{\RR} \frac{u^{4} \mu(\frac{1}{2}u^2)}{u+ \tau/|k| - i\gamma/|k|} \;du  &= \frac{\gamma}{|k|}\int_{\RR} \frac{u^{4} \mu(\frac{1}{2}u^2)}{(u+ \tau/|k|)^2 + \gamma^2/|k|^2} \;du.
%\\& =\frac{\tau^4}{|k|^4} \mu\Big(\frac{\tau^2}{2|k|^2}\Big)\int_{\RR} \frac{\gamma/|k|}{(u+ \tau/|k|)^2 + \gamma^2/|k|^2} \;du 
%\\&\quad + 
%\int_{\RR} \frac{\gamma/|k|}{(u+ \tau/|k|)^2 + \gamma^2/|k|^2} \Big[ u^{4} \mu(\frac{1}{2}u^2) - \frac{\tau^4}{|k|^4} \mu\Big(\frac{\tau^2}{2|k|^2}\Big)\Big]\;du.
\end{aligned}$$
In the region where $|u|\le \frac{\tau_0}{2|k|}$, we note that $|u+ \tau/|k||\ge \frac{\tau_0}{2|k|}$ for sufficiently small $|k|$. Therefore, 
$$ \frac{|\gamma|}{|k|}\int_{\{ |u|\le \frac{\tau_0}{2|k|}\}} \frac{u^{4} \mu(\frac{1}{2}u^2)}{(u+ \tau/|k|)^2 + \gamma^2/|k|^2} \;du\le \frac{4 |\gamma| |k|}{\tau_0^2} \int u^4\mu(\frac{1}{2}u^2)\;du \le |\gamma| \cO(|k|),
$$
which can be put on the left hand side of \eqref{est-gamma0}. As for the integral over $|u|\ge \frac{\tau_0}{2|k|}$, we write 
$$
\begin{aligned}
 \frac{\gamma}{|k|}&\int_{\{ |u|\ge \frac{\tau_0}{2|k|}\}} \frac{u^{4} \mu(\frac{1}{2}u^2)}{(u+ \tau/|k|)^2 + \gamma^2/|k|^2} \;du
 \\&=\frac{\tau^4}{|k|^4} \mu\Big(\frac{\tau^2}{2|k|^2}\Big)\int_{\{ |u|\ge \frac{\tau_0}{2|k|}\}} \frac{\gamma/|k|}{(u+ \tau/|k|)^2 + \gamma^2/|k|^2} \;du 
\\&\quad + 
\int_{\{ |u|\ge \frac{\tau_0}{2|k|}\}} \frac{\gamma/|k|}{(u+ \tau/|k|)^2 + \gamma^2/|k|^2} \Big[ u^{4} \mu(\frac{1}{2}u^2) - \frac{\tau^4}{|k|^4} \mu\Big(\frac{\tau^2}{2|k|^2}\Big)\Big]\;du
\end{aligned}$$
in which the last integral is clearly bounded by $|\gamma|\cO(|k|)$, since $\mu(\cdot)$ and its derivatives decay rapidly to zero, and so it can be put on the left hand side of  \eqref{est-gamma0}. On the other hand, recalling $\gamma<0$ and considering the case when $\tau<0$, we compute 
$$
\int_{\{ |u|\ge  \frac{\tau_0}{2|k|}\}} \frac{\gamma/|k|}{(u+ \tau/|k|)^2 + \gamma^2/|k|^2} \;du = - \frac{\pi}{2} + \mathrm{arctan}(\frac{\tau_0+2\tau}{2|k|}) = - \pi + \cO(e^{-\tau_0/(2|k|)}).
$$
The case when $\tau>0$ is done similarly. Putting these into \eqref{est-gamma0}, we thus obtain 
$$
 \Big[  \tau_0 + \cO(|k|)\Big] \gamma =-  \frac{\pi ^2|k|}{\tau_0}  \frac{\tau^4}{|k|^4} \mu\Big(\frac{\tau^2}{2|k|^2}\Big) (1 + \cO(e^{-\tau_0/(2|k|)})).
$$
This proves \eqref{Landau}, upon recalling  $\tau = \pm (\tau_0 + \cO(|k|^2))$ and $\nu_* = \tau_*/|k|$.

\subsubsection*{Case 2: $\Upsilon<\infty$}

We now study the case when $\Upsilon < \infty$. In this case we recall that $\kappa_0>0$ and $\tau_*(\kappa_0) = \kappa_0 \Upsilon$. We shall study the dispersion relation $D(\lambda,k)=0$ for $|k| \to \kappa_0$ and $\lambda \to \pm i \tau_*(\kappa_0)$. We focus on the case when $\lambda \sim - i\tau_*(\kappa_0)$; the $+$ case is similar. 
As in the previous case, we use the geometric series of $\frac{1}{1-y}$ to write 
$$ \frac{1}{z-u} = \frac{1}{\Upsilon-u}\frac{1}{1-\frac{\Upsilon-z}{\Upsilon-u}} = \frac{1}{\Upsilon-u}\sum_{j=0}^{m} \frac{(\Upsilon-z)^j}{(\Upsilon-u)^j} + \frac{(\Upsilon-z)^{m+1}}{(\Upsilon-u)^{m+1} (z-u)}.$$
 Therefore, we get from \eqref{reexp-cH} that 
\begin{equation}\label{exp-cHcpt}
\begin{aligned}
\cH(z) 
&= 
2\pi\sum_{j=0}^{m}(\Upsilon-z)^j \int_{\RR} \frac{u\mu(\frac{1}{2}u^2)}{(\Upsilon-u)^{j+1}} \;du + 2\pi(\Upsilon-z)^{m+1}\int_{\RR} \frac{u \mu(\frac{1}{2}u^2)}{(\Upsilon-u)^{m+1} (z-u)} \;du,
\end{aligned}\end{equation}
for $0\le m \le N_0-2$, where $N_0$ is defined as in \eqref{decay-mu}. 
Set
$$\kappa_j^2 =2\pi \int_{\RR} \frac{u\mu(\frac{1}{2}u^2)}{(\Upsilon-u)^{j+1}} \;du = 2\pi \int_{\RR} \frac{u(\Upsilon+u)^{j+1}\mu(\frac{1}{2}u^2)}{(\Upsilon^2-u^2)^{j+1}} \;du,$$
which are strictly positive numbers (upon using the radial symmetry of $\mu(\cdot)$). 
The dispersion relation $|k|^2 = \cH(z)$ then becomes
\begin{equation}\label{disp1-cpt}
|k|^2 =   \kappa^2_0+ \sum_{j=1}^{m} \kappa_j^2 (\Upsilon-z)^j + \cR^H(z)(\Upsilon-z)^{m+1},
 \end{equation}
 where 
$$\cR^H(z) = 2\pi\int_{\RR} \frac{u \mu(\frac{1}{2}u^2)}{(\Upsilon-u)^{m+1} (z-u)} \;du.$$
Following \eqref{bd-cH1}, we have $|\cR^H(z)| \lesssim 1$. The existence of $z_-(k)$, or $\lambda_- = -iz_-(k)|k|$, that satisfies the dispersion relation \eqref{disp1-cpt} follows from the implicit function theorem for $0<|k|-\kappa_0\ll1$, noting that the coefficients $\kappa_j$ are strictly positive {and $z$-derivative of the right hand side of \eqref{disp1-cpt} is equal to $-\kappa_1 + \mathcal{O}(|k|-\kappa_0)$, which is non zero for $|k|-\kappa_0\ll1$, since $\kappa_1\not =0$}. 
Since there are no pure imaginary solutions $\lambda_-(k)$ for $|k|>\kappa_0$, the solution $z_-(k)$ must be complex and $\Im z_-(k) >0$ for $|k|>\kappa_0$. In addition, we claim that 
\begin{equation}\label{vanishingImz}  
\begin{aligned}
|\Re z_-(k) - \Upsilon + \frac{2\kappa_0}{\kappa_1^2} (|k| - \kappa_0)| &\lesssim (|k| - \kappa_0)^2,
\\
 |\Im z_-(k)|&\lesssim (|k|-\kappa_0)^{N_0-2}
\end{aligned}\end{equation}
for $|k|$ sufficiently close to $\kappa_0^+$, where $N_0$ is given as in \eqref{decay-mu}. The estimate on $\Re z_-$ is direct from \eqref{disp1-cpt}, recalling $\kappa_1^2$ is strictly positive. On the other hand, note that $z_-(\kappa_0) = \Upsilon$ and in view of \eqref{disp1-cpt}, $\partial_z^n\cH(z)$ are all real numbers at $z = \Upsilon$ for $0\le n\le m$. Hence, taking the derivative in $|k|$ of  the equation $\cH(z_-(k)) = |k|^2$ and evaluating the result at $|k|=\kappa_0$, we get 
$$ \cH'(\Upsilon) z_-'(\kappa_0) = 2 \kappa_0,$$ 
which yields that $z_-'(\kappa_0) $ is real-valued and is qual to $-2\kappa_0/\kappa_1^2$, since $\cH'(\Upsilon) =- \kappa_1^2$ is real valued. Thus, by induction, together with a use of the Fa\`a di Bruno's formula for derivatives of a composite function, $\partial_k^\alpha z_-(\kappa_0)$ are real valued for $0\le |\alpha| \le m$. This proves the claim \eqref{vanishingImz} for $|k|$ sufficiently close to $\kappa_0$ via the standard Taylor's expansion series.

Let us study further the real part of the dispersion relation $\lambda_-(k)=-iz_-(k) |k|$. From \eqref{disp1-cpt}, with $m=1$, the curve $z_-(k)$ solves  
$$
|k|^2 =   \kappa^2_0 + \kappa_1^2 (\Upsilon-z)+ 2\pi(\Upsilon-z)^2 \int_{\RR} \frac{u \mu(\frac{1}{2}u^2)}{(\Upsilon-u)^2(z-u)} \;du. 
$$
Writing $z = i\gamma + \tau$, with $\gamma<0$, and taking the imaginary part of the above identity, we get 
$$\Big(\kappa_1^2 + 4 \pi (\Upsilon - \tau) \Re \int_{\RR} \frac{u \mu(\frac{1}{2}u^2)}{(\Upsilon-u)^2(z-u)} \;du\Big) \gamma = 2\pi [(\Upsilon - \tau)^2 - \gamma^2]\Im \int_{\RR} \frac{u \mu(\frac{1}{2}u^2)}{(\Upsilon-u)^2(z-u)} \;du.$$
Note that for $|k|\to \kappa^+_0$, $\tau \to \Upsilon$, while $\gamma = \cO((|k|-\kappa_0)^{N_0-2})$, see \eqref{vanishingImz}. The above thus yields 
\begin{equation}\label{est-gamma0-cpt}
 \Big[ \kappa_1^2 + \cO(|k| - \kappa_0)\Big] \gamma = 2\pi (\Upsilon - \tau)^2\Im\int_{\RR} \frac{u \mu(\frac{1}{2}u^2)}{(\Upsilon-u)^2(z-u)} \;du .
\end{equation}
It remains to study the integral on the right. Indeed, for $z = i\gamma + \tau$, we have 
$$
\begin{aligned}
\Im\int_{\RR} \frac{u \mu(\frac{1}{2}u^2)}{(\Upsilon-u)^2(z-u)} \;du = -  \int_{\RR} \frac{ \gamma u \mu(\frac{1}{2}u^2)}{(\Upsilon-u)^2 (\gamma^2 + (u-\tau)^2)} \;du.
\end{aligned}$$
Hence, exactly as done in the previous case, since $\gamma \lesssim (|k|-\kappa_0)^{N_0-2}\to 0$ as $|k|\to \kappa_0$, the kernel $\frac{\gamma}{\gamma^2 + (u-\tau)^2}$ is approximated the Dirac delta function at $u = \tau$, yielding 
$$
 \Big[ \kappa_1^2 + \cO(|k| - \kappa_0)\Big] \gamma = -2\pi^2 \tau\mu(\frac{1}{2}\tau^2).
$$
Recalling $\tau = \Re z_-(k)$, $\gamma = \Im z_-(k)$, and using again \eqref{vanishingImz}, we obtain \eqref{Landau-cpt} for $\lambda_- = -iz_-(k)|k|$. The bounds for $\lambda_+(k)$ follow similarly. 
This completes the proof of Theorem \ref{theo-Landau}. 
\end{proof}

%%%%%

\section{Green function}\label{sec-Green}

%%%%%

In view of the resolvent equation \eqref{resolvent} for the electric potential $\phi$, we introduce the resolvent kernel
\begin{equation}\label{def-THk}
\TG_{k}(\lambda) :=  \frac{1}{D(\lambda,k)} ,
\end{equation}
and the corresponding temporal Green function
\begin{equation}\label{def-FHk}
\begin{aligned}
\FG_{k}(t) &=  \frac{1}{2\pi i}\int_{\{\Re \lambda = \gamma_0\}}e^{\lambda t} \TG_{k}(\lambda)\; d\lambda,
\end{aligned}\end{equation}
which are well-defined for $\gamma_0 >0$, recalling from Lemma \ref{lem-Dlambda} that $D(\lambda,k)$ is holomorphic in $\Re \lambda>0$. The main goal of the remainder of this section is to establish decay estimates for the Green function through the representation \eqref{def-FHk}. We stress that since $\mu(v)$ may not be analytic in $v$, the resolvent kernels $\TG_k(\lambda)$ may not have an analytic extension to the stable half plane in $\Re \lambda <0$. As a consequence, isolating the poles to compute the residue of $\TG_k(\lambda)$ and deriving decay estimates for the remainder turn out to be rather delicate (c.f. \cite{BMM-lin,HKNF3}).

\subsection{Green function in Fourier space}

We first study the electric Green function $\FG_k(t)$. We obtain the following. 

\begin{proposition}\label{prop-GreenG} Let $\FG_k(t)$ be defined as in \eqref{def-FHk}, and let ${\lambda_\pm(k)}$
be the electric dispersion relation constructed in Theorems \ref{theo-LangmuirE} and \ref{theo-Landau}. Then, we can write 
\begin{equation}\label{decomp-FG} 
\FG_k(t) = \delta(t) + \sum_\pm \FG^{osc}_{k,\pm}(t)  +   \FG^{r}_k(t) ,
\end{equation}
with 
$$
\FG^{osc}_{k,\pm}(t) = e^{\lambda_\pm(k)t} a_\pm(k),
$$
for some sufficiently smooth functions $a_\pm(k)$ whose support is contained in $\{|k|\le \kappa_0 + 1\}$, with $a_\pm(0) =  \pm { i \tau_0 \over 2} $. In addition, there hold 
\begin{equation}\label{bd-Gr}
| |k|^{|\alpha|}\partial_k^\alpha \FG^{r}_k(t)| \le C_0 |k|^3 \langle k\rangle^{-4} \langle kt\rangle^{-K_0+|\alpha|} ,
\end{equation}
{uniformly in $k\in \RR^3$ for some universal constant $C_0$ and for $0\le |\alpha| <  K_0$.  
}

\end{proposition}

\begin{figure}[t]
\centering
\subfigure{
\includegraphics[scale=.5]{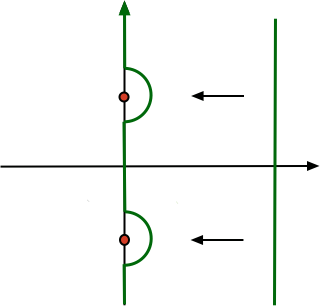}
}
%\hspace{.8in}
\subfigure{
\includegraphics[scale=.5]{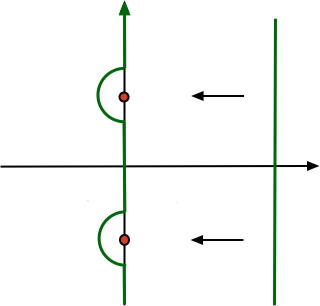}
}
\put(-284,60){$0$}
\put(-110,60){$0$}
\put(-261,90){$\cC_+$}
\put(-261,17){$\cC_-$}
\put(-127,90){$\cC^*_+$}
\put(-127,17){$\cC^*_-$}
\put(-22,1){$\Gamma = \{\Re \lambda \gg1\}$}
\put(-196,1){$\Gamma = \{\Re \lambda \gg1\}$}
\put(-95,150){$\mathbb{C}$}
\put(-268,150){$\mathbb{C}$}
\put(-294,102){$i\tau_*$}
\put(-301,30){$-i\tau_*$}
\put(-95,102){$i\tau_*$}
\put(-98,30){$-i\tau_*$}
\caption{\em Illustrated on the left is the contour of integration $\Gamma =  \Gamma_1 \cup \Gamma_2 \cup \cC_\pm$, while on the right is the contour $\Gamma^* =  \Gamma_1 \cup \Gamma_2 \cup \cC^*_\pm$.}
\label{fig-contour0}
\end{figure}

\begin{proof} The case when $|k|\gtrsim 1$ was easier and much studied; see, e.g., \cite{MV, GNR1, HKNF2}, where no oscillatory component of the Green function was present. We shall focus on the singular case when $|k|\le \kappa_0 +1$, where $\TG_k(\lambda)$ has poles at $\lambda_\pm(k)$. The issue was to isolate the poles of the resolvent kernel $\TG_k(\lambda)$ without analyticity past the imaginary axis. As the resolvent kernel $\TG_k(\lambda) = \frac{1}{D(\lambda,k)}$ is holomorphic in $\{\Re \lambda >0\}$, by Cauchy's integral theorem, we may move the contour of integration $\Gamma = \{\Re \lambda = \gamma_0\}$ towards the imaginary axis so that 
\begin{equation}\label{int-FGdecomp}
\FG_{k}(t) =  \frac{1}{2\pi i}\int_{\Gamma}e^{\lambda t} \TG_{k}(\lambda)\; d\lambda,
\end{equation}
where we have decomposed $\Gamma$, as depicted in Figure \ref{fig-contour0}, into 
\begin{equation}\label{def-Gamma} 
\Gamma = \Gamma_1 \cup \Gamma_2 \cup \cC_\pm
\end{equation}
having set 
$$
\begin{aligned}
 \Gamma_1 &= \{ \lambda = i\tau, \quad |\tau \pm \tau_*(k)|\ge  |k|, \quad |\tau| > |k|\Upsilon_*\} ,
 \\
  \Gamma_2 &= \{\lambda = i\tau, \quad |\tau \pm \tau_*(k)|\ge  |k|, \quad |\tau| \le |k|\Upsilon_*\} ,
   \\
  \cC_\pm &= \{\Re \lambda \ge 0, \quad |\lambda \mp {i\tau_*(k)}| =  |k|\} ,
    \end{aligned}$$
where 
\begin{equation}\label{def-Upsilonstar}
\Upsilon_* = \left\{ \begin{aligned}
\Upsilon, \qquad & \mbox{if}\quad \Upsilon<\infty, 
\\M, \qquad &\mbox{if}\quad \Upsilon = \infty,
\end{aligned}
\right.
\end{equation}    
for some sufficiently large constant $M>0$. 
Note that the semicircle $\cC_\pm$ is to avoid the singularity due to the poles at $\lambda_\pm(k)$ of $\TG_{k}(\lambda)$, while the integrals over $\Gamma_1$ and $\Gamma_2$ are understood as taking the limit of $\Re \lambda \to 0^+$. To establish decay in $t$ as claimed in \eqref{bd-Gr}, we may integrate by parts in $\lambda$ repeatedly, for $\Re \lambda>0$, and then take the limit of $\Re \lambda \to 0^+$, to get 
\begin{equation}\label{int-FGdecomp-n}
\FG_{k}(t) =  \frac{(-1)^n}{2\pi i (|k|t)^n }\int_{\Gamma}e^{\lambda t} |k|^n \partial_\lambda^n \TG_{k}(\lambda)\; d\lambda,
\end{equation}
for any $n\ge 0$,  without introducing any boundary terms. In what follows, we shall use the formulation \eqref{int-FGdecomp-n}, instead of \eqref{int-FGdecomp}, to bound the Green function $\FG_{k}(t) $. 

\subsubsection*{Bounds on $D(\lambda,k)$}

Recall from \eqref{def-Dlambda1} that $D(\lambda,k) = 1 - \frac{1}{|k|^2}\cH(i\lambda/|k|)$. Therefore, using the expansion \eqref{exp-cH} and \eqref{def-tauj}, with $m=2$, we may write  
\begin{equation}\label{expansion-DL}
D(\lambda,k) = 1 + \frac{\tau_0^2}{\lambda^2} - \frac{\tau_1^2 |k|^2}{\lambda^4} - \frac{1}{\lambda^4}\cR(\lambda,k)
\end{equation}
where the remainder $\cR(\lambda,k) $ is defined by 
\begin{equation}\label{def-cRtau}\cR(\lambda,k) = 2\pi|k|^2\int_{\RR} \frac{u^5 \mu(\frac{1}{2}u^2)}{i\lambda/|k|-u} \;du = -2i\pi|k|^3\int_{\RR} \frac{u^5 \mu(\frac{1}{2}u^2)}{\lambda + iu|k|} \;du .
\end{equation}
Let us first bound $\cR(\lambda,k) $. For $\Re\lambda>0$, we write 
\begin{equation}\label{def-cRtau1}
\cR(\lambda,k) = -2i\pi|k|^3 \int_0^\infty e^{-\lambda t} \int_{\RR} e^{-i|k| t u}u^5 \mu(\frac{1}{2}u^2) \;dudt
 \end{equation}
in which the $u$-integration is bounded by $C_0 \langle kt\rangle^{-K_0}$, upon repeatedly integrating by parts in $u$ and using the regularity assumption on $\mu(\cdot)$. This yields 
\begin{equation}\label{good1st-bdR}
|\partial_\lambda^n\mathcal{R}(\lambda,k) |\lesssim \int_0^\infty |k|^3 t^n \langle kt\rangle^{-K_0} dt\lesssim |k|^{2-n},\end{equation}
uniformly for all $\Re \lambda \ge 0$ and for $0\le n\le K_0-2$. In addition, we may further expand the remainder $\mathcal{R}(\lambda,k)$, up to higher orders in $z = i\lambda/|k|$, as done in \eqref{exp-cH}, we obtain 
\begin{equation}\label{good-Rb}
|\partial_\lambda^n \mathcal{R}(\lambda,k) |\lesssim |k|^4 |\lambda|^{-2-n},
\end{equation}
for $0\le n\le K_0-2$, where we stress that the bounds hold uniformly for $\Re \lambda \ge 0$. In particular, as seen below, we may treat $\mathcal{R}(\lambda,k) $ as a remainder in the region where $|\lambda|\gg |k|$.

\subsubsection*{Decomposition of $\FG_k(t)$.}

Let us now bound the Green function $\FG_{k}(t)$ via the representation \eqref{int-FGdecomp-n}. 
By definition, we write  
$$
\begin{aligned}
\TG_k(\lambda ) &=  \frac{1}{D(\lambda,k)}   = 1 + \frac{1-D(\lambda,k)}{D(\lambda,k)}  .
\end{aligned}
$$
Using the expansion \eqref{expansion-DL} of ${D(\lambda,k)}$, we write 
$$
\begin{aligned}
\TG_k(\lambda ) 
&=1 +  \frac{ - \lambda^2 \tau_0^2+  |k|^2 \tau_1^2+ \mathcal{R}(\lambda,k)}{\lambda^4 +\lambda^2 \tau_0^2 -  |k|^2\tau_1^2
- \mathcal{R}(\lambda,k)} .
\end{aligned}
$$
In viewing $\cR(\lambda,k)$ as a remainder, we further decompose 
\begin{equation}\label{TG-decompose}
\begin{aligned}
\TG_k(\lambda) 
&=1 + \TG_{k,0}(\lambda) + \TG_{k,1}(\lambda ) 
\end{aligned}
\end{equation}
where 
\begin{equation}\label{def-TGplus}
\begin{aligned}
\TG_{k,0}(\lambda) & =   \frac{-\lambda^2 \tau_0^2+  |k|^2 \tau_1^2}{\lambda^4 + \lambda^2\tau_0^2 - |k|^2\tau_1^2} 
\\
\TG_{k,1}( \lambda ) 
&= \frac{\lambda^4  \mathcal{R}(\lambda,k) }{(\lambda^4 +\lambda^2 \tau_0^2-  |k|^2\tau_1^2)(\lambda^4 +\lambda^2\tau_0^2 -  |k|^2\tau_1^2
- \mathcal{R}(\lambda,k))} .
\end{aligned}
\end{equation}
We also denote by $\FG_{k,0}(t)$ and $\FG_{k,1}(t)$ the corresponding Green function, see \eqref{int-FGdecomp}. Denote by $\lambda_{\pm,0}(k)$ and $\lambda_\pm(k)$ the pure imaginary poles of $\TG_{k,0}(\lambda)$ and $\TG_{k}(\lambda)$, respectively. 
Via the representation \eqref{int-FGdecomp}, we shall prove that  
\begin{equation}\label{bd-Grk0}
\Big| \FG_{k,0}(t)  -  \sum_\pm e^{\lambda_{\pm,0}t} \mathrm{Res} (\TG_{k,0}(\lambda_{\pm,0}))  \Big| \lesssim |k|^3 e^{- \tau_1 |kt|} ,
\end{equation}
while via the representation \eqref{int-FGdecomp-n}, we claim that 
\begin{equation}\label{bd-Grk1}
\Big| \FG_{k,1}(t)  -  \sum_\pm e^{\lambda_{\pm}t}\mathrm{Res} (\TG_{k}(\lambda_{\pm}))  +  \sum_\pm e^{\lambda_{\pm,0}t} \mathrm{Res} (\TG_{k,0}(\lambda_{\pm,0}))   \Big| \le C_1 |k|^3 \langle kt\rangle^{-K_0} .
\end{equation}
In view of the decomposition \eqref{TG-decompose}, this would complete the proof of the decomposition \eqref{decomp-FG} and the remainder bounds \eqref{bd-Gr}, noting the residue $\mathrm{Res} (\TG_{k,0}(\lambda_{\pm,0}))$ in \eqref{bd-Grk1} is cancelled out with that in \eqref{bd-Grk0}.

\subsubsection*{Bounds on $\FG_{k,0}(t)$.}

Let us start with bounds on $\FG_{k,0}(t)$ via the representation \eqref{int-FGdecomp}. 
Note that the polynomial $\lambda^4 + \lambda^2\tau_0^2 - |k|^2\tau_1^2$ has four distinct roots: 
\begin{equation}\label{def-lambda1234}\lambda_{\pm,0} = \pm i \Big(\frac{\tau_0^2+\sqrt{\tau_0^4+ 4 \tau_1^2 |k|^2}}{2}\Big)^{1/2} , \qquad \mu_{\pm,0} = \pm \Big( \frac{2\tau_1^2 |k|^2}{\tau_0^2 + \sqrt{\tau_0^4 + 4 \tau_1^2 |k|^2}}\Big)^{1/2} .\end{equation}
Note that $\lambda_{\pm,0}$ are pure imaginary roots, while $\mu_{\pm,0}$ are real valued. Note in particular that 
\begin{equation}\label{fake-roots}|\lambda_{\pm,0}(k) \mp i \tau_*(k)| \lesssim |k|^2 , \qquad |\mu_{\pm,0}(k) \mp \tau_1 |k|| \lesssim |k|^2.\end{equation}
Except $\mu_{+,0} \sim \tau_1|k|$, the other three roots lie to the left of $\Gamma$, and thus, by Cauchy's theorem, we compute 
$$
\begin{aligned} 
\FG_{k,0}(t)& =\frac{1}{2\pi i } \int_\Gamma e^{\lambda t} \TG_{k,0}(\lambda) \; d\lambda
\\&= \sum_\pm e^{\lambda_{\pm,0}t} \mathrm{Res} (\TG_{k,0}( \lambda_{\pm,0}))  + e^{\mu_{-,0}t} \mathrm{Res} (\TG_{k,0}(\mu_{-,0}))  
\\&\quad+ \frac{1}{2\pi i}\lim_{\gamma_0 \to -\infty}\int_{\Re \lambda = \gamma_0}
e^{\lambda t} \TG_{k,0}(\lambda)\; d\lambda
\\&=  \sum_\pm e^{\lambda_{\pm,0}t} \mathrm{Res} (\TG_{k,0}(\lambda_{\pm,0}))   + e^{\mu_{-,0}t} \mathrm{Res} (\TG_{k,0}(\mu_{-,0}))  ,
\end{aligned}$$
upon using the fact that the last integral vanishes in the limit of $\gamma_0 \to -\infty$, noting that $\TG_{k,0}(\lambda) $ decays at order $\lambda^{-2}$ for $|\lambda|\to \infty$. It remains to bound $ e^{\mu_{-,0}t} \mathrm{Res} (\TG_{k,0}(\mu_{-,0})) $. Indeed, a direct calculation yields 
\begin{equation}\label{fakeres-G0}
\begin{aligned}
e^{\mu_{-,0}t} \mathrm{Res} (\TG_{k,0}(\mu_{-,0})) 
&=  \frac{\mu_{-,0}^4 e^{\mu_{-,0}t} }{(\mu_{-,0} - \lambda_{+,0})(\mu_{-,0} - \lambda_{-,0})(\mu_{-,0} - \mu_{+,0})}
\\
& = \frac{\mu_{-,0}^3 e^{\mu_{-,0}t} }{2(\mu_{-,0} - \lambda_{+,0})(\mu_{-,0} - \lambda_{-,0})}
\end{aligned}\end{equation}
upon using $\mu_{+,0} = -\mu_{-,0}$. By definition, we note that $\lambda_{\pm,0} = \pm i \tau_0(1 + \cO(|k|^2))$ and $\mu_{\pm,0} = \pm \sqrt{ \tau_1^2} |k| (1 + \cO(|k|^2))$. 
Hence, 
$$e^{\mu_{-,0}t} |\mathrm{Res} (\TG_{k,0}(\mu_{-,0})) | \lesssim |k|^3 e^{- \tau_1 |kt|},$$
giving \eqref{bd-Grk0}.

\subsubsection*{Bounds on $\TG_{k,1}(\lambda)$ on $\Gamma_1$.}

Next, we prove \eqref{bd-Grk1} via the representation \eqref{int-FGdecomp-n}. We start with the integral of $\TG_{k,1}(\lambda)$ on $\Gamma_1$: namely for $\lambda = i\tau$, where $|\tau|> |k|\Upsilon_*$ and $|\tau \pm \tau_*| \ge  |k|$. 
Recall that the polynomial $\tau^4 -\tau^2 \tau_0^2- |k|^2\tau_1^2$ has four distinct roots $\tau_{\pm,0} \sim \pm \tau_0 $ and $\tau_{\pm,1} \sim \pm i\tau_1 |k|$. Note also that $|\tau_* - \tau_0|\lesssim |k|^2$. Therefore, when $|\tau \pm \tau_*|\ge |k|$, we have 
\begin{equation}\label{lowbd-tau} 
\Big| \frac{1}{\tau^4 -\tau^2\tau_0^2 - |k|^2\tau_1^2} \Big |\lesssim \left \{ \begin{aligned} (1+|\tau|^4)^{-1}
\qquad& \quad \mbox{if}\quad |\tau|\ge \frac{3\tau_*}2 
\\
(|\tau \pm \tau_*| + |k|)^{-1} \qquad& \quad \mbox{if}\quad \frac{\tau_*}2\le |\tau|\le \frac{3\tau_*}2 
 \\
 ( |\tau|^2 + |k|^2 )^{-1}\qquad& \quad \mbox{if}\quad |k|\Upsilon_*< |\tau| \le \frac{\tau_*}2 
  \end{aligned}\right.
\end{equation}
Similarly, we now check that the exact same upper bounds hold for $(\tau^4 - \tau^2 \tau_0^2-  |k|^2\tau_1^2
- \mathcal{R}(i\tau,k))^{-1}$ in the region when $|\tau|> |k|\Upsilon_*$ and $|\tau \pm \tau_*| \ge  |k|$. Indeed, the bounds are clear in the case when $|k|\gtrsim 1$, since $|\tau \pm \tau_*| \ge |k|\gtrsim 1$, recalling $\tau_*(k)$ is the solution of $\tau^4 - \tau^2 \tau_0^2-  |k|^2\tau_1^2
- \mathcal{R}(i\tau,k)=0$. It remains to study the case when $|k|\ll1$. Using \eqref{good-Rb}, we have $|\mathcal{R}(i\tau,k)| \lesssim |k|^4\tau^{-2}$. Therefore, $\mathcal{R}(i\tau,k)$ is a perturbation of $|k|^2$ in the case when $|k|\ll |\tau|$, and so the same bounds as in \eqref{lowbd-tau} remain valid. Finally, we consider the case when $|k|\Upsilon_* \le |\tau| \lesssim |k|$, which is only relevant for $\Upsilon <\infty$. In this case, by definition \eqref{def-cRtau}, we compute 
$$\cR(i\tau,k) = -2\pi|k|^2\int_{\RR} \frac{u^5 \mu(\frac{1}{2}u^2)}{\tau/|k| + u} \;du =  2\pi|k|^2\int_{\{|u|<\Upsilon\}} \frac{u^6 \mu(\frac{1}{2}u^2)}{\tau^2/|k|^2 - u^2} \;du ,$$
which is well-defined and in particular strictly positive, since $|\tau|>|k|\Upsilon$ (recalling we are in the case when $\Upsilon<\infty$). Therefore, since $|\tau| \lesssim |k|$ and $|k|\ll1$, we have $|\tau|\le \tau_0/2$ and so $\tau^4 \le \frac{\tau_0^2}4 \tau^2$. This gives 
\begin{equation}\label{low-Rsmall} 
\Big |\tau^4 - \tau^2 \tau_0^2-  |k|^2\tau_1^2
- \mathcal{R}(i\tau,k) \Big | \ge \frac 34 \tau^2\tau_0^2+  |k|^2\tau_1^2 + \mathcal{R}(i\tau,k)
\gtrsim |\tau|^2 + |k|^2,\end{equation}
yielding the same bounds as in \eqref{lowbd-tau} for $(\tau^4 - \tau^2 \tau_0^2-  |k|^2\tau_1^2
- \mathcal{R}(i\tau,k))^{-1}$.

Next, recalling \eqref{good-Rb}, we have $|\tau^4 \mathcal{R}(i\tau,k) | \lesssim |\tau|^2 |k|^4$. 
Putting these into \eqref{def-TGplus} yields 
\begin{equation}\label{bd-TGk10}
|\TG_{k,1}(i\tau)|
\lesssim \left \{ \begin{aligned} |k|^4 |\tau|^2(1 + |\tau|^4)^{-2} \quad& \quad \mbox{if}\quad |\tau|\ge \frac{3\tau_*}2 
\\
 |k|^4 |\tau|^2  (|\tau \pm \tau_*|^2 + |k|^2)^{-1}  \quad& \quad \mbox{if}\quad \frac{\tau_*}2\le |\tau|\le \frac{3\tau_*}2 
 \\
 |k|^4 |\tau|^2( |\tau|^2 + |k|^2)^{-2} \quad& \quad \mbox{if}\quad |k|\Upsilon_*< |\tau| \le \frac{\tau_*}2 
  \end{aligned}\right.
\end{equation}
whenever $|\tau \pm \tau_*|\ge |k|$. Using the fact that 
$$\int_{\RR} (x^2 + |k|^2)^{-1}dx \le 4|k|^{-1},$$
we thus obtain 
$$
\Big|\int_{\{|\tau\pm \tau_*|\ge |k|, \;|\tau| > |k|\Upsilon_*\}} e^{i \tau t} \TG_k(i\tau) \; d\tau\Big|  \lesssim |k|^3 .$$

Similarly, we next bound the integral of $\partial_\lambda^n\TG_k(\lambda) $ on $\Gamma_1$. It suffices to show that  $|k|^n \partial_\lambda^n \TG_{k,1}(i\tau)$ satisfies the same bounds as those for $\TG_{k,1}(i\tau)$. Indeed, in view of \eqref{lowbd-tau}, we have
$$
\Big | \partial_\lambda \Big(\frac{1}{\lambda^4+ \lambda^2\tau_0^2 - |k|^2\tau_1^2}\Big)_{\lambda = i\tau}\Big| 
\lesssim \left \{ \begin{aligned} \langle \tau\rangle^3(1 + |\tau|^4)^{-2} 
\qquad& \quad \mbox{if}\quad |\tau|\ge \frac{3\tau_*}2 
\\
(|\tau \pm \tau_*| + |k|)^{-2} \qquad& \quad \mbox{if}\quad \frac{\tau_*}2\le |\tau|\le \frac{3\tau_*}2 
 \\
|\tau| (  |\tau|^2 + |k|^2)^{-2} \qquad& \quad \mbox{if}\quad |k|\Upsilon_* < |\tau| \le \frac{\tau_*}2 .
  \end{aligned}\right.
$$
Note that $(|\tau \pm \tau_*| + |k|)^{-1} \le |k|^{-1} $ and $|\tau| (  |\tau|^2 + |k|^2)^{-1} \le |k|^{-1}$. This proves that the $|k|\partial_\lambda$ derivative of $(\tau^4 -\tau^2 \tau_0^2 - |k|^2\tau_1^2)^{-1}$ satisfies the same bounds as those for $(\tau^4 -\tau^2 \tau_0^2- |k|^2\tau_1^2)^{-1}$, see \eqref{lowbd-tau}. The $|k|\partial_\lambda$ derivatives of $(\lambda^4 +\lambda^2\tau_0^2 -  |k|^2\tau_1^2
- \mathcal{R}(\lambda,k))^{-1}$ follow similarly, upon using \eqref{good-Rb} and the lower bound \eqref{low-Rsmall}.
%\red{no, you have a problem because the bound 
%$|(k\partial_\lambda)^n \mathcal{R}(\lambda,k))|\les |k|^4 |\lambda|^{-2} (|k|/|\lambda|)^n$ is horrible when $|\lambda|\ll |k|$. So, in the regime $|\tau|\le 1/2$, every time you are going to lose $|k|/||\tau|$.} 
This proves that $|k|^n \partial_\lambda^n \TG_{k,1}(i\tau)$ satisfy the same bounds as in \eqref{bd-TGk10} for $\TG_{k,1}(i\tau)$, yielding  
$$
\Big|\int_{\{|\tau\pm \tau_*|\ge |k|, \;|\tau| > |k|\Upsilon_*\}} e^{i \tau t} |k|^n \partial_\lambda^n  \TG_k(i\tau) \; d\tau\Big|  \lesssim |k|^3,$$
for $0\le n\le K_0$.

\subsubsection*{Bounds on $\TG_{k,1}(\lambda)$ on $\Gamma_2$.}

We next consider the integral on $\Gamma_2$, where $\lambda = i\tau$ with $|\tau | \le |k|\Upsilon_*$ and $|\tau \pm \tau_*| \ge  |k|$. Note that in this case, $\mathcal{R}(i\tau,k)$ is of order $|k|^2$ and thus no longer a remainder in the expansion of $D(\lambda,k)$. However, since $|\tau | \le |k|\Upsilon_*$, we are in the interior of the essential spectrum, where we can make use of the lower bound on $D(\lambda,k)$. Indeed, recalling from \eqref{def-TGplus} and \eqref{expansion-DL}, we have
\begin{equation}\label{recallGk11}
\TG_{k,1}( \lambda ) 
= \frac{\mathcal{R}(\lambda,k) }{(\lambda^4 +\lambda^2 \tau_0^2-  |k|^2\tau_1^2)D(\lambda,k)} .
\end{equation}
Evaluating at $\lambda = i\tau$, we have 
$$\TG_{k,1}( i\tau) 
= \frac{\mathcal{R}(i\tau,k) }{(\tau^4 -\tau^2\tau_0^2 -  |k|^2\tau_1^2)D(i\tau ,k)}.
$$
We first claim that 
 \begin{equation}\label{low-tau4}
|D(i \tau,k)| \gtrsim \Big(1 + \frac{1}{|k|^2}\Big),
\end{equation}
for all $\tau$ so that $|\tau | \le |k|\Upsilon_*$ and $|\tau \pm \tau_*| \ge  |k|$. Indeed, when $\Upsilon =\infty$, the bound follows directly from \eqref{lowbd-Ditau2}, since $\Upsilon_* = M$ and $|\tau | \le |k|M$ is a compact subset in $\RR$. Next, when $\Upsilon <\infty$, we have $\kappa_0>0$ and the bound \eqref{low-tau4} holds for $|k|\le \kappa_0/2$, using \eqref{lowbd-Ditau3}. Finally, for $\kappa_0/2\le |k|\le \kappa_0+1$, we clearly have $|D(i \tau,k)| \gtrsim 1$, since $\tau$ is away from the unique solutions $\pm\tau_*(k)$ of $D(i \tau,k) =0$. This proves \eqref{low-tau4}. {Next, we claim that
\begin{equation}\label{loweasy} |\tau^4 -\tau^2 \tau_0^2-  |k|^2\tau_1^2| \gtrsim  \tau^2 + |k|^2\end{equation}
for all $\tau$ so that $|\tau | \le |k|\Upsilon_*$ and $|\tau \pm \tau_*| \ge  |k|$. Indeed, the bound is direct for $|k|\ll1$, since $\tau^2 \lesssim |k|^2 \le \frac12 \tau_0^2$. On the other hand, since we are in the region where $\lambda = i\tau$ is away from the zeros of the polynomial $\lambda^4 +\lambda^2 \tau_0^2-  |k|^2\tau_1^2$, the lower bound thus follows for $|k|\gtrsim1$. 
}

Therefore, using \eqref{low-tau4}, \eqref{loweasy}, and $|\mathcal{R}(i\tau,k)|\lesssim |k|^2$ (recalling \eqref{good1st-bdR}), we obtain 
$$
\begin{aligned}
| \TG_{k,1}(i\tau)|
& \lesssim 
\frac{ |k|^2}{( |\tau|^2 + |k|^2)(1+|k|^{-2})} 
\lesssim |k|^2,
\end{aligned}$$
which gives 
\begin{equation}\label{intGbdk11}
\Big|\int_{\{ |\tau| \le |k|\Upsilon_*\}} e^{ i\tau t}  \TG_{k,1}(i\tau) \; d\tau \Big|\lesssim |k|^3.
\end{equation}
{Similarly, we shall next prove that the above estimates also hold for $|k|^n \partial_\lambda^n \TG_{k,1}(i\tau)$ for $1\le n\le K_0$. Indeed, we first check each term in \eqref{recallGk11}. Recalling from \eqref{good1st-bdR}, we have $|k|^n |\partial_\lambda^n\mathcal{R}(\lambda,k)| \lesssim |k|^2$. On the other hand, using \eqref{def-Dlambda1} and \eqref{unibd-cH}, for $n\ge 1$, we bound 
$$\begin{aligned}
|k|^n |\partial_\lambda^nD(\lambda,k)|
= |k|^{-2}|\partial^n_z\cH(i\lambda/|k|)| \lesssim |k|^{-2}.
\end{aligned}
$$
Next, to compute the derivatives of $\frac1{D(\lambda,k)}$, we recall the  Fa\`a di Bruno's formula for derivatives of a composite function, namely
\begin{equation}\label{Faa} \partial_x^n f(u) = \sum_{\{m_j\}} C_{m_j,n} \partial_u^m f (u) \prod_{j=1}^n (\partial_x^j u)^{m_j}  
\end{equation}
where $m_j \ge 0$, $\sum m_j = m$, and the summation is over all the partitions $\{ m_j\}_{j=1}^n$ of $n$ so that $\sum_j jm_j = n$. Using \eqref{Faa} with $f(u) = \frac{1}{u}$, we compute 
$$\begin{aligned}
|k|^n\Big |\partial^n_\lambda\Big(\frac{1}{D(\lambda,k)}\Big)\Big|
\lesssim |k|^n\sum_{\{m_j\}} |D(\lambda,k)|^{-m-1}\prod_{j=1}^n |\partial_\lambda^j D(\lambda,k)|^{m_j}  .
\end{aligned}
$$
Now, evaluating at $\lambda = i\tau$ and using the lower bound \eqref{low-tau4}, which also reads $|D(i \tau,k)| \ge |k|^{-2}$ since $k$ is bounded, we obtain  
$$\begin{aligned}
|k|^n\Big |\partial^n_\lambda\Big(\frac{1}{D(\lambda,k)}\Big)\Big|_{\lambda = i\tau}
&\lesssim |k|^n\sum_{\{m_j\}} |k|^{2(m+1)}\prod_{j=1}^n |k|^{m_j(-2-j)} 
\\
&\lesssim |k|^n\sum_{\{m_j\}} |k|^{2(m+1)} |k|^{\sum_j m_j(-2-j)}  
\\
&\lesssim |k|^n\sum_{\{m_j\}} |k|^{2(m+1)} |k|^{-2m -n}  
\\
&\lesssim |k|^2,
\end{aligned}
$$
upon recalling $\sum_j m_j = m$ and $\sum_j j m_j = n$. Finally, it is direct to check that 
\begin{equation}\label{poly1}\begin{aligned}
|k|^n\Big |\partial^n_\lambda\Big(\frac{1}{\lambda^4 +\lambda^2 \tau_0^2-  |k|^2\tau_1^2}\Big)\Big|_{\lambda = i\tau} \lesssim (\tau^2+|k|^2)^{-1}
\end{aligned}
\end{equation}
for $|\tau|\lesssim |k|$, provided \eqref{loweasy}. Putting the above estimates together into \eqref{recallGk11}, we obtain the derivative estimates $|k|^n |\partial_\lambda^n \TG_{k,1}(i\tau)| \lesssim |k|^2$, and therefore 
\begin{equation}\label{intGbdk11d}
\Big|\int_{\{ |\tau| \le |k|\Upsilon_*\}} e^{ i\tau t}  |k|^n \partial_\lambda^n \TG_{k,1}(i\tau)\; d\tau \Big|\lesssim |k|^3.
\end{equation}
}

%On the other hand, using \eqref{DR-lambda} and the above lower bounds in \eqref{lowbd-tau}, we observe that $|k|^n \partial_\tau^n\TG_k(i\tau)$ satisfy the similar bounds for $n\ge 1$, up to a constant that depends on $n$. Therefore, upon integrating by parts in $\tau$, we obtain \eqref{bd-Gr} for the above integral term.  

\subsubsection*{Bounds on $\TG_{k,1}(\lambda)$ on $\cC_\pm$.}

Finally, we study the case when $\lambda$ near the singularity of $\TG_k(\lambda)$: namely, when $\lambda$ is on the semicircle $|\lambda \mp i\tau_*(k)| = |k|$ with $\Re \lambda \ge 0$. In this case, we first claim that 
\begin{equation}\label{good-GRb1}
|\partial_\lambda^n \mathcal{R}(\lambda,k) |\lesssim |k|^4 ,
\end{equation}
uniformly in $\Re \lambda \ge 0$ for $0\le n\le K_0$. Indeed, for $|k|\gtrsim 1$, the estimate \eqref{good-GRb1} follows from \eqref{good1st-bdR}, while for $|k|\ll1$, we note that $|\lambda|\ge \tau_*(k)/2 \gtrsim 1$, and so \eqref{good-GRb1} follows from 
\eqref{good-Rb}. Next, recall from the decomposition \eqref{TG-decompose} that $\TG_{k,1}(\lambda ) = \TG_{k}(\lambda ) - \TG_{k,0}(\lambda )$, 
which reads
\begin{equation}\label{recomp-G1}
\begin{aligned}
\TG_{k,1}(\lambda ) = \frac{ \lambda^4 }{\lambda^4 +\lambda^2\tau_0^2 -  |k|^2\tau_1^2
- \mathcal{R}(\lambda,k)}  - \frac{ \lambda^4}{\lambda^4 +\lambda^2\tau_0^2 -  |k|^2\tau_1^2} .
\end{aligned}\end{equation}
Let $\lambda_\pm(k)$ and $\lambda_{\pm,0}(k)$ be the poles of the Green kernels 
$\TG_k(\lambda)$ and $\TG_{k,0}(\lambda) $, respectively, that lie on the imaginary axis, see \eqref{Landau} and \eqref{def-lambda1234}. Clearly, they are isolated zeros of $\lambda^4 +\lambda^2\tau_0^2 -  |k|^2\tau_1^2
- \mathcal{R}(\lambda,k)$ and $\lambda^4 +\lambda^2\tau_0^2 -  |k|^2\tau_1^2$, respectively. By construction, we have
\begin{equation}\label{diff-lambdaG} 
 |\lambda_\pm(k) - \lambda_{\pm,0}(k)|\lesssim |k|^4.
\end{equation}
In addition, we have 
$$
\begin{aligned}
\lambda^4 +\lambda^2\tau_0^2 -  |k|^2\tau_1^2
& = \sum_{n=1}^{K_0} a_{\pm,n,0}(k) (\lambda - \lambda_{\pm,0}(k))^n+  \mathcal{R}_{\pm,0}(\lambda,k),
\\
\lambda^4 +\lambda^2\tau_0^2 -  |k|^2\tau_1^2
- \mathcal{R}(\lambda,k)
 &=
\sum_{n=1}^{K_0} a_{\pm,n}(k) (\lambda - \lambda_\pm(k))^n+  \mathcal{R}_{\pm}(\lambda,k),
\end{aligned}$$
for any $\lambda \in \{|\lambda \mp i\tau_*| \le |k|\}$ with $\Re \lambda \ge 0$.  
In view of \eqref{good-GRb1} and \eqref{diff-lambdaG}, it is direct to deduce 
\begin{equation}\label{coeff-RaG} |a_{\pm,n,0}(k)  - a_{\pm,n}(k)| \lesssim |k|^4, \qquad  |\partial_\lambda^n \mathcal{R}_{\pm,0}(\lambda,k) - \partial_\lambda^n \mathcal{R}_{\pm}(\lambda,k)|\lesssim |k|^4,\end{equation}
for $0\le n\le K_0$. Since the leading coefficients $a_{\pm,1,0}(k)$ and $a_{\pm,1}(k)$ never vanish, we obtain 
\begin{equation}\label{exp-inverseGR}
\begin{aligned}
\frac{\lambda^4}{\lambda^4 +\lambda^2\tau_0^2 -  |k|^2\tau_1^2} 
& = \frac{1}{\lambda - \lambda_{\pm,0}(k)}
\sum_{n=0}^{K_0} b_{\pm,n,0}(k) (\lambda - \lambda_{\pm,0}(k))^n+  \widetilde{\mathcal{R}}_{\pm,0}(\lambda,k),
\\
\frac{\lambda^4}{\lambda^4 +\lambda^2\tau_0^2 -  |k|^2\tau_1^2
- \mathcal{R}(\lambda,k)} 
 &= 
 \frac{1}{\lambda - \lambda_\pm(k)}
\sum_{n=0}^{K_0} b_{\pm,n}(k) (\lambda - \lambda_\pm(k))^n+  \widetilde{\mathcal{R}}_{\pm}(\lambda,k),
\end{aligned}\end{equation}
for $\lambda \in \{|\lambda \mp i \tau_*(k)| \le |k|\}$ with $\Re \lambda \ge 0$, where the coefficients $b_{\pm,n,0}(k)$ and $b_{\pm,n}(k)$ can be computed in terms of $a_{\pm,n,0}(k)$ and $a_{\pm,n}(k)$ for $0\le n\le K_0$. Importantly, it follows from \eqref{diff-lambdaG} and \eqref{coeff-RaG} that 
\begin{equation}\label{coeff-GRa1}
 |b_{\pm,n,0}(k)  - b_{\pm,n}(k)| \lesssim |k|^4, \qquad  |\partial_\lambda^n \widetilde{\mathcal{R}}_{\pm,0}(\lambda,k) - \partial_\lambda^n \widetilde{\mathcal{R}}_{\pm}(\lambda,k)|\lesssim |k|^4,\end{equation}
for $0\le n\le K_0$. Note in particular that by construction, 
\begin{equation}\label{def-Gbn0}
b_{\pm,0,0}(k) =  \mathrm{Res} (\TG_{k,0}(\lambda_{\pm,0}))  , \qquad b_{\pm,0}(k) = \mathrm{Res} (\TG_{k}(\lambda_{\pm})) 
\end{equation}
for each $\pm$.

We are now ready to bound $\TG_{k,1}(\lambda ) $, using \eqref{recomp-G1} and the above expansions. Indeed, using \eqref{exp-inverseGR}, we first write
\begin{equation}\label{exp-Glambdak1}
\begin{aligned}
\TG_{k,1}(\lambda )  &= 
\sum_{n=0}^{K_0} \Big[ \frac{b_{\pm,n}(k)}{ (\lambda - \lambda_\pm(k))^{-n+1}} -  \frac{b_{\pm,n,0}(k)}{ (\lambda - \lambda_{\pm,0}(k))^{-n+1}}\Big] +  \widetilde{\mathcal{R}}_{\pm}(\lambda,k) -  \widetilde{\mathcal{R}}_{\pm,0}(\lambda,k),
\end{aligned}\end{equation}
where, using \eqref{coeff-GRa1}, the remainder satisfies $|  \widetilde{\mathcal{R}}_{\pm}(\lambda,k) -  \widetilde{\mathcal{R}}_{\pm,0}(\lambda,k) |\lesssim |k|^4$ for $\lambda \in \{|\lambda \mp i \tau_*(k)| \le |k|\}$ with $\Re \lambda \ge 0$. Since the terms in the summation are holomorphic in $\mathbb{C}$, we may apply the Cauchy's integral theorem to deduce 
$$
\begin{aligned}
 \frac{1}{2\pi i}\int_{\cC_\pm}e^{\lambda t} \TG_{k,1}(\lambda)  \; d\lambda &=  \sum_\pm  e^{\lambda_\pm t} b_{\pm,0}(k)  -  \sum_\pm e^{\lambda_{\pm,0} t} b_{\pm,0,0}(k) 
  \\&\quad + 
 \frac{1}{2\pi i}  \sum_{n=0}^{K_0}  \int_{\cC_\pm^*}  e^{\lambda t} \Big[ \frac{b_{\pm,n}(k)}{ (\lambda - \lambda_\pm(k))^{-n+1}} -  \frac{b_{\pm,n,0}(k)}{ (\lambda - \lambda_{\pm,0}(k))^{-n+1}}\Big]  \; d\lambda
\\&\quad  + \frac{1}{2\pi }\int_{|\tau \mp \tau_*|\le  |k|} e^{i\tau t} \Big[ \widetilde{\mathcal{R}}_{\pm}(i\tau,k) -  \widetilde{\mathcal{R}}_{\pm,0}(i\tau ,k)\Big]  \; d\tau
 \end{aligned}$$
 where $\cC_\pm^*$ denotes the semicircle $|\lambda \mp i \tau_*| = |k|$ with $\Re \lambda \le 0$, as depicted in Figure \ref{fig-contour0}. Since $| \widetilde{\mathcal{R}}_{\pm}(i\tau,k) -  \widetilde{\mathcal{R}}_{\pm,0}(i\tau ,k) |\lesssim |k|^4$, the last integral term is clearly bounded by $C_0 |k|^5$. As for the integral on $\cC_\pm^*$, recalling \eqref{diff-lambdaG}, we have 
 $|\lambda- \lambda_\pm(k)| \ge |k|/2$ and $|\lambda- \lambda_{\pm,0}(k)| \ge |k|/2$ on $\cC_\pm^*$. Therefore, together with \eqref{diff-lambdaG} and \eqref{coeff-GRa1}, we bound for $\lambda \in \cC_\pm^*$, 
$$\Big |  \frac{b_{\pm,0}(k) }{\lambda - \lambda_\pm(k)}  - \frac{b_{\pm,0,0}(k)}{\lambda - \lambda_{\pm,0}(k)} \Big | \lesssim \frac{|k|^4}{|\lambda - \lambda_\pm(k) ||\lambda - \lambda_{\pm,0}(k)|} \lesssim |k|^2,$$
and 
$$  \sum_{n=1}^{K_0} \Big | \frac{b_{\pm,n}(k)}{ (\lambda - \lambda_\pm(k))^{-n+1}} -  \frac{b_{\pm,n,0}(k)}{ (\lambda - \lambda_{\pm,0}(k))^{-n+1}}\Big | \lesssim |k|^4.
$$
This gives
 $$
\Big|  \frac{1}{2\pi i}  \int_{\cC_\pm^*}  e^{\lambda t}  \Big[  \frac{b_{\pm,0}(k) }{\lambda - \lambda_\pm(k)}  - \frac{b_{\pm,0,0}(k)}{\lambda - \lambda_{\pm,0}(k)} \Big] \; d\lambda\Big| \lesssim  \int_{\cC_\pm^*}  |k|^2 d|\lambda|\lesssim |k|^3$$ 
as desired. 

Similarly, we now bound the integral of $|k|^n \partial_\lambda^n\TG_{k,1}(\lambda ) $, using the higher-order expansions in \eqref{exp-inverseGR}. We first compute 
\begin{equation}\label{exp-DGlambdak1}
\begin{aligned}
\partial_\lambda^n\TG_{k,1}(\lambda )  &= 
\sum_{j=0}^{K_0} \frac{d^n}{d\lambda^n}\Big[ \frac{b_{\pm,j}(k)}{ (\lambda - \lambda_\pm(k))^{-j+1}} -  \frac{b_{\pm,j,0}(k)}{ (\lambda - \lambda_{\pm,0}(k))^{-j+1}}\Big] \\&\quad+  \partial_\lambda^n\widetilde{\mathcal{R}}_{\pm}(\lambda,k) -   \partial_\lambda^n\widetilde{\mathcal{R}}_{\pm,0}(\lambda,k),
\end{aligned}\end{equation}
for any $0\le n\le K_0$. Using \eqref{coeff-GRa1}, we have 
$$
\Big| \frac{1}{2\pi }\int_{|\tau \mp \tau_*|\le  |k|} e^{i\tau t} |k|^n \Big[\partial_\lambda^n\widetilde{\mathcal{R}}_{\pm}(i\tau,k) -  \partial_\lambda^n\widetilde{\mathcal{R}}_{\pm,0}(i\tau ,k)\Big]  \; d\tau\Big| \lesssim |k|^{n+5}. 
$$
We next check the terms in the summation. By Cauchy's integral theorem, we have
$$
\begin{aligned}
\frac{(-1)^n}{2\pi i t^n } \int_{\{|\lambda \mp i \tau_*| =  |k|\}} e^{\lambda t}
\frac{d^n}{d\lambda^n}\Big( \sum_{j=0}^{K_0}\frac{b_{\pm,j}(k)}{ (\lambda - \lambda_\pm(k))^{-j+1}} \Big) \; d\lambda = e^{\lambda_\pm t} b_{\pm,0}(k)  
\\
\frac{(-1)^n}{2\pi i t^n } \int_{\{|\lambda \mp i \tau_*| =  |k|\}} e^{\lambda t}
\frac{d^n}{d\lambda^n}\Big( \sum_{j=0}^{K_0}\frac{b_{\pm,j,0}(k)}{ (\lambda - \lambda_{\pm,0}(k))^{-j+1}} \Big) \; d\lambda = e^{\lambda_{\pm,0} t} b_{\pm,0,0}(k) . 
\end{aligned}$$
Therefore, it remains to prove that 
\begin{equation}\label{Dl-gamma3aG} 
\Big|\frac{1}{2\pi i}  \int_{\cC_\pm^*}  e^{\lambda t} |k|^n\frac{d^n}{d\lambda^n}\Big[ \frac{b_{\pm,j}(k)}{ (\lambda - \lambda_\pm(k))^{-j+1}} -  \frac{b_{\pm,j,0}(k)}{ (\lambda - \lambda_{\pm,0}(k))^{-j+1}}\Big] \; d\lambda \Big| \lesssim |k|^3,
\end{equation}
for any $0\le j,n \le K_0$. We focus on the most singular term: namely, the term with $j=0$; the others are similar. Since $|\lambda \mp i \tau_*| = |k|$, we have $|\lambda- \lambda_\pm(k)| \ge |k|/2$ and $|\lambda- \lambda_{\pm,0}(k)| \ge |k|/2$. Therefore, on $\cC_\pm^*$, we bound
$$
\begin{aligned}
\Big |\frac{d^n}{d\lambda^n} \Big[ \frac{b_{\pm,0}(k) }{\lambda - \lambda_\pm(k)}  - \frac{b_{\pm,0,0}(k)}{\lambda - \lambda_{\pm,0}(k)} \Big] \Big | & =  n! 
\Big | \frac{b_{\pm,0}(k) }{(\lambda - \lambda_\pm(k))^{n+1}}  - \frac{b_{\pm,0,0}(k)}{(\lambda - \lambda_{\pm,0}(k))^{n+1}} \Big | 
\\& \lesssim \frac{ |b_{\pm,0}(k) - b_{\pm,0,0}(k)|}{|\lambda - \lambda_\pm(k) |^{n+1}}  +  \frac{ |\lambda_\pm(k) - \lambda_{\pm,0}(k)| }{|\lambda - \tilde\lambda_\pm(k)|^{n+2}} 
\end{aligned}$$
for some $\tilde\lambda_\pm(k)$ in between $\lambda_\pm(k)$ and $\lambda_{\pm,0}(k)$. Using \eqref{diff-lambdaG} and \eqref{coeff-GRa1}, the above fraction is bounded by $|k|^{-n+2}$, and the estimates \eqref{Dl-gamma3aG} thus follow. Combining, we have therefore obtained 
$$
\begin{aligned}
 \frac{(-1)^n}{2\pi i (|k|t)^n}\int_{\cC_\pm}e^{\lambda t} |k|^n\partial_\lambda^n\TG_{k,1}(\lambda)  \; d\lambda 
 &=  \sum_\pm  e^{\lambda_\pm t} b_{\pm,0}(k)  -  \sum_\pm e^{\lambda_{\pm,0} t} b_{\pm,0,0}(k) \\&\quad+ \cO(|k|^3 \langle kt\rangle^{-n})
 \end{aligned}$$
as claimed for any $0\le n\le K_0$. Recalling \eqref{int-FGdecomp-n} and \eqref{def-Gbn0}, we obtain \eqref{bd-Grk1} as claimed.

\subsubsection*{Residue of $\TG_{k}(\lambda)$.}

We now compute the residue of the Green's  function at the poles $\lambda_\pm (k)$, which gives the oscillatory component of the Green function as stated in \eqref{decomp-FG}. 
%Indeed, from \eqref{dispersion1}, taking $\tau_0 =1$ for sake of presentation, we note that \red{This is not enough. Needs to be done properly by 
%looking at the expression for $D(\lambda,k)$.}
%Recalling 
%
%
%Hence, for $\lambda_\pm(k)$ defined as in Theorem \ref{theo-Langmuir}, we  write 
Using Theorem \ref{theo-LangmuirE}, we write $$
\TG_k(\lambda )= 1+ \frac{1-D(\lambda,k)}{D(\lambda,k)}  = 1+ \sum_\pm \frac{ a_\pm (\lambda,k)}{\lambda - \lambda_\pm(k)}  
 $$ 
for some function $a_\pm(\lambda,k)$ that is holomorphic in $\lambda \in \{\Re \lambda>0\}$ and uniformly bounded on the imaginary axis. In addition, since $\partial_\lambda^nD(\lambda,k)$ are finite in $\{\Re \lambda \ge 0\}$ for $0\le n\le K_0$, the functions $a_\pm(\lambda,k)$ are also $C^{K_0}$ differentiable in $\lambda$ up to the imaginary axis. This yields 
$$ \mathrm{Res} (\TG_k(\lambda_\pm(k))) = a_\pm(\lambda_\pm(k),k)  = \frac{1}{\partial_\lambda D(\lambda_\pm(k),k)}$$
at each pole $\lambda_\pm(k)$. Note in particular that 
$$ 
\partial_\lambda D(\lambda_\pm(0),0) = \pm {2 \over i \tau_0}  + {\mathcal{O}}(|k|^2) ,
$$
which gives the oscillatory term as stated in \eqref{decomp-FG}. 

%In addition, $|a_\pm(\lambda,k)|\lesssim 1$ and we have 
%$$ 
%a_\pm(\lambda_\pm(k),k) = \pm {i \over 2} + \mathcal{O}(|k|^2) 
%$$
%for $|k|\ll1$. As a consequence, we have 
%$
%\mathrm{Res} (\TG_k(\lambda_\pm(k))) =  \pm {i \over 2}  + \mathcal{O}(|k|^2) 
%$

%This ends the proof of the bounds \eqref{bd-Gr}. 

\subsubsection*{Regularity of $\TG_{k}(\lambda)$ in $k$.}

Finally, we study the regularity of $\TG_{k}(\lambda)$ in $k$. In view of \eqref{fakeres-G0}, it follows that
\begin{equation}\label{est-dkFGtr0k}
\begin{aligned} 
\Big |\partial_k^\alpha [ |k|^{-2} e^{\mu_{-,0}t} \mathrm{Res} (\TG_{k,0}(\mu_{-,0})) ] \Big| \lesssim \langle t\rangle^{|\alpha|-1} e^{-\theta_0 |kt|}, \qquad \forall |\alpha|\ge 1,
\end{aligned}\end{equation}
for any $\theta_0<\tau_1$. On the other hand, recalling from \eqref{def-cRtau1}, we compute 
$$
\partial_k^\alpha\cR(\lambda,k) = -2i\pi \partial_k^\alpha \Big[|k|^3 \int_0^\infty e^{-\lambda t} \int_{\RR} e^{-i|k| t u}u^5 \mu(\frac{1}{2}u^2) \;dudt\Big]
$$
and so 
\begin{equation}\label{dk-Rlambda1}
||k|^\alpha\partial_k^\alpha \mathcal{R}(\lambda,k)|  \lesssim \int_0^\infty |k|^{\alpha}  \Big[|k|^{3-|\alpha|} +  \langle t\rangle^{|\alpha|}\Big] \langle kt\rangle^{-K_0} dt \lesssim |k|^2,
\end{equation}
for $0\le |\alpha|\le K_0 -2$, recalling that $|k|\le \kappa_0+1$. That is, $|k|^\alpha\partial_k^\alpha\mathcal{R}(\lambda,k)$ derivatives satisfy the same bounds as those for $\mathcal{R}(\lambda,k)$. Hence, following the similar lines as done above, we obtain the bounds on $|k|^\alpha\partial_k^\alpha \FG^r_k(t)$ as claimed.  This ends the proof of Proposition \ref{prop-GreenG}. 
\end{proof}

\subsection{Green function in the physical space}

In this section, we bound the Green function $G(t,x)$ in the physical space. To this end, we will use the homogeneous Littlewood-Paley decomposition of $\mathbb{R}^{3}$. That is, for any function $h$, we decompose 
\begin{equation}\label{def-LP}
h(x)= \sum_{q \in \mathbb{Z}} P_qh(x), 
\end{equation}
where $P_q$ denotes the Littlewood-Paley projection on the dyadic interval $[2^{q-1}, 2^{q+1}]$, whose Fourier transform in $x$ is given by $\widehat{P_qh}(k)= \Fh(k) \varphi(k/2^q)$, for a fixed smooth cutoff function $\varphi\in [0,1]$ that is compactly supported in the annulus $ { 1 \over 4} \leq |k | \leq 4$ and equal to one in the inner annulus  $ { 1 \over 2} \leq |k | \leq 2$. {In the paper, we also use the following classical Bernstein inequalities (see, e.g.,
\cite{BCD})
\begin{equation}\label{Berstein}
\| P_k \partial_x h \|_{L^p} \lesssim 2^{k} \| h\|_{L^p}, \qquad  2^{k} \| P_kh\|_{L^p} \lesssim \| \partial_x h\|_{L^p_x}
\end{equation}
for all $p \in [1,\infty]$ and $k \in \mathbb{Z}$. 
}

We shall prove the following proposition. 

\begin{proposition}\label{prop-Greenphysical} Let $G(t,x)$ be the temporal Green's function {whose Fourier transform $\FG_k(t)$ is constructed} in Proposition \ref{prop-GreenG}. Then, {there is some universal constant $C_0$ so that }
\begin{equation}\label{est-HoscLp}
\begin{aligned}
\|G^{osc}_\pm\star_x f\|_{L^p_x} &\le C_0 \langle t\rangle^{-3(\frac 12-\frac 1p)}\| f\|_{L^{p'}_x},
\end{aligned}
\end{equation}
for any $p\in [2,\infty]$ with $\frac 1p+\frac 1{p'}=1$. In addition, letting $\chi(k)$ be a smooth cutoff function whose support is contained in $\{|k|\le 1\}$, for any $0\le n\le K_0/2$ and $p\in [1,\infty]$, there hold 
\begin{equation}\label{est-HrLp1} 
\begin{aligned}
\|\chi(i\partial_x) \partial_x^n \Delta_x^{-1}G^{r}(t) \|_{L^p_x} &\le C_0 \langle t\rangle^{-4+3/p - n} ,
\end{aligned}\end{equation}
and 
\begin{equation}\label{est-HrLp2} \| (1-\chi(i\partial_x)) \partial_x^n G^r (t) \|_{L^p_x}  \le C_0 \langle t\rangle^{-K_0/2}.\end{equation}
%Moreover, for any $\delta>0$, 
%\begin{equation}\label{conv-HrLp}
%\| \chi(i\partial_x) \Delta_x^{-1}G^{r} (t)\star_x f\|_{L^p_x} \lesssim \langle t\rangle^{-1-\delta } \sup_{q\in \ZZ} 2^{-\delta q}\| P_qf\|_{L^p_x}
%\end{equation}
%for any $p\in [1,\infty]$. %, where $P_q$ is the Littlewood-Paley projection on the dyadic interval $[2^{q-1}, 2^{q+1}]$. 

\end{proposition}

\begin{proof} We start with the oscillatory Green function, which is defined by 
\begin{equation} \label{Hsreal}
\begin{aligned}
G^{osc}_\pm(t,x) &= \int e^{\lambda_\pm(k) t + i k\cdot x} a_\pm(k) dk 
.\end{aligned}\end{equation}
Note that $G^{osc}_\pm$ is a smoothed version of the Green's function for the Klein-Gordon operator for bounded frequencies, since $\tau_*(k) = \pm\Im \lambda_\pm(k)$ behaves like $\sqrt{1 + |k|^2}$, while $\Re \lambda_\pm(k)\le 0$, see Theorem \ref{theo-LangmuirE}. Therefore, the operator bounds \eqref{est-HoscLp} follow from the unitary in $L^2$ and the dispersion in $L^\infty$ of the Klein-Gordon's solution operators $e^{\pm i \tau_*(i\partial_x) t}$, noting there is no loss of derivatives in \eqref{Hsreal}, since the symbols $a_\pm(k) $ are compactly supported in the spatial frequency.

Next, we study the Green function $G^r(t,x)$. Recall that the Fourier transform $\FG^r_k(t)$ satisfies
\begin{equation}\label{bound-Grk3}
| |k|^{|\alpha|}\partial_k^\alpha \FG^{r}_k(t)| \le C_1 |k|^3 \langle k\rangle^{-4} \langle kt\rangle^{-K_0+|\alpha|} .
\end{equation}
Hence, {for $|k|\le 1$,} we bound 
$$
\begin{aligned}
| \chi(i\partial_x) \partial_x^n \Delta_x^{-1}G^r(t,x)|
& \lesssim \int_{\{|k|\le 1\}} |k|^{n} |\FG_k^r(t)| \; dk
\\& \lesssim \int_{\{|k|\le 1\}} |k|^{n+1}\langle |k|t \rangle^{-K_0} \; dk 
\lesssim \langle t\rangle^{-n-4}
\end{aligned}$$
proving \eqref{est-HrLp1} for $p=\infty$. {On the other hand, using \eqref{bound-Grk3} for $|k|\ge 1$ and $\alpha=0$, namely 
$|\FG^{r}_k(t)| \le C_1 \langle k\rangle^{-1} \langle kt\rangle^{-K_0}$,} we bound 
$$
\begin{aligned}
| (1-\chi(i\partial_x)) \partial_x^n G^r(t,x)|
& \lesssim \int_{\{|k|\ge 1\}} |k|^{n} |\FG_k^r(t)| \; dk
\\& \lesssim \int_{\{|k|\ge 1\}} \langle k\rangle^{n-1}\langle |k|t \rangle^{-K_0} \; dk 
\lesssim \langle t\rangle^{- K_0/2} ,
\end{aligned}$$
proving \eqref{est-HrLp2} for $p=\infty$. It remains to give bounds in $L^1_x$. Indeed, we first note that for the high frequency part, we bound 
$$
\begin{aligned}
\| (1-\chi(i\partial_x)) \partial_x^n G^r(t)\|_{L^1_x}^2
& \lesssim \sum_{|\alpha|\le 2}\int_{\{|k|\ge 1\}} |k|^{2n} |\partial_k^\alpha \FG_k^r(t)|^2 \; dk
\\&  \lesssim  \sum_{|\alpha|\le 2} \int_{\{|k|\ge 1\}} |k|^{2n-2|\alpha|-2} \langle |k|t \rangle^{-2K_0} \; dk 
\lesssim \langle t\rangle^{- K_0/2} ,
\end{aligned}$$
for $n\le K_0/2$, proving \eqref{est-HrLp2}. Let us now treat the low frequency part. Using the Littlewood-Paley decomposition \eqref{def-LP}, we write 
$$
\begin{aligned}
\chi(i\partial_x) G^r(t,x) &= \sum_{q\le 0} P_q[G^r(t,x) ]
 \end{aligned}$$
where 
$$
\begin{aligned}
P_q[G^r(t,x)] 
&=  \int_{\{2^{q-2}\le |k|\le 2^{q+2}\}} e^{ik \cdot x} \FG^r(t,k) \varphi(k/2^q)\; dk
\\
&= 2^{3q} \int_{\{\frac14\le |\tilde k|\le 4\}} e^{i \tilde k \cdot 2^q x}\FG^r(t,2^q \tilde k) \varphi(\tilde k)\; d\tilde k
\end{aligned}$$
Using the estimates \eqref{bound-Grk3} for $|k|\le1$ (and so $q\le 0$), we obtain 
$$| \partial_{\tilde k}^\alpha [|k|^{-2}\FG^r(t,2^q \tilde k)]| \le 2^{q|\alpha|}|\partial_{k}^\alpha [|k|^{-2}\FG^r(t,2^q \tilde k)]|  \lesssim 2^q \langle 2^{q} \rangle^{-K_0 +|\alpha|} 
$$
for $\frac14 \le |\tilde k|\le 4$. Therefore, integrating by parts repeatedly in $\tilde k$ and taking $\alpha =4$ in the above estimate, we have 
$$
\begin{aligned}
|P_q[G^r(t,x)]| 
&
\lesssim  2^{3q} \Big| 
 \int e^{i \tilde k \cdot 2^q x}\FG^r(t,2^q \tilde k) \varphi(\tilde k)\; d\tilde k
 \Big|
\\&
\lesssim  2^{3q} 
 \int \langle 2^q x\rangle^{-4} |\partial_{\tilde k}^4 [\FG^r(t,2^q \tilde k) \varphi(\tilde k)]| \; d\tilde k
 \Big|
\\&
\lesssim  2^{4q} \langle 2^q x\rangle^{-4} \langle 2^{q} t \rangle^{-K_0 +4}.
\end{aligned}$$
Taking $L^1_x$, we obtain 
$$
\begin{aligned}
 \| \chi(i\partial_x) \partial_x^n G^r(t)\|_{L^1_x} 
 &\lesssim \sum_{q\le 0} 2^{nq}\|P_q[G^r(t)]\|_{L^1_x} 
\lesssim \sum_{q\le 0} 2^{q(n+1)} \langle 2^{q} t \rangle^{-K_0 + 4}\lesssim \langle t\rangle^{-n-1},
  \end{aligned}$$ 
which gives \eqref{est-HrLp1} for $p=1$. The $L^p$ estimates follow from a standard interpolation between $L^1$ and $L^\infty$ estimates. This completes the proof of the proposition. 
\end{proof}

%%%%%%%%%%%

\subsection{Field representation}\label{sec-fields}

%%%%%%%%%%%

In this section, we give a complete representation of the electric field $E$ in term of the initial data. First, taking the inverse of the Laplace transform both sides of the resolvent equation \eqref{resolvent}, we get 
\begin{equation}
\label{rep-phi} 
\begin{aligned}
\Fphi_k(t) &= \frac{1}{|k|^2}\int_0^t \FG_k(t-s) \FS_k(s)\; ds
\end{aligned}
\end{equation}
where $\FG_k(t)$ be the Green function defined as in \eqref{def-FHk}, and 
\begin{equation}\label{free-density}
\begin{aligned}
\FS_k(t) &=   \frac{1}{2\pi i}\int_{\{\Re \lambda = \gamma_0\}}e^{\lambda t} \Big(\int \frac{\Ff_{0,k}(v)}{\lambda +  ik \cdot  v} \; dv\Big)\; d\lambda =\int e^{-ikt\cdot v}\Ff_{0,k}(v) \;dv. 
\end{aligned}
\end{equation}
Observe that $\FS_k(t)$ is the exact density generated by the free transport dynamics. In the physical space, using the representation \eqref{decomp-FG} on the Green function, we thus obtain 
\begin{equation}\label{rep-electric}
\begin{aligned}
\phi & = (-\Delta_x )^{-1}\Big[ S + \sum_\pm G^{osc}_\pm \star_{t,x} S + G^{r} \star_{t,x} S\Big] .
\end{aligned}
\end{equation}
From the dispersion of the free transport, the electric field component $\nabla_x \Delta_x^{-1}S(t,x)$ decays only at rate $t^{-2}$, 
which is far from being sufficient to get decay for the oscillatory electric field $G^{osc}_\pm \star_{t,x} \nabla_x \Delta_x^{-1}S$ 
through the spacetime convolution. Formally, one would need $\nabla_x \Delta_x^{-1}S$ to decay at order $t^{-4}$ in order to establish a $t^{-3/2}$ 
decay for $G^{osc}_\pm \star_{t,x} \nabla_x \Delta_x^{-1}S$. 
However it turns out that, by integrating by parts in time, the first term $S$ may be absorbed in the
second one, leading to a good  time decay, as we will now detail. We obtain the following.

%
%$$
%-\partial_t A  = \partial^2_t  H(t) \star_x A^{\mathrm{free}} + \partial_t H(t) \star_x A^1 + b_\pm(i\partial_x) \mP S^j(t) + H\star_{t,x}  \mP S^j(t)
%$$

\begin{proposition}[Electric potential decomposition]\label{prop-decompE} 
Let $\phi$ be the electric potential described as in \eqref{rep-electric}. Then, there holds 
\begin{equation}\label{rep-electricE}
\begin{aligned}
\phi &= \sum_\pm \phi^{osc}_\pm(t,x) + \phi^r(t,x) 
\end{aligned}\end{equation}
with 
 \begin{equation}\label{rep-phiosc}
\begin{aligned}
-\Delta_x\phi^{osc}_\pm(t,x) &= G_{\pm}^{osc}(t) \star_{x} \Big[  \frac{1 }{\lambda_\pm(i\partial_x)} S(0) +  \frac{1}{\lambda_\pm(i\partial_x)^2}\partial_tS(0)\Big] 
\\&\quad+ G^{osc}_{\pm} \star_{t,x}  \frac{1}{\lambda_\pm(i\partial_x)^2}\partial^2_t S 
\\
-\Delta_x\phi^r(t,x) &= G^r \star_{t,x} S+  \cP_2(i\partial_x)S  + \cP_4(i \partial_x) \partial_t S,
\end{aligned}\end{equation}
where $G_\pm^{osc}(t,x), G^r(t,x)$ are defined as in Proposition \ref{prop-GreenG}, with 
%$$
%\FG_{\pm,1}^{osc}(t,k) = \frac{1}{\lambda_\pm(k)}\FG_{k,\pm}^{osc}(t), \qquad \FG_{\pm,2}^{osc}(t,k) = \frac{1}{\lambda_\pm(k)^2}\FG_{k,\pm}^{osc}(t),
%$$
%having 
$\lambda_\pm (k)$ as in  Theorems \ref{theo-LangmuirE} and \ref{theo-Landau}. In addition, $\cP_2(i\partial_x)$,  $\cP_4(i\partial_x)$  
denote smooth Fourier multipliers, which are sufficiently smooth and satisfy 
\begin{equation}\label{bounds-P24k}
\begin{aligned}
|  \cP_2(k)| + |  \cP_4(k)| & \lesssim |k|^2 \langle k\rangle^{-2}, \qquad \forall~k\in \RR^3,
\end{aligned}\end{equation}
and $\cP_4(k)$ is compactly supported in $\{|k|\le\kappa_0+ 1\}$.

\end{proposition}

\begin{proof} 
In the high frequency regime, the proposition follows from \eqref{rep-electric}. We thus focus on the low frequency regime: 
namely, the region $\{ |k|\lesssim 1\}$, where $k\in \RR^3$ is the Fourier frequency. In this regime, we use the representation \eqref{rep-electric}.
 Now, making use of time oscillations of $G^{osc}_\pm(t,x)$, we integrate by parts in time the second term of \eqref{rep-electric}, 
which gives
$$
G^{osc}_\pm \star_{t,x} S=   \frac{1}{\lambda_\pm(i\partial_x)} \Big[ G^{osc}_{\pm}(t,\cdot)\star_x S(0,x) 
-  G^{osc}_{\pm}(0,\cdot)\star_x S(t,x)  + G^{osc}_{\pm} \star_{t,x} \partial_tS(t,x) \Big].
$$
{We recall from Proposition \ref{prop-GreenG} that the Fourier transform of $G^{osc}_\pm(t,x)$ is of the form $\FG^{osc}_{k,\pm}(t) = e^{\lambda_\pm(k)t} a_\pm(k)$.} Therefore, recalling that the support of $a_\pm(k)$ is contained in $\{|k|\le \kappa_0+ 1\}$, we write 
$$
\sum_\pm  \frac{1}{\lambda_\pm(i\partial_x)} G^{osc}_{\pm}(0,x) 
= \int e^{i k\cdot x}  \sum_\pm \frac{a_\pm(k)}{\lambda_\pm(k)} dk = \int_{\{|k|\le \kappa_0+1 \}} e^{i k\cdot x} \Big(\widetilde \cP_3(k)   - {\widetilde \cP_2(k) }\Big)dk,
%\\
%G^{osc}_{\pm,2}(0,\cdot) 
%&= \int e^{i k\cdot x}  \frac{a_\pm(k)}{\lambda^2_\pm(k)} dk = \int e^{i k\cdot x} \Big(\mp \frac i2 + \cO(|k|^2)\Big)dk.
%\end{aligned}
$$
{in which we have set $\widetilde \cP_3(k) = \chi_{\{ |k|\le \kappa_0+2\}}$ and $\widetilde \cP_2(k) = (1 - \sum_\pm \frac{a_\pm(k)}{\lambda_\pm(k)})\chi_{\{ |k|\le \kappa_0+2\}}$ for some smooth cut-off function $\chi_{\{ |k|\le \kappa_0+2\}}$ that is equal to one on $\{|k|\le \kappa_0+1\}$. Recalling that $\lambda_\pm(0) = \pm i\tau_0$ and $a_\pm(0) = \pm \frac{i\tau_0}{2}$, we therefore have  $\widetilde \cP_2(k)=0$ for $|k|\ge \kappa_0+2$ and $|\widetilde \cP_2(k)| \lesssim |k|^2$, uniformly in $k\in \RR^3$. 
As a result, viewing $\widetilde \cP_j(i\partial_x)$ as the Fourier multipliers, we may write 
}
\begin{equation}\label{cal-GoscE}
 \sum_\pm  \frac{1}{\lambda_\pm(i\partial_x)}G^{osc}_{\pm}(0,\cdot)\star_x S =  \widetilde \cP_3(i \partial_x) S - \widetilde \cP_2(i\partial_x) S  ,
\end{equation}
Note that the low frequency part of $S$ in (\ref{rep-electric}) cancels with $\widetilde \cP_3(i \partial_x) S$ in (\ref{cal-GoscE}). Precisely, we have 
$$ 
\begin{aligned}
S + G^{osc}_\pm \star_{t,x} S &=  \frac{1}{\lambda_\pm(i\partial_x)} G^{osc}_{\pm}(t,\cdot)\star_x S(0,x) +  \frac{1}{\lambda_\pm(i\partial_x)}G^{osc}_{\pm} \star_{t,x} \partial_tS(t,x) 
\\&\quad + \widetilde \cP_2(i\partial_x) S + (1- \widetilde \cP_3(i \partial_x)) S,
\end{aligned}$$
{in which $(1- \widetilde \cP_3(k))$ is supported away from $\{|k|\ge \kappa_0+1\}$.}
Finally, as it turns out that $\partial_tS(t,x) $ decays not sufficiently fast for the convolution $G^{osc}_{\pm} \star_{t,x} \partial_tS(t,x) $, we further integrate by parts in $t$, yielding 
$$
G^{osc}_{\pm} \star_{t,x} \partial_tS =   \frac{1}{\lambda_\pm(i\partial_x)} \Big[ G^{osc}_{\pm}(t,\cdot)\star_x \partial_tS(0,x) 
- G^{osc}_{\pm}(0,\cdot)\star_x \partial_t S(t,x) +  G^{osc}_{\pm} \star_{t,x} \partial^2_tS(t,x) \Big],
$$
%where
%\begin{equation} \label{Goscpm2}
%G^{osc}_{\pm,2}(t,x) = \int e^{\lambda_\pm(k) t + i k . x} {a_\pm(k) \over \lambda_\pm(k)^2} dk .
%\end{equation}
Using \eqref{Landau}, we have
$$
\sum_\pm  \frac{1}{\lambda_\pm(i\partial_x)^2}G^{osc}_{\pm}(0,x) 
= \sum_\pm \int e^{i k\cdot x}  \frac{a_\pm(k)}{\lambda_\pm(k)^2} dk = \int_{\{|k|\le \kappa_0+1\}} e^{i k\cdot x} \cP_4(k)dk
%\\
%G^{osc}_{\pm,2}(0,\cdot) 
%&= \int e^{i k\cdot x}  \frac{a_\pm(k)}{\lambda^2_\pm(k)} dk = \int e^{i k\cdot x} \Big(\mp \frac i2 + \cO(|k|^2)\Big)dk.
%\end{aligned}
$$
{where $\cP_4(k) = \sum_\pm \frac{a_\pm(k)}{\lambda_\pm(k)^2} \chi_{\{|k|\le \kappa_0+2\}}$. Note that $\cP_4(k)$ has the same support as that of $a_\pm(k)$ which is contained in $\{|k|\le \kappa_0+1\}$. In addition, recalling $\lambda_\pm(0) = \pm i\tau_0$ and $a_\pm(0) = \pm \frac{i\tau_0}{2}$, we have $|\cP_4(k)|\le C_0 |k|^2$, namely there is a cancellation that takes place at the leading order of $k$ for $k$ small. }This proves that 
$$ 
\begin{aligned}
S + \sum_\pm G^{osc}_\pm \star_{t,x} S &= \sum_\pm  \frac{1}{\lambda_\pm(i\partial_x)} G^{osc}_{\pm}(t,\cdot)\star_x S(0,x)  \\&\quad- \sum_\pm  \frac{1}{\lambda_\pm(i\partial_x)^2}G^{osc}_{\pm}(t,\cdot)\star_x \partial_tS(0,x) 
\\& \quad + \widetilde \cP_2(i\partial_x) S + (1- \widetilde \cP_3(i \partial_x)) S + \cP_4(i \partial_x) \partial_t S.
\\& \quad + \sum_\pm  \frac{1}{\lambda_\pm(i\partial_x)^2}G^{osc}_{\pm} \star_{t,x} \partial^2_tS(t,x) .
\end{aligned}$$
{Finally, setting $\cP_2(k) =\widetilde \cP_2(k) + 1- \widetilde \cP_3(k) $, we obtain the proposition. }
\end{proof}

\section{Decay estimates}\label{sec-decay}

We are now ready to prove Theorem \ref{theo-main}, giving the decay estimates on the electric field.    

\subsection{Free transport dispersion}
First, recalling from \eqref{free-density}, the charge density $S$ by the free transport in the physical space reads
$$
S(t,x) =\int f_0(x - v t, v) \;dv.
$$
Hence, introducing the change of variables $y = x - v t$, we bound 
$$
\begin{aligned}
| S(t,x)| {\le t^{-3} \Big|\int f_0(y, \frac{x-y}{t}) \;dy\Big|}
\le
t^{-3} \| \sup_v f_0(\cdot,v)\|_{L^1_x}.
\end{aligned}$$
Similarly, we have 
\begin{equation}\label{bound-Sfree}
\| \partial_t^n\partial_x^\alpha S(t)\|_{L^\infty} \le C_0 t^{-3-n-|\alpha|} \sum_{|\beta|\le n + |\alpha|}\| \sup_v { \langle v\rangle^n} |\partial_v^\beta f_0(\cdot,v)|\|_{L^1_x} ,
\end{equation}
for $n,|\alpha|\ge 0$ and {for some universal constant $C_0$.} It also follows directly that $\| \partial_t^n\partial_x^\alpha S(t)\|_{L^1_x} \lesssim t^{-n-|\alpha|}$, which also gives decay estimates in $L^p_x$ for $p\in [1,\infty]$. These dispersive estimates play a key role in studying the large time behavior of solutions to the Vlasov-Poisson system near vacuum (e.g., \cite{Bardos}) or in the screened case (\cite{HKNF2}). 

\subsection{Bounds on $E$}\label{sec-bdE}

We now bound each term in the representation for $\phi$, see \eqref{rep-electricE}. Note that $S(0) = \rho[f_0]$. Therefore, using the assumption that $ \int S(0)\; dx = \iint f_0 \; dxdv =0$, the estimate \eqref{aLp-elliptic}, and the bound \eqref{est-HoscLp}, we 
have 
$$\begin{aligned}
 \| G_{\pm}^{osc}(t) \star_{x} \nabla_x \Delta_x^{-1}S(0)\|_{L^p} & \lesssim  t^{- 3(1/2-1/p)} \| \nabla_x \Delta_x^{-1}\rho[f_0]\|_{L^{p'}} 
 \\& \lesssim  t^{- 3(1/2-1/p)} \| \langle x\rangle \rho[f_0]\|_{L^1_x \cap L^\infty_x}
 \end{aligned}$$
which is bounded by $C_0 t^{- 3(1/2-1/p)}$ for $p\in [2,\infty)$. Observe that the above estimate in general fails for $p=\infty$, unless some additional assumption on the vanishing of higher moments of $\rho[f_0]$. Similarly, using $\partial_t S= - \nabla_x \cdot S^j$, where $S^j = \int vf_0(x - v t, v) \;dv$, and the boundedness of the operator $\nabla^2_x \Delta^{-1}_x$ in $L^p$ for $p\in (1,\infty)$, we bound  
$$
\begin{aligned}
 \| G_{\pm}^{osc}(t) \star_{x} \nabla_x \Delta_x^{-1} \partial_t S(0)\|_{L^p} & \lesssim  t^{- 3(1/2-1/p)} \| \nabla_x \Delta_x^{-1} \nabla_x \cdot \bj[f_0]\|_{L^{p'}} 
 \\&\lesssim  t^{- 3(1/2-1/p)} \| \bj[f_0]\|_{L^{p'}_x}
 \end{aligned}$$
for $p\in [2,\infty)$, with $\bj[f_0] = \int v f_0\;dv$. Finally, we treat the spacetime convolution term $G^{osc}_{\pm} \star_{t,x} \nabla_x \Delta_x^{-1} \partial^2_t S$. Observe that we can further integrate by part in time, yielding  
$$
\begin{aligned}
G^{osc}_{\pm} \star_{t,x} \nabla_x \Delta_x^{-1} \partial^2_t S &=   \frac{1}{\lambda_\pm(i\partial_x)} \Big[ G^{osc}_{\pm}(t,\cdot)\star_x \nabla_x \Delta_x^{-1} \partial^2_t S(0,x) 
\\&\quad- G^{osc}_{\pm}(0,\cdot)\star_x \nabla_x \Delta_x^{-1} \partial^2_t S(t,x) 
\\&\quad +  G^{osc}_{\pm} \star_{t,x} \nabla_x \Delta_x^{-1} \partial^3_t S(t,x) \Big].
\end{aligned}$$
Note that $\partial_t^2S = \sum_{ij}\partial_{x_ix_j}^2S^{ij}$ and $\partial_t^3S = \sum_{ijk}\partial_{x_ix_jx_k}^3S^{ijk}$, where $S^{ij} = \int v_i v_jf_0(x-v tv)\; dv$ and $S^{ijk} = \int v_i v_jv_kf_0(x-v tv)\; dv$. In particular, taking a smooth cutoff Fourier symbol $ \chi(k)$ with bounded support
and using Lemma \ref{sec-potential}, we bound 
$$
\begin{aligned}
\| \chi(i\partial_x)\nabla_x\Delta_x^{-1}\partial_t^2S(t)\|_{L^p_x} &\lesssim \| |\partial_x|^{1/2} S^{ij}(t)\|_{L^p_x} \lesssim \langle t\rangle^{-3(1-1/p)-1/2},
\\
\| \chi(i\partial_x)\nabla_x\Delta_x^{-1}\partial_t^3S(t)\|_{L^p_x} &\lesssim \| |\partial_x|^{3/2} S^{ijk}(t)\|_{L^p_x} \lesssim \langle t\rangle^{-3(1-1/p)-3/2},
\end{aligned}
$$ 
for $1\le p\le \infty$. Therefore, the first two terms in the above expression for $G^{osc}_{\pm} \star_{t,x} \nabla_x \Delta_x^{-1} \partial^2_t S$ are the boundary terms and can be treated as before. On the other hand, noting $G^{osc}_{\pm}$ only consists of small spatial frequency, we bound the last term by 
$$
\begin{aligned}
 \| G^{osc}_{\pm} \star_{t,x} \nabla_x \Delta_x^{-1} \partial^3_t S\|_{L^\infty_x} 
 &\lesssim 
 \int_0^{t/2} \| G^{osc}_\pm(t-s)\|_{L^\infty_x} \| \chi(i\partial_x)\partial_x \Delta_x^{-1}\partial_t^3 S(s)\|_{L^1_x} \; ds
\\&\quad + \int_{t/2}^t \| \chi(i\partial_x)\nabla_x \Delta_x^{-1} \partial^3_t S(s)\|_{L^2_x} \; ds
\\ &\lesssim 
 \int_0^{t/2} (t-s)^{-3/2} \langle s\rangle^{-3/2}\; ds
+ \int_{t/2}^t \langle s\rangle^{-3} \; ds
\\ &\lesssim \langle t\rangle^{-3/2}. 
\end{aligned}$$
Similarly, $ \| G^{osc}_{\pm} \star_{t,x}\nabla_x \Delta_x^{-1} \partial^3_t S\|_{L^2_x} \lesssim \int_0^t \| \chi(i\partial_x)\nabla_x \Delta_x^{-1} \partial^3_t S\|_{L^2_x}\; ds \lesssim 1$.  Finally, we note that the symbol $1/\lambda_\pm(k)$ is regular and compactly supported in $\{|k|\le 1\}$, and therefore $1/\lambda_\pm(i\partial_x)$ is a bounded operator from $L^p$ to $L^p$ for $1\le p\le \infty$. 
Recalling \eqref{rep-phiosc} and combining the above estimates, we have 
$$\| \nabla_x\phi^{osc}_\pm(t)\|_{L^p_x} \lesssim \langle t\rangle^{-3(1/2 - 1/p)} $$
for $2\le p<\infty$. 

Next, we bound $\nabla_x \phi^r$, which we recall from \eqref{rep-phiosc} that 
$$
\begin{aligned}
\nabla_x \phi^r(t,x) &= -\nabla_x \Delta_x^{-1} \Big[G^r \star_{t,x} S+  \cP_2(i\partial_x)S  + \cP_4(i \partial_x) \partial_t S\Big].
\end{aligned}
$$
Recall that $\cP_2$ and $\cP_4$ are defined as in Proposition \ref{prop-decompE} with $\cP_2(k) = \cO(|k|^2\langle k\rangle^{-2})$ and $\cP_4(k) = \cO(|k|^2\langle k\rangle^{-2})$ (the latter of which is compactly supported in $|k|\le 1$). 
Therefore, both $\cP_2(i\partial_x) \Delta_x^{-1}\nabla_x$ and $\cP_4(i\partial_x) \Delta_x^{-1}\nabla_x$ are bounded in $L^p$ for $1\le p\le \infty$, see Lemma \ref{lem-potential}. This proves  
$$ \| \cP_2(i\partial_x) \Delta_x^{-1}\nabla_x S(t)\|_{L^p_x} \lesssim \| S(t)\|_{L^p_x}  \lesssim  \langle t\rangle^{-3(1-1/p)} $$
for $p\in [1,\infty]$. 
The estimates for $ \cP_4(i \partial_x) \Delta_x^{-1} \nabla_x \cdot \partial_t S$ also follow identically. On the other hand, using Proposition \ref{prop-Greenphysical}, we have 
\begin{equation}\label{est-Grphysical} 
\|\partial_x^\alpha \Delta_x^{-1} G^r(t)\|_{L^p_x} \lesssim \langle t\rangle^{-4+3/p-|\alpha|} ,
\end{equation}
for any $p\in [1,\infty]$, noting that the high frequency part satisfies better decay estimates. 
Therefore, we bound 
$$  
\begin{aligned}
\| \Delta_x^{-1}G^r \star_{t,x} \partial_x S\|_{L^p_x} 
&\lesssim
 \int_0^{t/2}  \|\partial_x\Delta_x^{-1}G^r(t-s)\|_{L^p_x}\|S(s)\|_{L^1_x} \; ds
\\&\quad +  \int_{t/2}^t  \| \Delta_x^{-1}G^r(t-s)\|_{L^1_x}\|\partial_x S(s)\|_{L^p_x} \; ds
\\
&\lesssim 
 \int_0^{t/2} \langle t-s\rangle^{-5+3/p} \; ds 
+ \int_{t/2}^t  \langle t-s\rangle^{-1}\langle s\rangle^{-3(1-1/p)-1}\; ds
\\
&\lesssim
\langle t\rangle^{-3+3/p} .
\end{aligned} $$
Therefore, we obtain $\| \nabla_x \phi^r(t)\|_{L^p_x} \lesssim \langle t\rangle^{-3+3/p}$ for $1\le p\le \infty$. The higher derivative estimates follow similarly.

\appendix 

\section{Potential estimates}\label{sec-potential}

In this section, we recall some classical estimates for the Poisson equation in $\RR^3$. 

\begin{lemma}\label{lem-potential} Let $\chi(k)$ be sufficiently smooth, compactly supported in $\{|k|\le 1\}$, { and $\chi(k)=1$ for $|k|\le 1/2$}. Then, the followings hold: 

\begin{itemize}

\item[(i)] $\nabla^2_x \Delta^{-1}_x$ is a bounded operator from $L^p$ to $L^p$ for each $1<p<\infty$. In addition, for any $K>0$,
\begin{equation}\label{standard-pot1}
\begin{aligned}
 \|\nabla^2_x \Delta^{-1}_x\rho\|_{L^\infty_x} & \lesssim K^{-3}\|\rho\|_{L^1_x} + \|\rho\|_{L^\infty_x} \Big[\log (2+ \|\partial_x\rho\|_{L^\infty_x}) + \log (2+K) \Big] .
 \end{aligned}
 \end{equation}

\item[(ii)] $(1-\chi(i\partial_x)) \nabla_x \Delta^{-1}_x$ is a bounded operator from $L^p$ to $L^p$ for all $1\le p\le \infty$.  

\item[(iii)] for any $\delta>0$, $\chi(i\partial_x) |\partial_x|^\delta \nabla^2_x \Delta^{-1}_x$ is a bounded operator from $L^p$ to $L^p$ for all $1\le p \le\infty$.   

\item[(iv)] If $\int \rho(x)\; dx=0$, then 
\begin{equation}\label{aLp-elliptic} \| \chi(i\partial_x) \nabla_x \Delta^{-1}_x \rho\|_{L^p_x} \lesssim  \| \langle x\rangle\rho\|_{L^1_x \cap L^\infty_x}, \qquad \forall 1<p\le \infty.\end{equation}

%If in addition $\int x\rho(x)\; dx=0$, then 
%\begin{equation}\label{aLp-elliptic1} \| \chi(i\partial_x)\nabla_x \Delta^{-1}_x \rho\|_{L^1_x} \lesssim \| \langle x\rangle^2 \rho\|_{L^1_x \cap L^\infty_x}.
%\end{equation}
%

%\begin{equation}\label{standard-pot1}
%\begin{aligned}
% \|(1-\chi(i\partial_x)) \nabla^2_x \Delta^{-1}_x\rho\|_{L^\infty_x} & \lesssim 
% \|\rho\|_{L^\infty_x} \log (2+ \|\partial_x\rho\|_{L^\infty_x}) .
% \end{aligned}\end{equation}

\end{itemize}

%
%Let $E = \nabla_x \Delta^{-1}_x \rho.$ Then, there hold 
%$$ \| E\|_{L^\infty_x} \lesssim \| \rho\|_{L^1_x}^{1/3} \| \rho\|_{L^\infty_x}^{2/3}$$
%and for any $K>0$,
%$$
%\begin{aligned}
% \| \partial_x E\|_{L^\infty_x} & \lesssim K^{-3}\|\rho\|_{L^1_x} + \|\rho\|_{L^\infty_x} \Big[ 1 + \log (2+ \|\partial_x\rho\|_{L^\infty_x}) + \log (2+K) \Big] .
% \end{aligned}$$

\end{lemma}

\begin{proof} The lemma is classical. Indeed, the first statement follows from the fact that $\nabla^2_x \Delta^{-1}_x$ is a Calderon-Zygmund operator, while the estimate \eqref{standard-pot1} is direct and can be found, e.g., in \cite[Chapter 4]{Glassey-book}. As for the second statement, we use the {standard} Littlewood-Paley decomposition and the {classical} Bernstein inequalities, {see \eqref{Berstein},} to bound 
$$
\begin{aligned}
\|(1-\chi(i\partial_x)) \nabla_x \Delta^{-1}_x \rho \|_{L^p_x} \le \sum_{q\ge 0} 2^{-q} \| P_q \rho \|_{L^p_x} \lesssim  \| \rho \|_{L^p_x} \sum_{q\ge 0} 2^{-q} \lesssim  \| \rho \|_{L^p_x},
\end{aligned}
$$
for all $1\le p\le \infty$. Note that the summation is over $q\ge 0$, since $1-\chi(k)$ has its Fourier support contained in $\{|k|\ge 1/2\}$. Similarly for the third statement, we bound 
$$
\begin{aligned}
\|\chi(i\partial_x) |\partial_x|^\delta\nabla^2_x \Delta^{-1}_x \rho \|_{L^p_x} \le \sum_{q\le 0} 2^{\delta q} \| P_q \rho \|_{L^p_x} \lesssim  \| \rho \|_{L^p_x} \sum_{q\le 0} 2^{\delta q} \lesssim  \| \rho \|_{L^p_x},
\end{aligned}
$$
for all $1\le p\le \infty$, noting the series converges, since $\delta>0$. 

%we note that the integral kernel $K(x,y)$ of the operator $(1-\chi(i\partial_x)) \nabla_x \Delta^{-1}_x$ satisfies 
%$$ |K(x,y)|\lesssim \frac{1}{|x-y|^2} e^{-|x-y|} ,$$
%which in particular implies that $\| K(\cdot,y)\|_{L^1_x}$ and $\| K(x,\cdot)\|_{L^1_y}$ are uniformly bounded. The boundedness of $(1-\chi(i\partial_x)) \nabla_x \Delta^{-1}_x$ from $L^p$ to $L^p$ thus follows from the standard Hausdorff-Young's inequality. 

Finally, we check the statement $(iv)$. The zero average assumption implies that $\rho(x)$ has vanishing Fourier coefficient $\Frho(0) =0$, and so on the support of $\chi(k)$, we have $\Frho(k) \sim  k$. We thus bound  
$$
\begin{aligned}
 \chi(i\partial_x) \nabla_x \Delta^{-1}_x \rho (x) 
 &= \int_{\{|k|\le 1\}} e^{ik\cdot x} ik |k|^{-2} \Frho(k) \; dk 
 \\ 
 & = \int_0^1 \int_{\{|k|\le 1\}} e^{ik\cdot x} ik |k|^{-2} k \cdot \nabla_k \Frho(\theta k) \; dk d\theta .
 \end{aligned}$$ 
This proves $\|  \chi(i\partial_x) \nabla_x \Delta^{-1}_x \rho\|_{L^\infty_x} \lesssim \| \nabla_k \Frho\|_{L^2_k} = \| x\rho\|_{L^2_x}$. Next, we bound its $L^p_x$ norm, which is sufficient to study for $|x|\ge 1$, since the sup norm controls the $L^p$ norm for bounded $x$. 
Recalling that the integral kernel of $\Delta^{-1}$ is $|x-y|^{-1}$, we may write 
$$ \nabla_x \Delta^{-1}_x\rho(x) = \int \frac{x-y}{|x-y|^3} \rho(y) \; dy = \int \Big[\frac{x-y}{|x-y|^3}  - \frac{x}{|x|^3} \Big]\rho(y)  \; dy$$
in which we have used the assumption that $\rho$ has zero average in $x$. Noting 
$$\Big|\frac{x-y}{|x-y|^3}  - \frac{x}{|x|^3} \Big| \le \frac{4|y|}{|x| |x-y|^2} +  \frac{4|y|}{|x|^2 |x-y|} ,$$
we bound for $|x|\ge 1$, 
$$
\begin{aligned} 
|\nabla_x \Delta^{-1}_x\rho(x)| 
& \le \frac{4}{|x|} \int |x-y|^{-2}|y\rho(y)|  \; dy +  \frac{4}{|x|^2} \int |x-y|^{-1}|y\rho(y)|  \; dy
\\
& \lesssim  \frac{1}{|x|^3} \int_{\{|x-y|\ge |x|/2\}}|y\rho(y)|  \; dy + \frac{1}{|x|} \int_{\{|x-y|\le |x|/2\}} |x-y|^{-2}|y\rho(y)|  \; dy 
\\&\quad +  \frac{1}{|x|^2} \int_{\{|x-y|\le |x|/2\}} |x-y|^{-1}|y\rho(y)|  \; dy.
\end{aligned}$$
The first term is bounded by $C_0 \langle x\rangle^{-3}\| y \rho\|_{L^1_y} $, which is finite in $L^p_x$ for any $p>1$. The second and third integrals are bounded in $L^p_x$ by $C_0 \| y\rho\|_{L^p_y}$ for any $p\ge 1$ (including the endpoint $p=1$), since the integral kernels 
$$ \frac{1}{\langle x\rangle} |x-y|^{-2} \chi_{\{|x-y|\le |x|/2\}} , \qquad  \frac{1}{\langle x\rangle^2} |x-y|^{-1} \chi_{\{|x-y|\le |x|/2\}} $$
are bounded in $L^1_y$ . This proves \eqref{aLp-elliptic} and thus completes the proof of the lemma. 
%It remains to check the $L^1_x$ bounds under the additional assumption $\int x\rho(x)\; dx=0$. For this, we need to improve the spatial decay in the above calculation. Indeed, using $\int x \rho(x)\; dx=0$, we further write 
%$$ \nabla_x \Delta^{-1}_x\rho(x) = \int \Big[\frac{x-y}{|x-y|^3}  - \frac{x}{|x|^3}  - \frac{1}{|x|^3}\Big (\mathbb{I} + 3\frac{x\otimes x}{|x|^2}\Big) y\Big]\rho(y)  \; dy$$
%and use 
%$$\Big|\frac{x-y}{|x-y|^3}  - \frac{x}{|x|^3} - \frac{1}{|x|^3}\Big (\mathbb{I} + 3\frac{x\otimes x}{|x|^2}  \Big) y\Big| \le \frac{16|y|^2}{|x|^2 |x-y|^2} +  \frac{16|y|^2}{|x|^3 |x-y|}.$$
%Note that in the region when $|x-y|\ge |x|/2$ and $|x|\ge 1$, the above is bounded by $C_0 |y|^2 \langle x\rangle^{-4}$, which gives a sufficient decay to be bounded in $L^1_x$. The integral over the region $|x-y|\le |x|/2$ is already bounded in $L^1_x$ as done in the previous case. 
\end{proof}

%

%

%%%%%%%%%%%%%%%%%

\end{document}